\documentclass{amsart}

\usepackage{amsmath,amsthm,amsfonts,amssymb,latexsym,mathrsfs,graphicx}

\usepackage{hyperref}

\usepackage{enumerate}
\usepackage[shortlabels]{enumitem}
\usepackage{color}

\newcommand{\Irr}{{\mathrm {Irr}}}
\newcommand{\Irrek}{{\mathrm {Irr}_{ \Bbb{F}, even}}}

\newcommand{\acd}{{\mathrm {acd}}}
\newcommand{\acdek}{{\mathrm{acd}_{\Bbb{F}, even}}}
\newcommand{\acdk}{{\mathrm{acd}^*_{\Bbb{F}}}}

\newcommand{\PSL}{{\mathrm {PSL}}}

\newcommand{\SL}{{\mathrm {SL}}}

\newtheorem{theorem}{Theorem}[section]

\newtheorem{proposition}[theorem]{Proposition}
\theoremstyle{corollary}
\newtheorem{corollary}[theorem]{Corollary}
\newtheorem{lemma}[theorem]{Lemma}

\newtheorem{Remark}[theorem]{Remark}

\newtheorem*{ThmA}{Theorem A}
\newtheorem*{ThmB}{Theorem B}
\newtheorem*{ThmD}{Theorem D}

\newtheorem*{CorC}{Corollary C}

\begin{document}

\title[Average character degree]{Variations on average character degrees and solvability}


%    author one information

\author[N. Ahanjideh]{Neda Ahanjideh}
\address{Department of Pure Mathematics, Faculty of Mathematical Sciences, Shahrekord University,Shahrekord, Iran}

\email{ahanjidn@gmail.com; ahanjideh.neda@sci.sku.ac.ir}
\author[Z. Akhlaghi]{Zeinab  Akhlaghi}
\address{ Department of Math. and Computer Sci.,Amirkabir University of Technology (Tehran Polytechnic), 15914 Tehran, Iran.}
\address{
	School of Mathematics,
	Institute for Research in Fundamental Science(IPM)
	P.O. Box:19395-5746, Tehran, Iran.}
\email{z\_akhlaghi@aut.ac.ir;
z\_akhlaghi@imp.ir}
\thanks{The second author  is supported by a grant from IPM (No. 1401200113). }

\author[K. Aziziheris]{Kamal Aziziheris}
\address{
	Department of Pure Mathematics, Faculty of Mathematical Sciences,
	University of Tabriz,
	Tabriz,
	Iran.}
\email{azizi@tabrizu.ac.ir}
\thanks{The research of the third author was funded by the grant No. 99001382
	from the Iranian National Science Foundation (INSF). Also, The third author would like to thank the University of Tabriz
	for supporting of this work.}

%    \subjclass is required.

\date{}

\dedicatory{}

%    Abstract is required.
\begin{abstract}
	Let $G$ be a finite group, $\Bbb{F}$ be one of the fields $\mathbb{Q},\mathbb{R}$ or $\mathbb{C}$, and $N$ be a non-trivial normal subgroup of $G$.
	Let $\acdk(G)$ and $\acdek(G|N)$  be the average degree of all non-linear $\Bbb F$-valued irreducible characters of $G$ and of even degree $\Bbb F$-valued irreducible characters of $G$ whose kernels do not contain $N$, respectively.  We assume the average of an empty set is $0$ for more convenience.
	In this paper we prove that if ${\rm acd}^*_{\mathbb{Q}}(G)< 9/2$ or  $0<{\rm acd}_{\mathbb{Q},even}(G|N)<4$, then $G$ is solvable.
	Moreover, setting $\Bbb{F} \in \{\Bbb{R},\Bbb{C}\}$,  we obtain the solvability of $G$ by assuming   $\acdk(G)<29/8$ or $0<\acdek(G|N)<7/2$, and we conclude the solvability of $N$ when  $0<\acdek(G|N)<18/5$.  Replacing $N$ by $G$ in $\acdek(G|N)$  gives us an extended  form of a result by   Moreto and Nguyen. Examples are given to show that  all the bounds are sharp.
	%In particular, if $0<{\rm acd}_{\mathbb{Q},even}(G)<4$ or $\Bbb{F} \in \{\Bbb{R},\Bbb{C}\}$ and $\acdek(G)<18/5$, then $G$ is solvable, where $\acdek(G)$ is the average degree of even degree irreducible characters of $G$ whose values belong to $\Bbb{F}$..    Finally, we prove that if $0<{\rm acd}_{\mathbb{Q},even}(G|N)<4$ or $\mathbb{F} \in \{\mathbb{R},\mathbb{C}\}$ and $0<\acdek(G|N)<\frac{7}{2}$,
	%then $G$ is solvable. Note that all the bounds are sharp.
\end{abstract}

\maketitle

\section{Introduction}
The interaction between the structure of finite groups and character degrees has
been a main topic of interest for a long time. In this way of research, an invariant
concerning character degrees is often considered, and studies how it influences the
structure of the group. One of those invariants is the so-called average character degree,
which has attracted considerable interest recently. For a finite group $G$, let $\Irr(G)$
denote the set of all irreducible (complex) characters of $G$ and
$$\acd(G):=\frac{\sum_{\chi\in \Irr(G)}\chi(1)}{|\Irr(G)|}$$
denote the average character degree of $G$. It has been proved by I. M. Isaacs, M. Loukaki and A. Moreto
in \cite{isa} that if $\acd(G)<3$, then $G$ is solvable. In \cite{moreto}, the authors observed
that  there is no non-solvable group whose the average
character degree is smaller than $\acd(\rm{Alt}_5)= 16/5$. These results indicate that the structure of a finite group is controlled
by its average character degree.

It is natural that we just consider the degrees of certain
irreducible characters (like the rational/real irreducible characters, the irreducible characters of either
$p'$-degree or degree divisible by $p$, where $p$ is a fixed prime, or irreducible characters whose kernels do not
contain a fixed normal subgroup). Then, we expect also to obtain
structural information of the group.

Let $N$ be a non-trivial normal subgroup of the finite group $G$ and let $\Irr(G|N)$ denote the set of all complex irreducible characters of
$G$ whose kernels do not contain $N$. Hence, $\Irr(G|G')$ consists of all complex non-linear irreducible characters of
$G$. Assume  $\Bbb{F}$ is one of the fields $\mathbb{C},\mathbb{R}$ or $ \mathbb{Q}$. Then we write
$$\Irr^*_{\Bbb{F}}(G)=\{\chi \in \Irr(G)\ |\ \chi \ \ \text{is  non-linear and $\Bbb{F}$-valued} \},$$
$$\Irrek(G|N)=\{\chi\in \Irr(G|N)| \chi \text{ is an $\Bbb{F}$-valued even degree character}  \}.$$

Let $\acd(G|N)$, $\acdk(G)$, and $\acdek(G|N)$ be the average degree of
irreducible characters belong to $\Irr(G|N)$, $\Irr^*_{\Bbb{F}}(G)$, and $\Irrek(G|N)$, respectively.
Obviously, $\acd(G|G')={\rm acd}^{*}_{\Bbb{C}}(G)$. In \cite{mar}, it is proved that if $\acd(G|N)< 16/5$,
for some non-trivial normal subgroup $N$, then $G$ is solvable.   In this paper, we replace $N$ by the derived subgroup $G'$
and we obtain the following upper bounds  for $\acd^*_{\Bbb F}(G)$ that guarantees the solvability of $G$:
%%%%%%%%%%%%%%%%%%%%%%%%%%%%%%%%%%%%%%%%%5
\begin{ThmA}\label{A}
	Let $G$ be a finite non-abelian group. Then:
	\begin{enumerate}
		\item If ${\rm acd}^{*}_{\Bbb{C}}(G)<{\rm acd}^{*}_{\Bbb{C}}({\rm SL}_{2}(5))=\frac{29}{8}$, then $G$ is solvable.
		\item If ${\rm acd}^{*}_{\Bbb{R}}(G)<{\rm acd}^{*}_{\Bbb{R}}({\rm SL}_{2}(5))=\frac{29}{8}$, then $G$ is solvable.
		\item If ${\rm acd}^{*}_{\Bbb{Q}}(G)<{\rm acd}^{*}_{\Bbb{Q}}({\rm Alt}_{5})=\frac{9}{2}$, then $G$ is solvable.
	\end{enumerate}
\end{ThmA}

It follows from Theorem~B of \cite{moreto} that if ${\rm acd}_{\Bbb{C},even}(G)<{\rm acd}_{\Bbb{C},even}({\rm SL}_{2}(5))=\frac{18}{5}$, then
the group $G$ is a solvable group. We extend this result to  $\acdek(G|N)$, where $\Bbb{F}$ is one of the fields $\mathbb{C},\mathbb{R}$ or $ \mathbb{Q}$ and $N$ is a non-trivial normal subgroup of $G$.
Adopting  the convention that the average of the empty set is $0$,  we get the results as follows:
%%%%%%%%%%%%%%%%%%%%%%%%%%%%%%%%%%%%%5
\begin{ThmB}\label{B}
	Let $G$ be a finite and $N$ be a non-trivial normal subgroup of $G$. Then:
	\begin{enumerate}
		\item If $0<{\rm acd}_{\Bbb{C},even}(G|N)<{\rm acd}_{\Bbb{C},even}({\rm SL}_{2}(5)|{\rm SL}_{2}(5))=\frac{18}{5}$, then $N$ is solvable.
		\item If $0<{\rm acd}_{\Bbb{R},even}(G|N)<{\rm acd}_{\Bbb{R},even}({\rm SL}_{2}(5)|{\rm SL}_{2}(5))=\frac{18}{5}$, then $N$ is solvable.
		\item If $0<{\rm acd}_{\Bbb{Q},even}(G|N)<{\rm acd}_{\Bbb{Q},even}({\rm Alt}_5|{\rm Alt}_5)=4$, then $N$ is solvable.
	\end{enumerate}
\end{ThmB}

By {\rm \cite[Corollary 1.4]{hung}}, if $G$ is  non-solvable, then $G$ admits a real irreducible character of even degree. This shows that if $\Bbb{F} \in \{\Bbb{R},\Bbb{C}\}$ and ${\rm acd}_{\Bbb{F},even}(G)=0$, then $G$ is solvable. Hence, if we replace $N$ by $G$ in the  previous theorem, we get the following corollary in which the first part is exactly Theorem~B of \cite{moreto}.
%%%%%%%%%%%%%%%%%%%%%%%%%%%%%%%%%%%%%%%%%%%%%%55
\begin{CorC}\label{C}
	Let $G$ be a finite. Then:
	\begin{enumerate}
		\item If ${\rm acd}_{\Bbb{C},even}(G)<{\rm acd}_{\Bbb{C},even}({\rm SL}_{2}(5))=\frac{18}{5}$, then $G$ is solvable.
		\item If ${\rm acd}_{\Bbb{R},even}(G)<{\rm acd}_{\Bbb{R},even}({\rm SL}_{2}(5))=\frac{18}{5}$, then $G$ is solvable.
		\item If $0<{\rm acd}_{\Bbb{Q},even}(G)<{\rm acd}_{\Bbb{Q},even}({\rm Alt}_5)=4$, then $G$ is solvable.
	\end{enumerate}
\end{CorC}

Note that Tiep and  Tong-Viet in {\rm \cite[Theorem B]{pham}} proved that  if $G$ is  non-solvable such that no simple group ${\rm PSL}_2(3^{2f+1})$ (with $f \geq 1$) is involved in $G$, then $G$ has a  rational irreducible character of even degree. So, it follows from Corollary C that if ${\rm acd}_{\Bbb{Q},even}(G)=0$, then either $G$ is solvable or some simple group ${\rm PSL}_2(3^{2f+1})$ (with $f \geq 1$) is involved in $G$.

Finally,  we prove that:
%%%%%%%%%%%%%%%%%%%%%%%%%%%%%%%%%%%%%%%%%%%%
\begin{ThmD}\label{D}
	Let $G$ be a finite  group and $N$ be a non-trivial normal subgroup of $G$. Then:
	\begin{enumerate}
		\item If $0<{\rm acd}_{\Bbb{C},even}(G|N)<{\rm acd}_{\Bbb{C},even}({\rm SL}_{2}(5)|{\bf Z}(\SL_2(5)))=\frac{7}{2}$, then $G$ is solvable.
		\item If $0<{\rm acd}_{\Bbb{R},even}(G|N)<{\rm acd}_{\Bbb{R},even}({\rm SL}_{2}(5)|{\bf Z}(\SL_2(5)))=\frac{7}{2}$, then $G$ is solvable.
		\item If $0<{\rm acd}_{\Bbb{Q},even}(G|N)<{\rm acd}_{\Bbb{Q},even}({\rm Alt}_5| { \rm Alt}_5)=4$, then $G$ is solvable.
	\end{enumerate}
\end{ThmD}
Throughout the paper, we denote by ${\rm Irr}_{\Bbb{F}}(G)$, ${\rm Irr}_{\Bbb{F}}^{*}(G)$, and ${\rm Irr}_{\Bbb{F},even}(G)$
the sets of $\Bbb{F}$-valued irreducible characters of $G$, $\Bbb{F}$-valued non-linear irreducible characters of $G$, and
$\Bbb{F}$-valued even degree irreducible characters of $G$, respectively. For a positive integer $d$, we denote by $n^{\Bbb{F}}_d(G)$ and $n^{\Bbb{F}}_d(G|N)$, the number of $\Bbb{F}$-valued  irreducible characters of $G$ of
degree $d$ and the number of $\Bbb{F}$-valued  irreducible characters of $\Irr(G|N)$ of  degree $d$, respectively.
For a character $\lambda$ of some subgroup of the group $G$ and for a character $\psi$ of $G$, let
${\rm Irr}_{\Bbb{F}}(\lambda^{G})$, ${\rm Irr}_{\Bbb{F}}^{*}(\lambda^{G})$, ${\rm Irr}_{\Bbb{F},even}(\lambda^{G})$, ${\rm Irr}_{\Bbb{F}}(\psi)$, ${\rm Irr}_{\Bbb{F}}^{*}(\psi)$, and
${\rm Irr}_{\Bbb{F},even}(\psi)$ be the sets of $\Bbb{F}$-valued irreducible characters of $G$ above $\lambda$, $\Bbb{F}$-valued non-linear irreducible characters of $G$ above $\lambda$,
$\Bbb{F}$-valued even degree irreducible characters of $G$ above $\lambda$, $\Bbb{F}$-valued irreducible constituents of $\psi$, $\Bbb{F}$-valued non-linear irreducible constituents of $\psi$,
and $\Bbb{F}$-valued even degree irreducible constituents of $\psi$, respectively. For a set $\Delta$ of irreducible characters of a finite group, we write $\acd_{\Bbb F}(\Delta)$,
${\rm acd}_{\Bbb{F}}^{*}(\Delta)$ and ${\rm acd}_{\Bbb{F},even}(\Delta)$ for the average degree of $\Bbb{F}$-valued elements of $\Delta$, average degree of $\Bbb{F}$-valued non-linear elements of $\Delta$, and
average degree of $\Bbb{F}$-valued even degree elements of $\Delta$, respectively. For the rest of the notation, we follow \cite{isaacs}.

\maketitle
\section{Preliminaries}
First, we present here some useful results and facts which will be used in the
proofs of the main results. As we mentioned before,  thorough this section we assume   $\Bbb{F}$ is one of the fields $\mathbb{C},\mathbb{R}$ or $ \mathbb{Q}$.
%%%%%%%%%%%%%%%%%%%%%%%%%%%%%%%%%%%%%%%%%%%%%%5555
\begin{lemma}\label{Lemma 2.2} Let $N \triangleleft G$ and  $\lambda ,\psi \in \operatorname{Irr}(N)$ be $G$-invariant. Suppose that $\psi(1) > 1$ and
	$\psi$ is extendible to an element $\chi \in {\rm Irr}_{\mathbb{F}}(G)$.
	\begin{enumerate}
		\item[{\rm (i)}] If $\lambda\psi \in {\rm Irr}(N)$ and
		${\rm Irr}_{\mathbb{F},even}(\lambda^{G}) \neq \emptyset$,
		then
		$$
		\operatorname{acd}_{\mathbb{F},even}(\lambda^{G}+(\lambda\psi)^{G})\geq \alpha_{\psi}(\psi(1)+1)/2,
		$$where $\alpha_{\psi}={\rm gcd}(2,\psi(1)+1)$. In particular, $$
		\operatorname{acd}_{\mathbb{F},even}(\lambda^{G}+(\lambda\psi)^{G})>(\psi(1)+1)/2.
		$$
		\item[{\rm (ii)}]  If $\lambda\psi \in {\rm Irr}(N)$ and
		${\rm Irr}^*_{\mathbb{F}}(\lambda^{G}) \neq \emptyset$, then $$
		\operatorname{acd}_{\mathbb{F}}^*(\lambda^{G}+(\lambda\psi)^{G})\geq\lambda(1)(\psi(1)+1)/2
		.
		$$
	\end{enumerate}
\end{lemma}
\begin{proof} Let  ${\rm Irr}_{\mathbb{F}}(\lambda^G)=\{\chi_{1}, \ldots,
	\chi_{h}\}$. Without loss of generality, suppose that ${\rm
		Irr}_{\mathbb{F},even}(\lambda^{G})=\{\chi_{1}, \ldots, \chi_l\}$.  Since
	$\psi(1)>1$, we have $\lambda{\psi}$ and $\lambda$ are
	not $G$-conjugate and so, ${\rm Irr}((\lambda\psi)^{G}) \cap {\rm Irr}(\lambda^{G})=\emptyset$. Obviously, one of the following cases occurs:
	%%%%%%%%%%%%%%%%%%%%%%%%%%%%%%%%%%%%5
	\\{\bf Case 1.} Suppose that $2 \mid \chi(1)$. Then, \cite[Theorem
	6.16]{isaacs} shows that $ {\rm Irr}_{\mathbb{F},even}((\lambda
	\psi)^{G})=\{\chi \chi_{i}: 1 \leq i \leq h\} . $ Hence $| {\rm
		Irr}_{\mathbb{F},even}((\lambda\psi)^{G})|=| {\rm Irr}_{\mathbb{F}}(\lambda^{G}) |=h.$
	Therefore, since we have  ${\rm Irr}((\lambda\psi)^{G}) \cap {\rm Irr}((\lambda)^{G})=\emptyset$,
	$$
	{\rm acd}_{\mathbb{F},even}(\lambda^{G}+(\lambda\psi)^{G})
	=\frac{1}{h+l}(\sum_{i=1}^{l} \chi_{i}(1)+\chi(1)
	\sum_{i=1}^{h}\chi_{i}(1))
	> (\psi(1)+1)/{2},
	$$ where the last inequality holds because $\sum_{i=1}^l \chi_i(1)+\chi(1)\sum_{i=1}^h\chi_i(1) \geq [2l+2l\psi(1)+\psi(1)(h-l)]>[\psi(1)(h+l)+(h+l)]/2$.
	\\{\bf Case 2.} Suppose that $2 \nmid \chi(1)$. Then, \cite[Theorem
	6.16]{isaacs} forces ${\rm Irr}_{\mathbb{F},even}((\lambda\psi)^{G})=\{\chi\chi_{i}:  1 \leq i \leq l\}
	$. Thus,
	$$
	{\rm
		acd}(\lambda^{G}+(\lambda\psi)^{G})=\frac{1}{2l}(\sum_{i=1}^{l}
	\chi_i(1)+\chi(1) \sum_{i=1}^{l} \chi_i(1))=(1+\chi(1))\sum_{i=1}^{l} \chi_i(1)/(2l) \geq(\psi(1)+1).$$ Now,
	the proof  of (i) is complete. Next, we are going to prove (ii). Let ${\rm Irr}^*_{\mathbb{F}}(\lambda^G)=\{\chi_{i_1},\ldots,\chi_{i_k}\}$. Recall that if $\lambda(1)>1$, then ${\rm Irr}^*_{\mathbb{F}}(\lambda^G)={\rm Irr}_{\mathbb{F}}(\lambda^G)$.
	Since $\chi(1)>1$, \cite[Theorem
	6.16]{isaacs} shows that $ {\rm Irr}^*_{\mathbb{F}}((\lambda
	\psi)^{G})=\{\chi \chi_{i}: 1 \leq i \leq h\} . $ Hence $| {\rm
		Irr}^*_{\mathbb{F}}((\lambda\psi)^{G})|=| {\rm Irr}_{\mathbb{F}}(\lambda^{G}) |=h.$
	Therefore, since  ${\rm Irr}((\lambda\psi)^{G}) \cap {\rm Irr}(\lambda^{G})=\emptyset$,
	$$
	{\rm acd}^*_{\mathbb{F}}(\lambda^{G}+(\lambda\psi)^{G})
	=\frac{1}{k+h}(\sum_{j=1}^{k} \chi_{i_j}(1)+\chi(1)
	\sum_{i=1}^{h}\chi_{i}(1))
	\geq \lambda(1)(\psi(1)+1)/{2},
	$$
	as needed in  (ii).
\end{proof}
%%%%%%%%%%%%%%%%%%%%%%%%%%%%%%%%%%%%%%%%%%%%%%
\begin{lemma}\label{2.2} Let $N \triangleleft G$ and  $\lambda ,\psi \in \operatorname{Irr}(N)$ such that $I_G(\lambda)=I_G(\lambda\psi)\vartriangleleft G$ and $\lambda\psi \in \operatorname{Irr}(N)$. Suppose that $\psi(1) > 1$ and
	$\psi$ is extendible to an element $\chi \in {\rm Irr}_{\mathbb{F}}(G)$. If
	${\rm Irr}_{\mathbb{F}}(\lambda^{G}) \neq \emptyset$, then ${\rm acd}_{\Bbb{F}}((\lambda\psi)^{G}+\lambda^{G})\geq [G:I_G(\lambda)] \lambda(1)(\psi(1)+1)/2$.
\end{lemma}
\begin{proof}
	%%%%%%%%%%%%%%%%%%%%%%%%%%%%%%%%%%
	Since
	$\psi(1)>1$, we have $\lambda{\psi}$ and $\lambda$ are
	not $G$-conjugate and so, ${\rm Irr}((\lambda\psi)^{G}) \cap {\rm Irr}(\lambda^{G})=\emptyset$. Set $T=I_G(\lambda)$. Let ${\rm Irr}(\lambda^T)=\{\mu_1,\ldots,\mu_n\}$. Then, ${\rm Irr}(\lambda^G)=\{\mu_i^G: 1 \leq i\leq n\}$, using Clifford correspondence. Without loss of generality, we can assume that ${\rm Irr}_{\mathbb{F}}(\lambda^G)=\{\mu_1^G,\ldots,\mu_m^G\}$
	for some positive integer $m \leq n$.   On the other hand, we can see easily  that  $\psi $ is extendible to $\rho \in {\rm Irr}_{\mathbb{F}}(T)$ such that $I_{G}(\rho)=G$. In light of \cite[Theorem 6.16]{isaacs} and Clifford correspondence, ${\rm Irr}((\lambda\psi)^{T})=\{\mu_i\rho:1 \leq i \leq n\}$ and ${\rm Irr}((\lambda\psi)^{G})=\{(\mu_i\rho)^{G}:1 \leq i \leq n\}$.  On the other hand,  $I_{G}(\rho)=G$ and $T\vartriangleleft G$, which forces $((\mu_i\rho)^{G})_{T}=((\mu_i^{G})_{T})\rho$ for every $1 \leq i \leq n$ and  $\mu_i^{G}(x)=(\mu_i\rho)^{G}(x)=0$ for every $x \in G-T$. Therefore, $(\mu_i\rho)^{G} \in {\rm Irr}_{\mathbb{F}}((\lambda\psi)^{G})$ if and only if $\mu_i^{G} \in {\rm Irr}_{\mathbb{F}}(\lambda^{G})$. Consequently,   $${\rm Irr}_{\mathbb{F}}((\lambda\psi)^{G})=\{(\mu_i\rho)^{G}:1 \leq i \leq m\}.$$ This yields that  ${\rm acd}_{\mathbb{F}}((\lambda\psi)^{G}+\lambda^{G})
	=\Sigma_{i=1}^m(\mu_i^{G}(1)\rho(1)+\mu_i^{G}(1))
	/2m \geq [G:I_G(\lambda)] \lambda(1)(\psi(1)+1)/2$, as desired.
\end{proof}
%%%%%%%%%%%%%%%%%%%%%%%%%%%%%%%%%%%%%%%%%%%%%%%%%%%%%%%555
\begin{lemma} \label{145}
	Let $N$ be a normal subgroup of $G$, $M/N$ be a non-abelian chief
	factor of $G$ and $\lambda \in {\rm Irr}(N)$.
	If $\lambda $ is extendible to $\lambda_0 \in {\rm Irr}(M)$, then $I_G(\lambda)=I_G(\lambda_0)$.
\end{lemma}
\begin{proof} As $(\lambda_0)_N=\lambda$, we have $I_G(\lambda_0) \leq I_G(\lambda)$. By Gallagher's theorem  \cite[Corollary 6.17]{isaacs}, ${\rm Irr}(\lambda^M)=\{\lambda_0 \theta : \theta \in {\rm Irr}(M/N)\}$.  On the other hand,  $\lambda_0^g \in {\rm Irr}(\lambda^M)$ for an arbitrary element $g \in I_G(\lambda)$. Hence, $\lambda_0^g=\lambda_0 \psi$ for some $\psi \in {\rm Irr}(M/N)$. Since $M/N$ is a non-abelian chief factor of $G$, $\theta(1) >1$ for every $\theta \in {\rm Irr}(M/N)-\{1_{M/N}\}$.  This shows that $\psi=1_{M/N}$. Consequently, $\lambda_0^g=\lambda_0$ which means that  $g\in I_G(\lambda_0)$. Thus, $I_G(\lambda) \leq I_G(\lambda_0)$. Therefore, $I_G(\lambda_0) = I_G(\lambda)$, as wanted.
\end{proof}
%%%%%%%%%%%%%%%%%%%%%%%%%%%%%%%%%%%%%%%%%%%%%%%55
\begin{lemma}\label{1109} Let $N$ be a normal subgroup of $G$ and $\lambda\in {\rm Irr}(N)$. If $T=I_G(\lambda)$ and $\lambda$ is extendible to $\lambda_0 \in {\rm Irr}(T)$, then $I_G(\lambda_0\psi)=T$ for every $\psi \in {\rm Irr}(T/N)$.
\end{lemma}
\begin{proof} By \cite[Corollary 6.17]{isaacs}, $\lambda_0\psi \in {\rm Irr}(\lambda^{T})$. It follows from Clifford correspondence that $(\lambda_0\psi)^{G} \in {\rm Irr}(\lambda^{G})$. This implies that
	$I_G(\lambda_0\psi)=T$, as claimed.
\end{proof}
%%%%%%%%%%%%%%%%%%%%%%%%%%%%%%%%%%%%%%%%%%%%%%%%%%5
For a non-empty subset of real numbers $X$, by ${\rm ave}(X)$ we mean the average of $X$. The following lemma about averages is  well known and its proof is easy.
\begin{lemma} \label{later} {\rm \cite[Lemma 3]{ahanjideh}} Let $X$ be a non-empty subset of real numbers and $\{A_1,\ldots,A_t\}$ be a partition of $X$. If there exists an integer $d$ such that  ${\rm ave}( A_i) \geq  d$ $(\leq d)$ for every $1 \leq i \leq t$, then ${\rm ave}( X) \geq d$  $(\leq d)$. \end{lemma}
%%%%%%%%%%%%%%%%%%%%%%%%%%%%%%%%%%%%%%%%%%%%%%
%%%%%%%%%%%%%%%%
\begin{lemma} \label{Lemma 4.1}{\rm \cite{f} and \cite[Lemma 4.1]{nav}} Let $S$ be a
	finite non-abelian simple group. Then, there exists a
	non-principal irreducible character of $S$ that is extendible to a
	rational-valued character of ${\rm Aut}(S)$.
\end{lemma}
%====================================================
%%%%%%%%%%%%%%%%%%%%
\begin{lemma}\label{Lemma 2.7}{\rm\cite[Lemma 2.7]{Qian}} Let $D$ be a minimal normal subgroup of $G$ and
	$r$ be a prime divisor of $|D|$. If $D$ is a direct product of $s$
	isomorphic non-abelian simple groups, then there exists $\psi \in
	{\rm Irr}(D)$ of degree at least $(r-1)^{s}$ such that $\psi$ is
	extendible to $G$, and if in addition $s \geq 2$, then $\chi(1)
	\geq r-1$ whenever $\chi \in {\rm Irr}(G| D) .$
\end{lemma}
%=======#########################################3
%%%%%%%%%%%%%%%%%%%%%%%%
%=======#########################################3
\begin{lemma} \label{Lemma 2.6} Let $W $ be a normal subgroup of $ G$ which  is  a non-abelian simple group. Then, there exists a character  $\psi \in {\rm Irr}(W)$  that is extendible to a character $\chi \in {\rm Irr}_{\mathbb{Q}}(G)$ such that either  $\psi(1) \geq 8$ or  $W  \cong {\rm PSL}_2(5)$ and $\psi(1)=5$.
\end{lemma}
%%%%%%%%%%%%%%%%%%%%%%%%%%%%%%%
\begin{proof} Let $C=C_G(W)$. Then, $C \cap W ={\bf Z}(W)=1$. Thus,  $W \cong WC/C\unlhd G/C \leq  {\rm Aut}(W) $. So, it is enough to show that there exists a character  $\psi \in {\rm Irr}(W)$ which is extendible  to a character  $\chi \in {\rm Irr}_{\mathbb{Q}}({\rm Aut}(W))$  such that   either $\psi(1) \geq 8$ or $W \cong {\rm PSL}_2(5)$ and $\psi(1)=5$.  Suppose that  $W$ is a simple group of Lie type over a field of order $q$, where $q$ is a power of a prime $p$.  Let ${\rm St}$ be the Steinberg character of ${W}$. Then, by \cite{f},  ${\rm St}$ is extendible to a character $\chi \in {\rm Irr}_{\mathbb{Q}}({\rm Aut}(W))$ and ${\rm St}(1)=|W|_p$. If $W \cong {\rm PSL}_2(5)$, then ${\rm St}(1)=5$. If $W \cong {\rm PSL}_2(7)$, then  Atlas \cite{atlas} guarantees the existence of $\psi \in {\rm Irr}(W)$ that is extendible to ${\rm Irr}_{\mathbb{Q}}({\rm Aut}(W))$ and its degree is $8$.  If $W \not \cong {\rm PSL}_2(5),{\rm PSL}_2(7)$, then  ${\rm St}(1) \geq 8$, as wanted. Next assume that $W \cong {\rm Alt}_n$ ($n \geq 6$) or $W$ is a Sporadic simple group. By Lemma \ref{Lemma 4.1}, there exists a
	non-principal irreducible character $\psi$ of $W$ of maximal  degree  that is extendible to a
	rational-valued character of ${\rm Aut}(W)$. If $W \cong {\rm Alt}_n$  ($n \geq 9$) or $W$ is a Sporadic simple group, then $\psi(1)  \geq 8$. If $W \cong {\rm Alt}_n$ ($6 \leq n \leq 8$), we get from Atlas \cite{atlas}  that there  exists $\psi \in {\rm Irr}(W)$ which is extendible to ${\rm Irr}_{\mathbb{Q}}({\rm Aut}(W))$ and its degree is $8$.   Now, the proof is complete.
\end{proof}
%%%%%%%%%%%%%%%%%%%%%%%%%%%%%%%%%%%%%%%%%%%%%55
\begin{Remark} \label{note1} Let $W $ be a normal subgroup of $ G$ which  is  a non-abelian simple group. Assume that $\psi \in {\rm Irr}(W)$. As stated in the proof of Lemma \ref{Lemma 2.6}, if $\psi$ is extendible to a character in $ {\rm Irr}_{\mathbb{F}}({\rm Aut}(W))$, then $\psi $ is extendible to  a character in $ {\rm Irr}_{\mathbb{F}}(G)$.
\end{Remark}
%%%%%%%%%%%%%%%%%%%%%%%%%%%%%%%%%%%%%%%%%%%%%%%5
%===============================
\begin{lemma}\label{us1} Let $D$ be a minimal normal subgroup of $G$. If $D$ is a direct product of $s$ copies of a non-abelian simple group, then there exists a character  $\psi \in
	{\rm Irr}(D)$ of degree at least $5^s$ such that $\psi$ is
	extendible to a character in  ${\rm Irr}_{\mathbb{F}}(G)$, and  in addition, if ${\rm
		Irr}_{\mathbb{Q},even}(G| D) \neq \emptyset$, then
	$\chi(1) \geq 4$ whenever $\chi \in {\rm
		Irr}_{\mathbb{Q},even}(G| D) $. \end{lemma}
%%%%%%%%%%%%%%%%%%%%%%%%5
\begin{proof} Let $ D = D_1 \times \cdots \times  D_s$, where $D_1,\ldots, D_s$ are isomorphic non-abelian simple
	groups. By Lemma \ref{Lemma 4.1}, we can choose a character  $\theta_1 \in {\rm  Irr}(D_1)$ of maximal degree
	which  is extendible to a rational-valued character of ${\rm Aut}(D_1)$. If  $D_1 \not \cong {\rm PSL}_2(5)$, then we get from Lemma \ref{Lemma 2.6} that $\theta_1(1) \geq 8$. If $D_1 \cong {\rm PSL}_2(5)$, then it is enough to assume that $\theta_1$ is the  Steinberg character of $D_1$. Then, $\theta_1(1)=5$.  Nevertheless, we have $\theta_1(1) \geq 5$.
	
	Now, we follow  the idea given in \cite[Lemma 5.1]{Bian}. Let $H = N_G(D_1)$ and $C = C_G(D_1)$. Since $D$ is a minimal normal subgroup of $G$, $G$ acts transitively on $\{D_1,\ldots,D_s\}$.  So, we can assume that $T=\{x_1,\ldots , x_s\} $  is a right transversal for
	$N_G(D_1)$ in $G$ such that $D_1^{x_i} = D_i$. Let $\theta_i \in  {\rm
		Irr}(D_i)$ be such that $\theta_i(x_i^{-1} d_1x_i) =
	\theta_1(d_1)$ for every $x_i^{-1} d_1x_i \in  D_i = D_1^{x_i}$.
	Put $\theta = \theta_1 \times \cdots \times \theta_s$. Then,
	$\theta \in {\rm Irr}(D)$ and $\theta (1) \geq 5^s$.  However, $H/C \leq {\rm Aut}(D_1)$ and  $D_1 \cap  C =
	{\bf Z}(D_1) = 1$.  So, $D_1C = D_1 \times  C \leq  H$ and $\theta_1 \times 1_C \in {\rm Irr}(D_1 \times C)$. Let $T=I_G(\theta_1 \times 1_C)$. Then, $T \leq H$ and  $ \theta_1 $
	extends to $\theta_1 \times  1_ C \in {\rm Irr}(DC/C)$. So, $\theta_1 \times  1_ C$  extends to a rational-valued character $\psi$ of $H/C$. Thus $T=H$ and $\chi=\psi^{G} \in {\rm Irr}(G)$. Since $\psi \in {\rm Irr}_{\mathbb{F}}(H)$, we observe that $\chi \in {\rm Irr}_{\mathbb{F}}(G)$. Moreover, $\chi(1) \geq 5^s$ and as stated in \cite[Lemma 5.1]{Bian}, $\chi_N=\theta$, as needed.
	
	Finally, let $\chi \in {\rm Irr}_{\mathbb{F},even}(G| D) $. We assume that $\varphi$ is an irreducible constituent of $\chi_D$. Considering the irreducible character degrees of simple groups, $\varphi(1) \geq 3$. Since $2 \mid \chi(1)$ and $\varphi(1) \mid \chi(1)$, we get that $\chi(1) \geq 4$, as wanted.
\end{proof}
%%%%%%%%%%%%%%%%%%%%%%%%%%%%%%%%%%%%%%%%%%%5

For proving the following corollary, we follow  the idea given in some parts of the proof of \cite[Lemma 2.7]{Qian}.
\begin{corollary}\label{corollary 2.7} Let $D$ be a minimal normal subgroup of $G$. If $D$ is non-simple and non-solvable, then $\chi(1) \geq 6$ for every $\chi \in {\rm Irr}(G|D)$.
\end{corollary}
\begin{proof} Since $D$ is a minimal normal subgroup of $G$ and $D$ is non-simple and non-solvable, we have $D=S_1 \times \cdots \times S_t$, where $S_1,\ldots,S_t$ are isomorphic to a non-abelian simple group $S$ and $t \geq 2$. Let  $\chi \in {\rm Irr}(G|D)$ and let $\psi \in {\rm Irr}(\chi_D)$. Then, $\psi=\psi_1\cdots \psi_t$, where $\psi_i \in {\rm Irr}(S_i)$ for every $1 \leq i \leq t$. If $\psi_{i}(1)=1$ for every $1\leqslant i\leqslant t$, then
	$\psi=1_{D}$ and $D\subseteq\ker\chi,$ which is a contradiction. Thus, without loss of generality, we may assume that  $\psi_1(1)>1$.  Hence,  considering  the irreducible character degrees of simple groups, $\psi_1(1) \geq 3$. If there exists $1 \leq j \leq t$ such that $j \neq 1$ and $\psi_j(1) >1$, then $\psi(1) \geq \psi_1(1)\psi_j(1) \geq  6$. Thus, $\chi(1) \geq 6$, as wanted. Otherwise, as  $G$ acts transitively on $\{S_1,\ldots,S_t\}$, we deduce that $I_G(\psi) \leq N_G(S_1)$. So, $\psi(1)[G:N_G(S_1)]$ divides $\chi(1)$.  However, $[G:N_G(S_1)] \geq 2$ and $\psi_1(1) \mid \psi(1)$. Therefore,  $\chi(1) \geq 6$, as wanted.
\end{proof}
%%%%%%%%%%%%%%%%%%%%%%%%%%%%%%%%%%%%%%%%%%%%%555

The following lemma is essentially obtained  from \cite[Lemma 2]{moreto}.
\begin{lemma}\label{cp0}
	Let $G=MC$ be a central product and $T\leq M\cap C$. For every $\tau\in \Irr(G|T)$, we have $\tau(1)=\alpha(1)\beta(1)$ where $\alpha\in \Irr(C|T)$ and $\beta\in \Irr(M|T)$ and   $n^{\mathbb{C}}_d(G|T)=\sum\limits_{t\mid d}n^{\mathbb{C}}_t(C|T)n^{\mathbb{C}}_{d/t}(M|T)$.
\end{lemma}
\begin{proof}
	Since $G = MC$ is a central product with the central  subgroup $T$, there is a
	bijection $(\alpha, \beta) 	\rightarrow \tau$ from $\Irr(M|T) \times \Irr(C|T)$ to $\Irr(G|T)$ such that $\tau(1) = \alpha(1)\beta(1)$. Therefore, $$n^{\mathbb{C}}_d(G|T)=\sum\limits_{t\mid d}n^{\mathbb{C}}_t(C|T)n^{\mathbb{C}}_{d/t}(M|T).$$
\end{proof}
%%%%%%%%%%%%%%%%%%%%%%%%%%%%%%%%%%%%%%%%%
%%%%%%%%%%%%%%%%%%%%%%%%%%%%%%%%%%%%%%%%%%

In  the following lemma,  we sketch the proof of a result that can be  deduced
from the proof of \cite[Proposition 3]{moreto}.

% However the proof of case $(c)$ is not  taken from \cite[Proposition 3]{moreto}, the techniques of the  proof of this  case is  similar to the first two cases. In fact, we have to  use the classification of the non-solvable primitive groups of degree 4.

\begin{lemma}\label{cp}
	Let $G$ be a finite group and $M\unlhd G$ be minimal such that $M$ is non-solvable. Assume
	that $T$ is an abelian minimal normal subgroup of $G$ contained in $M$.  Suppose, if  $[M, R] \not = 1$,
	then $ T \leq [M, R],$ where $R$ is the radical  solvable  subgroup of
	$M$.   Let $\chi\in \Irr(G|T)$. If  $\chi(1)\in \{2,3\}$,
	then there is a normal subgroup $C$ of $G$ such that $G=MC$ is a central product and $T\leq {\bf Z}(M)=M\cap C$. Moreover,
	
	$a)$ if $\chi(1)=2$, then  $M\cong \SL_2(5)$;
	
	$b)$ if $\chi(1)=3$, then  either $M\cong 3\cdot \rm{Alt}_6$ or $M\cong 6\cdot \rm{Alt}_6$.
\end{lemma}
\begin{proof}
	Assume $K=\ker\chi$.  Since $T\not\leq K$, we have $M\not\leq K$, and hence $MK/K$ is non-trivial.
	Also, recall that $M$ is perfect. We deduce that $MK/K$ is perfect and $G/K$ is non-solvable.

As $\chi(1)\in\{2,3\}$, we have
 $\overline G=G/K$ is a linear group of degree $2$ or $3$.
	If $\chi$  is imprimitive, then it is induced from a linear character of a subgroup $\overline H=H/K$ with index
	$\chi(1)\leq 3$, and thus $G/L$ is solvable, where $L=\bigcap\limits_{g\in G}H^g$.  Note that  $\overline L$ is normal in $\overline G$ and $\chi_{\overline L}$ has
	a linear constituent. Therefore, all irreducible constituents of $\chi_{\overline L}$ are linear. Since $\chi \in {\rm Irr}(\overline G)$ is
	faithful, $\overline L$ is abelian, and hence $\overline G$ is solvable, a contradiction.
	
	We conclude that $G/K$ is a primitive linear
	group of degree 2 or 3. Suppose  $C/K={ \bf Z}(G/K)$. 	
	 By the classification of the primitive
	linear groups of degrees 2 and 3 (see  \cite[Chapter III, V]{book}), we  observe that
	$G/C \cong \rm{Alt}_5$ when $\chi(1)=2$ and $G/C\in \{\rm{Alt}_6,\PSL_2(7) \}$ when $\chi(1)=3$. This implies that  $G=MC$, as desired.
	%%%%%%%%%%%%%%%%%%%%%%%%%%%%%%%%%%%%%%%%%%%%%%%%%%%%%%%%%
	
	Recall that
	$M/(M\cap C)\cong MC/C= G/C$ and
	$M\cap C\unlhd M$ is a proper subgroup of $M$. Therefore,
	$ M\cap C $  is a subgroup of radical solvable subgroup of $M$ by the
	minimality of $M$. Since $M/(M\cap C)$ is simple, we obtain that $M\cap C=R$, where $R$ is the radical solvable subgroup of $M$,
	and hence $R\leq C$. Thus $[M, R] \leq K$. But $ T \not\leq K$, so we have $ T \not\leq
	[M, R]$. By the choice of $T$, we have $R = {\bf Z}(M)$. We have proved that
	$M$ is a perfect central cover of the simple group $M/{\bf Z}(M) = M/(M \cap C) \cong G/C$.
	Moreover, as  $T \leq M$ is abelian, we observe that $1 \not = T \leq {\bf Z}(M)$.
	Since $C$ and $M$ are both normal in $G$, we have $[M, C]  \leq C \cap M =
	{\bf Z}(M)$, and so $[C, M, M] = [M, C, M] = 1$. By the three subgroups lemma, we deduce that
	$[M, M, C] = 1$ and hence $[M, C] = 1$ as $M$ is perfect. We conclude that $G = MC$ is
	a central product with a central subgroup $ M \cap  C =  {\bf Z}(M) \not = 1$.
	
	Since $M$ is perfect, we obtain that $T\leq {\bf Z}(M)$ is a non-trivial  subgroup of Schur multiplier of $M/{\bf Z}(M)$. Hence, if $\chi(1)=2$, then  $M\cong \SL_2(5)$, as it is stated in Case $a$. Let $\chi(1)=3$. Looking at Atlas \cite{atlas}, we see that  $\SL_2(7)$ and $\SL_2(9)$ do not contain any faithful irreducible character of degree $3$. Thus,
	we get that $M/{\bf Z}(M)\cong \rm{Alt}_6$ and  $|{\bf Z}(M)|=3$  or $6$, as stated in (b).
\end{proof}
%%%%%%%%%%%%%%%%%%%%%%%%%%%%%%%%%%%%%%%%%%%%%%%%%%%%%%%%%%%%%%%%%
\begin{corollary}\label{cpc}
	Let $G$ be a finite group and $M\unlhd G$ be minimal such that $M$ is non-solvable and $N$ be a non-trivial normal subgroup of $G$ such that $M/N$ is a cheif factor of $G$.   Let $\chi\in \Irr(G|N)$ with $\ker\chi\cap N=1$.  If  $\chi(1)\in \{2,3\}$,
	then there is a normal subgroup $C$ of $G$ such that $G=MC$ is a central product and $N={\bf Z}(M)=M\cap C$. Moreover,
	
	$a)$ if $\chi(1)=2$, then  $M\cong \SL_2(5)$;
	
	$b)$ if $\chi(1)=3$, then  either $M\cong 3\cdot \rm{Alt}_6$ or $M\cong 6\cdot \rm{Alt}_6$.
\end{corollary}
\begin{proof}Noting that $N$ is the radical solvable subgroup of $M$, by Lemma \ref{cp}, we only need to find a minimal normal subgroup $T$ of $N$ such that $\chi\in \Irr(G|T)$ and  if $[M,N]\not =1$, then $T\leq [M,N]$.    According to the assumption, for all minimal normal subgroup $T$ of  $G$ contained in $N$, $\chi\in \Irr(G|T)$. So,  choosing a minimal normal subgroup of $G$ that satisfies the assumption of Lemma \ref{cp} leads us to the results.
\end{proof}
%%%%%%%%%%%%%%%%%%%%%%%%%%%%%%%%%%%%%%%%%%%%%%%%%%%%%%%%%%%%%5
\begin{lemma}\label{simple} Let $M$ be a non-abelian simple group. Then,   $M$ admits some irreducible characters of degrees less than  $5$    if and only if  $M \cong {\rm Alt}_5$ or ${\rm PSL}_2(7)$. \end{lemma}
%%%%%%%%%%%%%%%%%%%%%%%%%%%%5
\begin{proof}  Let $\alpha \in {\rm Irr}(M)-\{1_M\}$ such that $\alpha(1) \leq 4$. By \cite{feit}, $\pi(M)\subseteq \{2,3,5,7\}$. Thus, $M \cong {\rm Alt}_n$ for $5 \leq n \leq 10$, $J_2$, ${\rm PSL}_2(7)$, ${\rm PSL}_2(8)$, ${\rm PSL}_2(17)$, ${\rm PSL}_2(49)$,  ${\rm PSL}_3(3)$, ${\rm PSL}_3(4)$,  ${\rm S}_4(7)$, ${\rm S}_6(2)$, ${\rm O}^+_8(2)$, ${\rm PSU}_3(3)$, ${\rm PSU}_3(5)$, ${\rm PSU}_4(2)$ or ${\rm PSU}_4(3)$, using \cite{herzog} and \cite{shi}. So, Atlas \cite{atlas} shows that the only possibilities are  $M \cong {\rm Alt}_5$  or ${\rm PSL}_2(7)$, as desired.
\end{proof}
%%%%%%%%%%%%%%%%%%%%%%%%%%%%%%%%%%%%%%%%%5
\begin{lemma}\label{quasi} Let $M$ be a perfect group such that ${\bf Z}(M) \neq 1$ and  $M/{\bf Z}(M)$ is non-abelian simple. Then, $M$ admits some faithful irreducible characters of degrees less than  $5$    if and only if  $M/{\bf Z}(M) \cong {\rm Alt}_5, {\rm Alt}_6, {\rm Alt}_7 $, ${\rm PSL}_2(7)$ or ${\rm PSU}_4(2)$. In particular, if $M$ admits some faithful rational-valued  irreducible characters of degrees less than  $5$, then $M/{\bf Z}(M) \cong {\rm Alt}_5$ or ${\rm Alt}_6$.   \end{lemma}
%%%%%%%%%%%%%%%%%%%%%%%%%%%%%%%%%%55
\begin{proof} Let $\psi \in {\rm Irr}(M)$ be faithful such that $\psi(1) \leq 4$. Then, $\psi \in {\rm Irr}(M|{\bf Z} (M))$. So, \cite[Theorem 8.1]{zal} and Atlas \cite{atlas} show that $M/{\bf Z}(M) \cong {\rm Alt}_5, {\rm Alt}_6, {\rm Alt}_7 $, ${\rm PSL}_2(7)$ or ${\rm PSU}_4(2)$. The second part of the lemma follows immediately by checking Atlas \cite{atlas}.
\end{proof}
%%%%%%%%%%%%%%%%%%%%%%%%%%%%%%%%%%%%%5
%%%%%%%%%%%%%%%%%%%%%%%%%%%%%%%%%%%%%5
\begin{lemma}\label{l7} Let $M \cong {\rm PSL}_2(7)$ be a minimal normal subgroup of $G$. Then, $\alpha(1) \geq 6$ for every $\alpha \in {\rm Irr}^*_{\mathbb{Q}}(G|M)$.
\end{lemma}
\begin{proof} Working towards a contradiction, suppose that $\alpha(1) <6$ for some  $\alpha \in {\rm Irr}^*_{\mathbb{Q}}(G|M)$. Let $\psi \in {\rm Irr}(\alpha_M)$. Then, $\psi \neq 1_M$ as $M \not \leq {\rm ker}\alpha$. Considering the irreducible character degrees of ${\rm PSL}_2(7)$,  $\psi(1) \in \{3,6,7,8\}$. However,  $\psi (1) \mid \alpha(1) $ and $\alpha(1) <6$. Thus, $\alpha(1)=\psi(1)=3$. So, $\alpha_M=\psi$. Therefore,  $\psi \in {\rm Irr}^*_{\mathbb{Q}}({\rm PSL}_2(7))$, which is a contradiction, as needed.
\end{proof}
%%%%%%%%%%%%%%%%%%%%%%%%%%%%%%%%%%%%%%%%%%%%%%%

\section{Proof of Theorem A}

In this section, we aim to prove Theorem A in two subsections. Since the method to prove parts (1) and (2) of Theorem A are similar, we collect them in the first subsection. The proof of  part (3) of Theorem A is more complicated and needs more delicate arguments which force us to gather  it in the second subsection.

\subsection{The average of non-linear real and complex irreducible characters.\\}

%\hspace{0.001mm}

In Theorem \ref{key2}, we prove parts (1) and (2) of  Theorem A, using the following lemma.
\begin{lemma}\label{A5} \label{A6}
	Let $G$ be a finite group, $C$  and $M$ be  normal subgroups of $G$ such that $G=MC$ is a central product and, either  $G/C\cong \rm{Alt}_5$ and  $M\cong \SL_2(5)$ or $G/C\cong \rm{Alt}_6$ and $M\cong 3 \cdot \rm{Alt}_6$. Then, for every $G$-invariant character  $\lambda\in \Irr(M \cap C)$ which is not extendable to $G$,
	either  $\acd^*_{\mathbb R}(\lambda^G)\geq 4$ or $M\cong \SL_2(5)$, $\lambda(1)=1$  and $\Irr(\lambda^G)=\Irr_{\mathbb{R}}(\lambda^G)$.
\end{lemma}
\begin{proof}
	Let $N=M \cap C$. Note that as $G=MC$,  $\lambda$ is $G$-invariant and so  $(G, N, \lambda)$ is a character triple isomorphic to $(\Gamma, A, \mu)$, where $\Gamma\cong M$ and $A={\bf Z}(\Gamma)$. First, let $G/C\cong \rm{Alt}_5$ and  $M\cong \SL_2(5)$.  Hence, ${\rm Irr}(\lambda^G)=\{\psi_1,\ldots,\psi_4\}$ such that $\psi_1(1)=\psi_2(1)=2 \lambda(1)$, $\psi_3(1)=4 \lambda(1)$ and $\psi_5(1)=6 \lambda(1)$. Since $\Irr(\lambda^G)$ contains exactly one character of degree $4\lambda(1)$ and one character of degree $6\lambda(1)$, we deduce that $\psi_3$ and $\psi_4$ are real.   Also, $\Irr(\lambda^G)$ contains exactly two irreducible characters $\psi_1$ and $\psi_2$ of degree $2$. Hence,  they are both real or they are complex-conjugate. Therefore,  $\Irr_{\mathbb R}(\lambda^G)=\{\psi_3,\psi_4\}$ or $\Irr_{\mathbb R}(\lambda^G)=\{\psi_1,\ldots,\psi_4\}$.
	So, either $\acd^*_{\mathbb{R}}(\lambda^G)\geq  4$ or $\lambda(1)=1$ and $\Irr_{\mathbb R}(\lambda^G)=\Irr(\lambda^G)$, as desired.
	
	%%%%%%%%%%%%%%%%%%%%%%%%%%%%%%%%%%%%%%%%%%%%
	Next, let $M\cong 3\cdot \rm{Alt}_6$ and $G/C \cong {\rm Alt}_6$. Hence, ${\rm Irr}(\lambda^G)=\{\psi_1,\ldots,\psi_5\}$ such that $\psi_1(1)=\psi_2(1)=3 \lambda(1)$, $\psi_3(1)=6 \lambda(1)$, $\psi_4(1)=9 \lambda(1)$ and $\psi_5(1)=15 \lambda(1)$. Since $\Irr(\lambda^G)$ contains exactly one character of degree $6\lambda(1)$, one of degree $9\lambda(1)$ and one  of degree $15\lambda(1)$, we deduce that $\psi_3$, $\psi_4$ and $\psi_5$ are real.   Also, $\Irr(\lambda^G)$ contains exactly two irreducible characters $\psi_1$ and $\psi_2$ of degree $3$. Hence,  they are both real or they are complex-conjugate. Therefore,  $\Irr_{\mathbb R}(\lambda^G)=\{\psi_3,\psi_4,\psi_5\}$ or $\Irr_{\mathbb R}(\lambda^G)=\{\psi_1,\ldots,\psi_5\}$.
	In both cases $\acd^*(\lambda^G)>4$.
\end{proof}
%%%%%%%%%%%%%%%%%%%%%%%%%%%%%%%%%%%%%%%%%%55
\begin{theorem}\label{key2}
	Let $G$ be a finite non-abelian group and $\Bbb F\in \{\Bbb C, \Bbb R\}$.
	%	\begin{enumerate}
	If ${\rm acd}^{*}_{\Bbb F}(G)<{\rm acd}^{*}_{\Bbb{F}}({\rm SL}_{2}(5))=\frac{29}{8}$, then $G$ is solvable.
	%	\item If ${\rm acd}^{*}_{\Bbb{R}}(G)<{\rm acd}^{*}_{\Bbb{R}}({\rm SL}_{2}(5))=\frac{29}{8}$, then $G$ is solvable.
	%	\item If ${\rm acd}^{*}_{\Bbb{Q}}(G)<{\rm acd}^{*}_{\Bbb{Q}}({\rm A}_{5})=\frac{9}{2}$, then $G$ is solvable.
	%	\end{enumerate}
\end{theorem}
\begin{proof}
	On the contrary, let $G$ be a counterexample of  minimal order. This implies that
	$G$ is  non-solvable  and  $\acdk(G)<29/8$, which means  $$\frac{\sum\limits_{ d >1 }dn^{\Bbb F}_d(G)}{\sum\limits_{d>1}n^{\Bbb F}_d(G)}< 29/8,$$ and this  leads us to
	$$\sum\limits_{d>3}(8d-29)n^{\Bbb F}_d(G)< 13n^{\Bbb F}_2(G)+5n^{\Bbb F}_3(G).$$

	% when $\Bbb F\in \{\mathbb{R},\mathbb{C}\}$ and
	%$$\sum\limits_{d>4}(2d-9)n^{\Bbb F}_d(G)< 5n^{\Bbb F}_2(G)+3n^{\Bbb F}_3(G)+n^{\Bbb F}_4(G),$$
	%when $\Bbb F=\mathbb{Q}$.

	We claim that every minimal normal subgroup of $G$ contained in $G'$ is abelian. To prove this, assume $L$ is a non-abelian
	minimal normal subgroup of $G$ contained in $G'$. It follows from Lemma \ref{us1} that $L$ has a rational-valued irreducible character  $\phi$
	such that $\phi(1)\geq 5$  and it is extendable to a rational-valued   character $\psi \in \Irr(G)$.  We deduce from Gallagher's theorem \cite[Corollary 6.17]{isaacs} that  $\Irr(\phi^G)=\{\lambda\psi \ | \lambda\in \Irr(G/L)\}$, and hence $$ n^{\Bbb F}_2(G/L)\leq \sum\limits_{ d\geq 10} n^{\Bbb F}_d(G|L)\leq \sum\limits_{  d\geq 10} n^{\Bbb F}_d(G),$$
	and
	$$ n^{\Bbb F}_3(G/L)\leq \sum\limits_{ d\geq 15} n^{\Bbb F}_d(G|L)\leq \sum\limits_{  d\geq 15} n^{\Bbb F}_d(G).$$
	
Therefore, as $n_2^{\Bbb F}(G/L)=n_2^{\Bbb F}(G)$,
	$$5n_3^{\Bbb F}(G/L)+ 13n^{\Bbb F}_2(G)\leq \sum\limits_{  10\leq d\leq 14} 13n^{\Bbb F}_d(G)+ \sum\limits_{  d\geq  15} 18n^{\Bbb F}_d(G)\leq \sum\limits_{ d\geq 10} (8d-29)n^{\Bbb F}_d(G). \ \ \ \ \ \ (*)$$
		First, let $L\not\cong \PSL_2(q)$ for $q=5,7$. We show that $n_3^{\Bbb F}(G)=n_3^{\Bbb F}(G/L)$. To do this, we know that
$L$ is a direct product of $k$ copies of some non-abelian simple group. If $k\geq 2$, then it follows from Lemma~2.7 that $\gamma(1)\geq 4$ for all $\gamma\in {\rm Irr}(G|L)$, and hence
$n_3^{\Bbb F}(G)=n_3^{\Bbb F}(G/L)$. Thus, we can assume that $k=1$. Since $L\not\cong \PSL_2(q)$ for $q=5,7$, we see that $L$ does not contain any irreducible character of degree $3$.
Then, $n^{\Bbb F}_3(G)=n^{\Bbb F}_3(G/L)$. Therefore, the inequality we have in  $(*)$ causes a contradiction.
	Now, assume  $L\cong \PSL_2(q)$, where $q=5,7$.
		Then, $L$  has exactly two irreducible characters of degree $3$. Also, $L$ has an irreducible character of degree $d_1$, where $d_1=5$ if $q=5$, and  $d_1=8$ if $q=7$,  which is extendable to a rational-valued irreducible character of  $G$.  Thus, using  Gallagher's theorem, we have $n^{\Bbb F}_1(G)=n^{\Bbb F}_1(G/L)\leq n^{\Bbb F}_{d_1}(G|L)\leq  n^{\Bbb F}_{d_1}(G)$.
		
		%Hence,   $$n^{\Bbb F}_2(G)+n^{\Bbb F}_3(G/L)\leq  n^{\Bbb F}_{2.d_1}(G)+n^{\Bbb F}_{3.d_1}(G).$$
	
	We prove  that  $n^{\Bbb F}_3(G|L)\leq 2n^{\Bbb F}_{d_1}(G)$. It is sufficient to show that $n^{\Bbb F}_3(G|L)\leq 2n^{\Bbb F}_{1}(G)$. Since $L$ does not have any irreducible
character of degree $2$ and any non-principal linear character, we deduce that
	$\chi_L$ is an irreducible character of $L$ for every irreducible character $\chi$ of  degree $3$ in $\Irr(G|L)$.
 Then, we know that $L$ has exactly $2$ irreducible characters of degree $3$, and so,  by Gallagher's theorem, there exist at most $2n^{\Bbb{ F}}_1(G)$ characters of degree $3$ in $\Irr(G|L)$, as wanted. Hence, looking at $(*)$, we have
	
				$$5n_3^{\Bbb F}(G)+13n_2^{\Bbb F}(G)\leq 10n^{\Bbb F}_{d_1}(G)+ \sum\limits_{  10\leq d\leq 14} 13n^{\Bbb F}_d(G)+ \sum\limits_{  d\geq  15} 18n^{\Bbb F}_d(G)\leq \sum\limits_{ d\geq d_1} (8d-29)n^{\Bbb F}_d(G),$$
				a contradiction.
	%Recalling that $n_1(G)=n_1(G/M)=n_{d_2}(G|M)\leq n_{d_2}(G|N)$, and together with the above conclusion we infer that
	
	%$$\sum\limits_
	%{d\geq d_2}
	%(5d - 16)n_d(G|N) \geq (5.d_2 - 16)n_{d_2}(G|N) +\sum\limits_{d\geq d_1}
	%(5d - 16)n_d(G|N)
	%\geq  $$$$
	%(5d_2 - 16)n_1(G) + 9n_1(G|N) + 6n_2(G|N) +  n_3(G|N) - 2kn_1(G)
	%\geq $$$$ 2n_1(G) + 9n_1(G|N) + 6n_2(G|N) + n_3(G|N) > 11n_1(G|N)+6n_2(G|N)+n_3(G|N),$$
	
	%as $5d_2-16\geq 2+2k$, for $k\geq 1$, which is  contradicting $(*)$.

	% when ${\Bbb F}\not =\mathbb{Q}$, and $5n^{\Bbb F}_2(G)\leq \sum\limits_{  d\geq 10} 5n^{\Bbb F}d(G)\leq \sum\limits_{ d\geq 10} (2d-9)n^{\Bbb F}_d(G)$, when ${\Bbb F}=\mathbb{Q}$,  a contradiction.

	Hence,  every minimal normal subgroup of $G$ contained in $G'$ is abelian. Let $M \unlhd G$ be  minimal such
	that $M\leq G'$ is non-solvable. Then clearly $M$ is perfect and contained in the last term of the derived
	series of $G$. Choose a minimal normal subgroup $T$ of $G$ such that $T \leq M$ and when
	$	[M,R]\not =1$ we choose $T \leq  [M,R],$ where $R$ denotes the radical solvable subgroup of $M$ and,  if it is possible we assume $T$ has order 2 .
	We know that $T$ is abelian. It follows that
	$G/T$ is non-solvable since $G$ is non-solvable. By the minimality of $G$, we  have
	$	\acdk(G/T) \geq 29/8$. This implies that
	$\acdk(G) <  29/8 \leq \acdk(G/T)$. It follows that
	$$5n^{\Bbb F}_3(G/T)+13n^{\Bbb F}_2(G/T)\leq  \sum\limits_{ d\geq 4}(8d-29)n^{\Bbb F}_d(G/T)\leq \sum\limits_{d\geq 4}(8d-29)n^{\Bbb F}_d(G)< 13n^{\Bbb F}_2(G)+5n^{\Bbb F}_3(G).$$
	
	% when ${\Bbb F}\not = \mathbb{Q}$ and
	%$$n^{\Bbb F}_4(G/T)+3n^{\Bbb F}_3(G/T)+5n^{\Bbb F}_2(G/T)\leq  %\sum\limits_{ d\geq 5}(2d-9)n^{\Bbb F}_d(G/T)\leq  \sum\limits_{d\geq  5}(2d-9)n^{\Bbb F}_d(G)< 5n^{\Bbb F}_2(G)+3n^{\Bbb F}_3(G)+n^{\Bbb F}_4(G/T),$$ when ${\Bbb F}=\mathbb{Q}$.

	Hence, $\Irr_{\Bbb F}^{*}(G|T)$ contains a character $\chi$ of degree $2$ or $3$.
	%First assume $\chi(1)\in\{2,3\}$.
	So, Lemma \ref{cp} implies that there is a normal subgroup $C$ of $G$ such that $G=MC$ is a central product with central
	subgroup ${\bf Z}(M)=M\cap C$. Moreover, if $\chi(1)=3$, we have $ M\cong 3\cdot \rm{Alt}_6$ or $6 \cdot \rm{Alt}_6$, and
	if $\chi(1)=2$, then $ M\cong \SL_2(5)$. By the choice of $T$, we deduce that $M\cong 6\cdot \rm{Alt}_6$ does not occur
	as otherwise $|T|=2$ and $\Irr(M|T)$ does not contain  any character of degree $3$.  Therefore,  $\chi(1)=3$ happens only when $M\cong 3\cdot\rm{Alt}_6$.

	%Note that  $M/(M\cap C)\cong G/C$, so by the choice of $T$, we get that $T\leq M\cap C={\bf Z}(M)$.
	
	% As ${\bf Z}(M)$ lies in the Schur multiplier of $M/{\bf Z}(M)$, we deduce that $T\cong C_2$ when $M/{\bf Z}(M)\cong A_5$,  and  $T\cong  C_3$, when  $M/{\bf Z}(M)\cong A_6$.
	%First, let  $M/{\bf Z}(M)\cong A_6$ and $T\cong C_2$. Then,  $\chi_M$ is an irreducible character of $M$ of degree $3$. In both cases, the kernel of  irreducible  characters of $M$ of degree $3$  contains $T\cong C_2$, which is impossible. Hence, $M/{\bf Z}(M)\cong A_5$ or $A_6$.
	
	%Note, also  that if  ${\Bbb F}=\mathbb{Q}$,  as $\chi_M$ is an  rational-valued irreducible character of $M$,  then it is easy to see that $\SL_2(5)$ and  $3.A_6$ do not contain any retinal-character of degree 2 and 3, respectively. Hence ${\Bbb F}\not= \mathbb{Q}$.

	%As we explained, if  $M/{\bf Z}(M)\cong A_6$, then $T\cong C_3$.  If  $|{\bf Z}(M)|=6$, then $T$ must be a cyclic group of order $2$,  by the choice of $T$, which is impossible. Hence  ${\bf Z}(M)\cong C_3$.

	Let $G/C\cong \rm{Alt}_6$ and $T\cong {\bf Z}(M)$ is a cyclic group of order $3$.
	It follows from Lemmas \ref{A6} and \ref{cp0} that   $\acdk(G|T)\geq \acdk(M|T)=36/5 > 29/8$ (see Atlas\cite{atlas}).
	Recall that   minimality of $G$  implies that  $\acdk(G/T) > 29/8$. Therefore, $\acdk(G) > 29/8$, a contradiction. So, the proof is complete.
	
	Let $G/C\cong \rm{Alt}_5$. Then $T\cong {\bf Z}(M)$ is of order $2$,   $M\cong \SL_2(5)$, and  $\Irr(G|T)= \Irr(\lambda^G)$, where $\lambda$ is the only  non-trivial character of $T$.
	% First assume $K=\mathbb{Q}$.
	%By \cite[Lemma 2]{moreto}, we have
	%$\acdk(G|T)=\acdk(G|\lambda)\geq \acdk(M|\lambda) =10/2>9/2=g(K)$.
	% Also by the minimality of $G$ we have $\acdk(G/T) > 9/2=g(K)$. Therefore, $\acdk(G|T) > g(K)$, a contradiction. So we may assume $K\not=\mathbb{Q}$.
	%%%%%%%%%%%%%%%%%%%%%%%%%%%%%%%%%%%%%%%%%%%%%%%%%
	%	Since $G = MC$ is a central product with the central  subgroup $T$, we get from Lemma \ref{cp0} that
	%	$$n^{\mathbb{C}}_d(G|T)=\sum\limits_{t\mid d}n^{\mathbb{C}}_t(C|T)n^{\mathbb{C}}_{d/t}(M|T). \ \ \ \ (*)$$
	%	It follows from Lemma \ref{cp0} that
	%	$2=\chi(1)=\alpha(1)\beta(1)$ for some $\beta\in \Irr^*_{{\Bbb F}}(C|T)$  and $\alpha\in \Irr(M|T)$.
	%Since $M\cong \SL_2 (5)$,  $\alpha(1)\in\{2,4,6\}$, and hence $\alpha(1)=2$ and   $\beta(1)=1$. Therefore,   $\beta \in \Irr(C|T)$ is an extension of
	%the unique non-principal linear character  $\lambda\in \Irr(T)$. Thus,  using Gallagher's theorem \cite[Theorem 6.17]{isaacs}, we  have   $n^{{\Bbb F}}_d(C|T) = n^{{\Bbb F}}_d(C/T)$ for each positive integer $d$.
	%Since $G/M\cong C/T$ and $M \leq G'$,  we obtain that
	%$n^{{\Bbb F}}_1(C|T) = n^{{\Bbb F}}_1(C/T)= n^{\Bbb F}_1(G)$.
	%%%%%%%%%%%%%%%%%%%%%%%%%%%%%%%%%%%%%%%%%%%%%%%%%%%%%%%%
	Note that $\Irr(M|T)=\Irr_{\Bbb R}(M|T)$ contains two characters of degree $2$, one of degree $4$ and one of degree $6$.
	We obtain from Lemma \ref{A5} that  either $\acd^*_{\mathbb{R}}(G|T)\geq 4$ or  $\Irr_{\mathbb{R}}(\lambda^G)=\Irr(\lambda^G)$. If the former case occurs,
	then as ${\rm acd}_{\Bbb{R}}^{*}(G/T) \geq  29/8$,  we deduce from Lemma \ref{later} that ${\rm acd}_{\Bbb{R}}^{*}(G) \geq  29/8$, which is a contradiction.
	So, we  assume the second case occurs.
	%%%%%%%%%%%%%%%%%%%%%%%%%%%%%%%%%%%%%%%%%%%%%
By Lemma \ref{cp0}, we have
	$2=\chi(1)=\alpha(1)\beta(1)$ for some $\beta\in \Irr(C|T)$  and $\alpha\in \Irr(M|T)$.
	Since $M\cong \SL_2 (5)$,  $\alpha(1)\in\{2,4,6\}$, and hence $\alpha(1)=2$ and   $\beta(1)=1$. Therefore,   $\beta \in \Irr(C|T)$ is an extension of
	the unique non-principal linear character  $\lambda\in \Irr(T)$. Using Gallagher's theorem \cite[Theorem 6.17]{isaacs}, we  have   $n^{{\Bbb C}}_d(C|T) = n^{{\Bbb C}}_d(C/T)$ for each positive integer $d$.

	%%%%%%%%%%%%%%%%%%%%%%%%%%%%%%%%%%%%%%%%%%%%%%%	

	We show that  $\chi_C$ is not irreducible. In fact, otherwise $\chi_C\in \Irr(C|\lambda)$ and as $\lambda$ extends to $\beta \in \Irr(C)$, Gallagher's theorem yields that $\chi_C=\eta\beta$ for some $\eta\in \Irr(C/T)$.  Recalling that $G/T\cong C/T\times M/T$, we have  $\eta$ extends to $\eta_0\in \Irr(G)$. Now,  applying  \cite[Theorem 6.16]{isaacs},  $\Irr((\eta\beta)^G)=\{\eta_0\theta\ |\ \theta\in \Irr(\beta^G)\}$. Thus, as $\eta\beta$ extends to $G$, we deduce that  $\beta$ must be extendable to $G$, which is not possible as $M$ is perfect.  	
	 Therefore $\chi_C=2\beta$  for some linear character $\beta\in \Irr(C|\lambda)$. 	Noting  that $\chi\in \Irr_{\Bbb R}(G|T)$ is an extension of $\chi_M\in \Irr_{\Bbb R}(M|T)$,  Gallagher's Theorem yields that $\gamma\chi\in \Irr(G|T)$ for all $\gamma \in \Irr(G/M)=\Irr(C/T)$.  Moreover,  $(\gamma\chi)_C=2\gamma\beta$ for all  character $\gamma\in \Irr(C/T)$.   From $\Irr_{\Bbb R}(G|T)=\Irr(G|T)$, we infer that $\gamma\beta$  and $\beta$ are  real-valued. As linear characters do not vanish on any conjugacy classes, we deduce that $\gamma$ is real-valued for every character $\gamma\in \Irr(C/T)$. Hence $n^{\Bbb C}_d(C|T)=n^{\Bbb C}_d(C/T)=n^{\Bbb R}_d(C/T)$.

	Further,  $ \Irr_{\Bbb R}(G/T)=\{\alpha\beta \ |\  \alpha\in \Irr_{\Bbb R}(C/T), \beta\in \Irr_{\Bbb R}(M/T)=\Irr(M/T)\}$.
	% So $\acd^*_{\Bbb C}(G)=\acd^*_{\Bbb R}(G)$ and
	Recall that $\Irr(M|T)$ contains exactly four elements, two of degree $2$, one of degree $4$ and one of degree $6$.	
	%\begin{eqnarray}\nonumber
	% \acd^*_{\Bbb R}(G|T)=\acd^*_{\Bbb C}(G|T)=\frac{14\sum\limits_{d\geq 1}dn^{\Bbb C}_d(C/T)}{4\sum\limits_{d\geq 1}n^{\Bbb C}_d(C/T)}.
	% \end{eqnarray}
	Also, note that  $\Irr(M/T)=\Irr_{\Bbb R}(M/T)$ contains exactly $5$ characters, one of degree $1$, two of degree 3, one of degree 4 and one of degree 5. Thus,  Lemma \ref{cp0} implies that
	% \begin{eqnarray}\nonumber
	%  \acd^*_{\Bbb R}(G/T)=\frac{16\sum\limits_{d\geq 1}dn_d^{\Bbb R}(C/T)-n^{\Bbb R}_1(C/T)+ 5\sum\limits_{d\geq 1}dm_d(C/T)}{5\sum\limits_{d\geq 1}n_d^{\Bbb R}(C/T)-n^{\Bbb R}_1(C/T)+\sum\limits_{d\geq 1}m_d(C/T)}.
	%  \end{eqnarray}
	
	$$	\acd^*_{{\Bbb R}}(G)=\frac{\sum\limits_{\theta\in \Irr_{\Bbb R}^*(G/T)}\theta(1)+\sum\limits_{\theta\in \Irr_{\Bbb R}^*(G|T)}\theta(1)}{|\Irr_{\Bbb R}^*(G|T)|+|\Irr_{\Bbb R}^*(G/T)|}=
	\frac{16\sum\limits_{d\geq 1}dn_d^{\Bbb R}(C/T)-n^{\Bbb R}_1(C/T)+14\sum\limits_{d\geq 1}dn^{\Bbb R}_d(C/T)}{5\sum\limits_{d\geq 1}n_d^{\Bbb R}(C/T)-n^{\Bbb R}_1(C/T)+4\sum\limits_{d\geq 1}n^{\Bbb R}_d(C/T)}$$$$
	=\frac{30\sum\limits_{d> 1}dn_d^{\Bbb R}(C/T)+29n^{\Bbb R}_1(C/T)}{9\sum\limits_{d> 1}n_d^{\Bbb R}(C/T)+8n^{\Bbb R}_1(C/T)}\geq
	\frac{60\sum\limits_{d> 1}n_d^{\Bbb R}(C/T)+29n^{\Bbb R}_1(C/T)}{9\sum\limits_{d> 1}n_d^{\Bbb R}(C/T)+8n^{\Bbb R}_1(C/T)}.
	$$
	Using Lemma \ref{later}, we see that $\acd^*_{\Bbb R}(G)\geq 29/8$, a  contradiction.  Now, let $\Bbb F=\Bbb C$. Note that $\Irr(C/T)$ might contain some characters that are not real-valued.   Again using Lemma \ref{cp0} we coclude that
	
	$$\acd^*_{{\Bbb C}}(G)=\frac{16\sum\limits_{d\geq 1}dn_d^{\Bbb C}(C/T)-n^{\Bbb C}_1(C/T)+ 14\sum\limits_{d\geq 1}dn^{\Bbb C}_d(C/T)}{5\sum\limits_{d\geq 1}n_d^{\Bbb C}(C/T)-n^{\Bbb C}_1(C/T)+4\sum\limits_{d\geq 1}n^{\Bbb C}_d(C/T)}
	=\frac{30\sum\limits_{d> 1}dn_d^{\Bbb C}(C/T)+29n^{\Bbb C}_1(C/T)}{9\sum\limits_{d> 1}n_d^{\Bbb C}(C/T)+8n^{\Bbb C}_1(C/T)}\geq 29/8,$$
	which leads us to  the final contradiction.
\end{proof}
\subsection{The average of non-linear rational irreducible characters. \\}

In this section, we prove part (3) of Theorem A.

\begin{lemma}\label{Nec} Let $M\unlhd G$  be minimal such that $M$ is non-solvable. Suppose that $M/N$ is  a chief factor of $G$ and $D/N=C_{G/N}(M/N)$. If $\varphi \in {\rm Irr}(D)$ such that ${\rm Irr}_{\mathbb{Q}}^*(\varphi^G) \neq \emptyset$, then either ${\rm acd}_{\mathbb{Q}}^*(\varphi^G) \geq 5$ or $MD \leq I_G(\varphi)$.\end{lemma}
\begin{proof}  Let $T=I_G(\varphi)$ and $n=[G:T]$. By Clifford correspondence, ${\rm Irr}(\varphi^G)=\{\alpha^G: \alpha \in {\rm Irr}(\varphi^T)\}$.  Consequently, $\chi(1)  \geq n$ for every $\chi \in {\rm Irr}(\varphi^G)$. So, if $n \geq 5$, then taking the fact ${\rm Irr}_{\mathbb{Q}}^*(\varphi^G) \neq \emptyset$ into account, ${\rm acd}^*_{\mathbb{Q}}(\varphi^G) \geq 5$, as desired. Recall that the assumption on $M$ forces $N$ to be solvable and hence, $M/N$ is non-solvable. Next, suppose that $n \leq 4$.   Then, $G/{\rm Core}_G(T) \leq {\rm Sym}_4$ is solvable, where ${\rm Core}_G(T)=\cap_{g \in G}T^g$. Since $N \leq T$ and $N \unlhd G$, we have $N \leq {\rm Core}_G(T)$. So, $N \leq M \cap {\rm Core}_G(T) \unlhd G$. However, $M \cap {\rm Core}_G(T) \leq  M$ and  $M/N $ is a chief factor of $G$. This implies that either $M \cap {\rm Core}_G(T)=N$ or $M \cap {\rm Core}_G(T)=M$. In the former case, $M/N=M/(M \cap {\rm Core}_G(T))\cong M {\rm Core}_G(T)/{\rm Core}_G(T) \leq G/{\rm Core}_G(T)$, which is a contradiction as $G/{\rm Core}_G(T)$ is solvable. So, $M=M \cap {\rm Core}_G(T) \leq T$. Also,  $\varphi \in {\rm Irr}(D)$. Hence,  $MD \leq I_G(\varphi)=T $, as wanted.
\end{proof}
%%%%%%%%%%%%%%%%%%%%%%%%%%%%%%%%%%%%%%%%%%%5
%%%%%%%%%%%%%%%%%%%%%%%%%%%%%%%%%%%%%%%%%%%%%%%%%%%%%
\begin{lemma} \label{Q20} Let  $C,M \trianglelefteq G$, $M \cong {\rm Alt}_5$ or ${\rm Alt}_6$ and let $G=C \times M$. Then,  ${\rm acd}^*_{\mathbb{Q}}(\varphi^G) \geq 9/2$ for every $\varphi \in {\rm Irr}(C)$ with  ${\rm Irr}^*_{\mathbb{Q}}(\varphi^G) \neq \emptyset$. In particular, if $\varphi(1)>1$ or $M\cong {\rm Alt}_6 $, then ${\rm acd}^*_{\mathbb{Q}}(\varphi^G) >9/2$.
\end{lemma}
\begin{proof} Let $\mu \in {\rm Irr}_{\mathbb{Q}}^*(\varphi^G) $. Since $G=C\times M$ and $\varphi\in {\rm Irr}(C)$, $I_G(\varphi)=G$. Thus, $\mu_C=e\varphi$ for some positive integer $e$. This yields that $\varphi \in  {\rm Irr}_{\mathbb{Q}}(C)$.    Therefore, either $\varphi(1) \geq 2$ and ${\rm  acd}_{\mathbb{Q}}^*(\varphi^{G})\geq 2{\rm acd}_{\mathbb{Q}}(M)$ or $\varphi(1) =1$ and ${\rm acd}_{\mathbb{Q}}^*(\varphi^G)={\rm acd}_{\mathbb{Q}}^*(M)$. However, ${\rm acd}_{\mathbb{Q}}({\rm Alt}_5)=10/3$, ${\rm acd}_{\mathbb{Q}}^*({\rm Alt}_5)=9/2$, ${\rm acd}_{\mathbb{Q}}({\rm Alt}_6)=6$ and ${\rm acd}_{\mathbb{Q}}^*({\rm Alt}_6)=29/4$, using Atlas \cite{atlas}. This implies that ${\rm acd}_{\mathbb{Q}}^*(\varphi^G) \geq 9/2$. Further, if $\varphi(1)>1$ or $M \cong {\rm Alt}_6$, then ${\rm acd}_{\mathbb{Q}}^*(\varphi^G) > 9/2$, as wanted.
\end{proof}
%%%%%%%%%%%%%%%%%%%%%%%%%%%%%%%%%%%%%%%%%%%%%%%5
%%%%%%%%%%%%%%%%%%%%%%%%%%%%%%%%%%%%%%%%%%%%%%%%
\begin{Remark}\label{rem1} {\rm Suppose that $N$ is a normal subgroup of $G$,  $\lambda \in {\rm Irr}(N)$ and  $ \chi \in {\rm Irr}_{\mathbb{Q}}(\lambda^G)$. Let $\varepsilon$ be the $|G|$th root of unity and $\sigma \in {\rm Gal}(\mathbb{Q}(\varepsilon)/\mathbb{Q})$. Then, $\chi^{\sigma}=\chi$. Also, $\lambda^{\sigma} \in {\rm Irr}((\chi^{\sigma})_N)$. Thus, $\lambda$ and $\lambda^{\sigma}$ are irreducible constituents of $\chi_N$. So, $\lambda$ and $\lambda^{\sigma}$ are $G$-conjugate. This implies that $\lambda^G=(\lambda^{\sigma})^G=(\lambda^G)^{\sigma} $. Let $\psi \in {\rm Irr}(\lambda^G)$. Then, $\psi^{\sigma},\psi \in {\rm Irr}(\lambda^G)$ and $\psi^{\sigma}(1)=\psi(1)$. If ${\rm Irr}(\lambda^G)$ has exactly one irreducible constituent of degree $\psi(1)$, then $\psi^{\sigma}=\psi$. This means that $\psi \in {\rm Irr}_{\mathbb{Q}}(\lambda^G)$ and  for every positive integer $m$, $\{\varphi \in {\rm Irr}(\lambda^G): \varphi(1)=m\}=\{\varphi^{\sigma}:\varphi \in {\rm Irr}(\lambda^G),~\varphi(1)=m\}$.}
\end{Remark}
%%%%%%%%%%%%%%%%%%%%%%%%%%%%%%%%%%%%%%%%%%%%%%%5
Noting that $\acd^*_{\Bbb Q}(\rm {Alt}_5)=9/2$, in the following we prove that if $G$ has a chief factor  isomorphic to $\rm {Alt}_5$, then $\acd^*_{Q}(G)\geq 9/2$.
%%%%%%%%%%%%%%%%%%%%%%%%%%%%%%%%%%%%%%%%%
\begin{theorem} \label{Q1} Let $M\trianglelefteq G$  be minimal such that $M$ is non-solvable and let $M/N$ be a chief factor of $G$. If $M/N \cong {\rm Alt}_5$, then ${\rm acd}_{\mathbb{Q}}^*(G) \geq 9/2$.
\end{theorem}
%%%%%%%%%%%%%%%%%%%%%%%%%%%%%%%%%%%%%%%%%%%%%%%%%%%%%%%%%%%%%%%%%%%%%%%%
\begin{proof} Set $D/N=C_{G/N}(M/N)$ and $H=MD$. As $\frac{G/N}{D/N}=\frac{G/N}{C_{G/N}(M/N)} \leq {\rm Aut}(M/N) \cong {\rm Sym}_5$, $[G:H]=[G/N:H/N] \leq 2$.  Let $\varphi_1, \ldots, \varphi_t$ be representatives of  the action of $G$ on ${\rm Irr}(D)$.  By Lemma \ref{Lemma 2.6}, we know ${\rm Irr}_{\mathbb{Q}}^*(G) \neq \emptyset$. Hence, since  ${\rm Irr}(G)= \dot{\cup}_{i=1}^t {\rm Irr}(\varphi_i^G)$, we get ${\rm Irr}_{\mathbb{Q}}^*(\varphi_i^G) \neq \emptyset$ for some $i \in \{1,\ldots,t\}$. If for every  $i \in \{1,\ldots,t\}$ either ${\rm Irr}_{\mathbb{Q}}^*(\varphi_i^G)=\emptyset$ or ${\rm acd}_{\mathbb{Q}}^*(\varphi_i^G) \geq 5$, then we get from Lemma \ref{later} that ${\rm acd}_{\mathbb{Q}}^*(G) \geq 5> 9/2$, as desired. Otherwise,  from Lemma \ref{Nec} there is  $\varphi \in \{\varphi_1, \ldots, \varphi_t\}$ such that ${\rm Irr}_{\mathbb{Q}}^*(\varphi^G) \neq \emptyset$ and $H \leq I_G(\varphi)$.  Recall that the assumption on $M$ implies that $M=M'$ and $N$ is solvable.   Suppose that $\varepsilon$ is the $|G|$th root of unity and $\sigma \in {\rm Gal}(\mathbb{Q}(\varepsilon)/\mathbb{Q})$. Now, we break the proof into the following cases:
	%%%%%%%%%%%%%%%%%%%%%%%%%%%%%%%%%%%%%%%%%%%%%%%%55
	\\{\bf Case i.} Assume that $\varphi(1) >1$. First, let $\varphi$ be extendible to $\chi \in {\rm Irr}(H)$.  Then, Gallagher's theorem  \cite[Corollary 6.17]{isaacs} shows that ${\rm Irr}(\varphi^H)=\{\chi \psi: \psi \in {\rm Irr}(H/D)\}$. However, $H/D \cong M/N \cong {\rm Alt}_5$. So,  ${\rm Irr}(H/D)=\{\lambda_1,\ldots,\lambda_5\}$ such that $\lambda_1(1)=1$, $\lambda_2(1)=\lambda_3(1)=3$, $\lambda_4(1)=4$, $\lambda_5(1)=5$ and $\lambda_5$ is extendible to $\rho_5 \in {\rm Irr}^*_{\mathbb{Q}}({G/D})$. Thus, ${\rm Irr}^*_{\mathbb{Q}}(\varphi^G)=\cup_{i=1}^5{\rm Irr}^*_{\mathbb{Q}}((\chi\lambda_i)^G)$. We have
	\begin{eqnarray}\label{103}
	\alpha(1) \geq 6~{\rm for~ every~}\alpha\in \cup_{i=2}^5{\rm Irr}^*_{\mathbb{Q}}((\chi\lambda_i)^G).
	\end{eqnarray} Therefore, if ${\rm Irr}^*_{\mathbb{Q}}((\chi\lambda_1)^G)=\emptyset$, then ${\rm acd}_{\mathbb{Q}}^*(G) \geq 6$. Now, let ${\rm Irr}^*_{\mathbb{Q}}((\chi\lambda_1)^G)\neq \emptyset$.  Recall that  $I_G(\chi\lambda_1)=I_G(\chi)$,  $(\chi \lambda_1)\lambda_5=\chi\lambda_5 \in {\rm Irr}(H)$ and $\lambda_5$ is extendible to $\rho_5 \in {\rm Irr}^*_{\mathbb{Q}}({G/D})$. If $I_G(\chi)=G$,  we have   ${\rm acd}^*_{\mathbb
		{Q}}((\chi\lambda_1)^G+(\chi\lambda_5)^G)\geq  (\varphi(1)+5\varphi(1))/2\geq 6$ applying Lemma \ref{Lemma 2.2}(ii) when $H \neq G$ and a straight calculation when $H=G$. If $I_G(\chi)\neq G$, then $I_G(\varphi)=I_G(\chi\lambda_1)=I_G(\chi\lambda_5)=H$ regarding $[G:H]=2$. Thus, ${\rm Irr}((\chi\lambda_i)^G) =\{(\chi\lambda_i)^G\}$ for  $i \in \{1,5\}$. However, $(\chi\lambda_1)^G$ and $(\chi\lambda_5)^G$ are the only elements in  ${\rm Irr}(\varphi^G)$ of degrees $2\varphi(1)$ and $10\varphi(1)$, respectively. So, $(\chi\lambda_1)^G, (\chi\lambda_5)^G \in {\rm Irr}^*_{\mathbb{Q}}(\varphi^G)$. This signifies that ${\rm acd}^*_{\mathbb
		{Q}}((\chi\lambda_1)^G+(\chi\lambda_5)^G)\geq (2+10)/2=6$.   As ${\rm Irr}^*_{\mathbb{Q}}(\varphi^G)=\dot\cup_{i \in \{1,5\}}{\rm Irr}^*_{\mathbb{Q}}((\chi\lambda_i)^G) \dot\cup (\cup_{i=2}^4 {\rm Irr}^*_{\mathbb{Q}}((\chi\lambda_i)^G))$, Lemma \ref{later} and \eqref{103} lead to ${\rm acd}_{\mathbb{Q}}^*(\varphi^{G}) \geq 6>9/2$, as desired.  Next, assume that $\varphi$ is not extendible to $H$. Let $I_G(\varphi)=H$. Then,  $(H,D,\varphi)$ is a character triple and $H/D \cong M/N \cong {\rm Alt}_5$. Hence,  $\varphi^H=e_1 \psi_1+e_2 \psi_2+e_3\psi_3+e_4\psi_4$, where $e_i$ is a positive integer and $\psi_i \in {\rm Irr}(\varphi^H)$ for every $i \in \{1,2,3,4\}$.  Also, $\psi_1(1)=\psi_2(1)=2 \varphi(1)$, $\psi_3(1)=4\varphi(1)$ and $\psi_4(1)=6\varphi(1)$. If $G=H$, then Remark \ref{rem1} yields that  $\psi_3, \psi_4 \in {\rm Irr}_{\mathbb{Q}}^*(\varphi^G)$ and $\{\psi_1^{\sigma}, \psi_2^{\sigma}\}=\{\psi_1,\psi_2\}$. Therefore,  ${\rm acd}^*_{\mathbb{Q}}(\varphi^G) \geq 2(2+2+4+6)/4=14/2=7$ or ${\rm acd}^*_{\mathbb{Q}}(\varphi^G)\geq 2(4+6)/2=10$. If $H \neq G$, then   Clifford correspondence shows that  $\varphi^G=e_1 \psi_1^G+e_2 \psi_2^G+e_3\psi_3^G+e_4\psi_4^G$. Since ${\rm Irr}_{\mathbb{Q}}^*(\varphi^G) \neq \emptyset$, Remark \ref{rem1} forces  $\psi_3^G, \psi_4^G \in {\rm Irr}_{\mathbb{Q}}^*(\varphi^G)$ and $\{(\psi_1^G)^{\sigma},(\psi_2^G)^{\sigma}\}=\{\psi_1^G,\psi_2^G\}$. Consequently, ${\rm acd}_{\mathbb{Q}}^*(\varphi^G)\geq(4+4+8+12)\varphi(1)/4=13\varphi(1)/2$ or ${\rm acd}_{\mathbb{Q}}^*(\varphi^G)\geq (8+12)\varphi(1)/2=10\varphi(1)$. Finally, let $H \neq G$ and $I_G(\varphi) = G$. Then,  $(G,D,\varphi)$ is a character triple and $G/D  \cong {\rm Sym}_5$. Since  all non-linear  irreducible characters of the central extensions of ${\rm Sym}_5$ are of degrees at least $4$, we get  ${\rm acd}_{\mathbb{Q}}^*(\varphi^G) \geq 8$.  Nevertheless, in this case,     ${\rm acd}_{\mathbb{Q}}^*(\varphi^G)> 9/2.$
	%%%%%%%%%%%%%%%%%%%%%%%%%%%%%%%%%%%%%%%%%%55
	\\{\bf Case ii.} Assume $H=G$ and $\varphi(1)=1$.  If $\varphi$ is extendible to $\chi \in {\rm Irr}(G)$, then $G' \leq {\rm ker}\chi$. As $N \leq M=M' \leq G'$, we get $N \leq {\rm ker}\varphi={\rm ker}\chi \cap D$.  Recall that $G/N=M/N \times D/N$, $\varphi \in  {\rm Irr}(D/N)$ and ${\rm Irr}^*_{\mathbb{Q}}(\varphi^G) \neq\emptyset$.  Therefore, Lemma \ref{Q20} guarantees  ${\rm acd}_{\mathbb{Q}}^*(\varphi^G)={\rm acd}_{\mathbb{Q}}^*(\varphi^{G/N}) \geq 9/2$.
	
	%%%%%%%%%%%%%%%%%%%%%%%%%%%%%%%%%%%
	Next, suppose that $\varphi$ is not extendible to $G$. Since $G/N=M/N \times D/N$, we get that $N \not\leq {\rm ker}\varphi$. Then,  $(G,D,\varphi)$ is a character triple and $G/D \cong M/N \cong {\rm Alt}_5$. Hence, $\varphi^G=e_1 \psi_1+e_2 \psi_2+e_3\psi_3+e_4\psi_4$, where $e_i$ is a positive integer and $\psi_i \in {\rm Irr}(\varphi^G)$ for every $i \in \{1,2,3,4\}$. Also, $\psi_1(1)=\psi_2(1)=2 $, $\psi_3(1)=4$ and $\psi_4(1)=6$.
	By Remark \ref{rem1},  $\psi_3, \psi_4 \in {\rm Irr}_{\mathbb{Q}}^*(\varphi^G)$ and $\{\psi_1^{\sigma}, \psi_2^{\sigma}\}=\{\psi_1,\psi_2\}$.   Hence, $G$ admits an irreducible character $\alpha$ of degree $2$ such that $\alpha \in {\rm Irr}(\varphi^G)$.  Set $N_0={\rm ker}\varphi \cap N$. Recall that ${\rm ker}\varphi={\rm ker}\alpha \cap D \lhd G$ because $I_G(\varphi)=G$. Thus, $N_0 \unlhd G$ and $\varphi \in {\rm Irr}(G/N_0)$. If $N_0 \neq 1$, then since $G/N_0$ satisfies the assumption of the theorem  as $M/N_0\unlhd G/N_0 $ is minimal such that $M/N_0$ is non-solvable and $\frac{M/N_0}{N/N_0} \cong {\rm Alt}_5$, and  ${\rm acd}^*_{\mathbb{Q}}(\varphi^G)={\rm acd}^*_{\mathbb{Q}}(\varphi^{G/N_0})$, the proof follows from induction. Next, let $N_0=1$.  This shows that ${\rm ker}\alpha \cap N=1$. By Corollary \ref{cpc},  $M \cong {\rm SL}_2(5)$. This forces $N={\bf Z}(M)$. %We can check at once that $C=C_G(M)$. Clearly, $C/N \leq  D/N$. However, $D/N \times M/N=G/N=C/N \times M/N$. This leads to  $D=C=C_G(M)$ and hence, $M \cap D={\bf Z}(M)$ is a cyclic group of order $2$.
	If $\psi_i \in {\rm Irr}_{\mathbb{Q}}(\varphi^G)$ for some $i \in \{1,2\}$,    then  $(\psi_i)_{M} \in {\rm Irr}_{\mathbb{Q}}(M)$ is of degree $2$, which is impossible because of the character table of ${\rm SL}_2(5)$. Therefore, ${\rm acd}^*_{\mathbb{Q}}(\varphi^G)=(4+6)/2=5$. Nevertheless, in this case,
	$ {\rm acd}_{\mathbb{Q}}^*(\varphi^G) \geq 9/2$.
	%%%%%%%%%%%%%%%%%%%%%%%%%%%%%%%%%%%%%%%%%%%%%%%%%%%5
	\\{\bf Case iii.}
	Assume $\varphi(1) =1$, $H \neq G$, $I_G(\varphi)=G$ and $\varphi$ is extendible to $H$.   As stated in Case ii,    $N \leq {\rm ker}\varphi$ because $\varphi$ is extendible to $H$ and $\varphi(1) =1$. Recall that $H/N = D/N \times M/N$. So,  ${\rm Irr}(\varphi^{H})={\rm Irr}(\varphi^{H/N})=\{\varphi \lambda_i: 1 \leq i\leq 5\}$ such that $\lambda_i \in {\rm Irr}(M/N)$ for every $1 \leq i \leq 5$, $\lambda_1(1)=1$, $\lambda_2(1)=\lambda_3(1)=3$, $\lambda_4(1)=4$ and $\lambda_5(1)=5$, considering the character table of ${\rm Alt}_5$ and the fact ${\rm Alt}_5 \cong M/N$. In view of Remark \ref{note1}, $\lambda_4,\lambda_5$ are extendible to $\rho_4,\rho_5 \in {\rm Irr}_{\mathbb{Q}}(G/N)$, respectively. Recall that $(\varphi\lambda_1)\lambda_i=\varphi\lambda_i \in {\rm Irr}(H/N)$ and $I_{G/N}(\varphi\lambda_1)=G/N$. It follows from  \cite[Theorem 6.16]{isaacs} that
	\begin{eqnarray} \label{Q1iia}
	{\rm Irr}_{\mathbb{Q}}((\varphi\lambda_i)^{G/N})=\{\rho \rho_i:\rho \in {\rm Irr}_{\mathbb{Q}}((\varphi\lambda_1)^{G/N})\},~{\rm for~every~}i \in \{4,5\}.
	\end{eqnarray} Also, $I_{G/N}(\lambda_2)=I_{G/N}(\lambda_3)=H/N$ and $I_{G/N}(\varphi)=G/N$. Thus, $$I_{G/N}(\varphi\lambda_2)=I_{G/N}(\varphi\lambda_3)=H/N.$$ Hence, $(\varphi\lambda_2)^{G/N}=(\varphi\lambda_3)^{G/N} $ is the only element of ${\rm Irr}(\varphi^{G/N})$ of degree $6$. Remark \ref{rem1} forces $(\varphi\lambda_2)^{G/N}  \in {\rm Irr}_{\mathbb{Q}}(\varphi^{G/N})$. As $\varphi(1)=1$, ${\rm Irr}^*_{\mathbb{Q}}((\varphi\lambda_1)^{G/N})=\emptyset$.  Put $a=|{\rm  Irr}_{\mathbb{Q}}((\varphi\lambda_1)^{G/N})|$. Regarding $[G/N:H/N]=[G:H]=2$,  we obtain that $ a \leq 2$.  We observe ${\rm Irr}_{\mathbb{Q}}(\varphi^{G/N})=\dot{\cup}_{i \in \{1,2,4,5\}}{\rm Irr}_{\mathbb{Q}}((\varphi\lambda_i)^{G/N})$. Now, \eqref{Q1iia} guarantees  ${\rm acd}_{\mathbb{Q}}^*(\varphi^G)={\rm acd}^*_{\mathbb{Q}}(\varphi^{G/N})= (4a+5a+6)/(2a+1) >9/2$.
	%%%%%%%%%%%%%%%%%%5
	%%%%%%%%%%%%%%%%%%%%%%%%%%%%%%%%%%%%%%%%%%%%%%55
	\\{\bf Case iv.} Assume $
	\varphi(1)=1$ and  $I_G(\varphi) \neq G$. By our former assumption,  $H \leq I_G(\varphi)$. Recall that $[G:H] \leq 2$.  So, $I_G(\varphi)=H \neq G$. It follows from Clifford correspondence  that% $\alpha^G \in {\rm Irr}(G)$ for every $\alpha \in {\rm Irr}(\varphi^H)$. It follows from  \cite[Problem 6.1]{isaacs} that
	\begin{eqnarray} \label{Q1iiia}
	\alpha^G \in {\rm Irr}(G) ~{\rm for ~every}~\alpha \in {\rm Irr}(\varphi^H).
	\end{eqnarray} If $\varphi$ is extendible to $H$, then as mentioned in Case ii,  $N \leq {\rm ker} \varphi$. Thus, $\varphi \in {\rm Irr}(D/N)$. Recall that $H/N=D/N \times M/N$, $M/N \cong {\rm Alt}_5$ and ${\rm Irr}({\rm Alt}_5)=\{\lambda_1,\ldots,\lambda_5\}$ such that $\lambda_1=1_{{\rm Alt}_5}$, $\lambda_2(1)=\lambda_3(1)=3$, $\lambda_4(1)=4$ and $\lambda_5(1)=5$.  It follows from  \eqref{Q1iiia} that ${\rm Irr}(\varphi^{G/N})=\{(\varphi\lambda_i)^{G/N}: 1 \leq i \leq 5\}$. Since  $(\varphi\lambda_1)^{G/N}$, $(\varphi\lambda_4)^{G/N}$ and $ (\varphi\lambda_5)^{G/N}$ are the only elements of ${\rm Irr}(\varphi^{G/N})$ of degrees $2$, $8 $ and $10 $, respectively, we get from Remark \ref{rem1} that $ (\varphi\lambda_1)^{G/N},(\varphi\lambda_4)^{G/N}, (\varphi\lambda_5)^{G/N} \in {\rm Irr}^*_{\mathbb{Q}}(\varphi^{G/N})$. Set $a=|\{(\varphi\lambda_2)^{G/N},(\varphi\lambda_3)^{G/N}\} \cap {\rm Irr}_{\mathbb{Q}}(\varphi^{G/N})|$. Clearly, $a \leq 2$. Therefore, ${\rm acd}^*_{\mathbb{Q}}(\varphi^G)={\rm acd}^*_{\mathbb{Q}}(\varphi^{G/N})=[G:H] (1+6a+4+5)/(3+a) \geq 2(10/3)>9/2 $, as desired.   Next, suppose that $\varphi$ is not extendible to $H$. Then,  $(H,D,\varphi)$ is a character triple. Consequently,  $\varphi^{H}=e_1 \psi_1+e_2 \psi_2+e_3\psi_3+e_4\psi_4$, where $e_i$ is a positive integer and $\psi_i \in {\rm Irr}(\varphi^H)$ for every $i \in \{1,2,3,4\}$. Also, $\psi_1(1)=\psi_2(1)=2 $, $\psi_3(1)=4$ and   $\psi_4(1)=6$.
	By \eqref{Q1iiia},  $\psi_i^G \in {\rm Irr}(G)$ for every $i \in \{1,2,3,4\}$. Therefore, $\varphi^G=e_1 \psi_1^G+e_2 \psi_2^G+e_3\psi_3^G+e_4\psi_4^G$. Since ${\rm Irr}_{\mathbb{Q}}^*(\varphi^G) \neq \emptyset$, Remark \ref{rem1} forces  $\psi_3^G, \psi_4^G \in {\rm Irr}_{\mathbb{Q}}^*(\varphi^G)$ and $\{(\psi_1^G)^{\sigma},(\psi_2^G)^{\sigma}\}=\{\psi_1^G,\psi_2^G\}$. Consequently, ${\rm acd}_{\mathbb{Q}}^*(\varphi^G)=(4+4+8+12)/4=7$ or ${\rm acd}_{\mathbb{Q}}^*(\varphi^G)=(8+12)/2=10$. Nevertheless,    ${\rm acd}_{\mathbb{Q}}^*(\varphi^G)> 9/2.$
	%%%%%%%%%%%%%%%%%%%%%%%%%%%%%%%%%%%%%%5
	%%%%%%%%%%%%%%%%%%%%%%%%%%%%%%%%%%%%%%5
	\\{\bf Case v.}
	Assume  $H \neq G$, $I_G(\varphi)=G  $, $\varphi(1) =1 $, and $\varphi $ is not extendible to $H$. If $N\leq\ker\varphi$, then since $H/N = D/N \times M/N$, we see that $\varphi $ is extendible to $H$, which is a contradiction. Hence, $\varphi\in {\rm Irr}(D|N)$. Let $\gamma\in {\rm Irr}^{*}_{\Bbb{Q}}(\varphi^{G})$. Then as $I_G(\varphi)=G$, we get that $\gamma_{D}=e\varphi$ for some positive integer $e$. This implies that $\varphi \in {\rm Irr}_{\mathbb{Q}}(D|N)$. First, let ${\rm ker}\varphi \cap N=1$. Assume that $M_1\unlhd H$ is non-solvable such that $M_1 \leq M$. Then,  $1\neq M_1N/N \unlhd M/N \cong {\rm Alt}_5$. Thus, $M_1N=M$. This shows that $M/(M_1\cap N)=M_1/(M_1 \cap N) \times N/(M_1 \cap N)$. However, $M=M'$ and $N$ is solvable. Consequently, $N/(M_1 \cap N)=1$, which means that $M_1=M$. Therefore, $M\unlhd H$ is minimal such that $M$ is non-solvable. Since $(H,D,\varphi)$ is a character triple and $H/D \cong M/N \cong {\rm Alt}_5$, we get that ${\rm Irr}(\varphi^H)$ has an element as $\mu$ of degree $2$. Then, ${\rm ker}\mu\cap N={\rm ker}\varphi \cap N=1$ considering $I_G(\varphi)=G$. Hence,  Corollary \ref{cpc} allows us to assume that $M \cong {\rm SL}_2(5)$, $D=C_H(M)$ and $N=M \cap D={\bf Z}(M)$ is a cyclic group of order $2$. Thus, ${\rm Irr}(D|N)={\rm Irr}(\lambda^D)$, where $\lambda$ is the only non-principal  character of $N$. Also,  $\varphi(1)=\lambda(1)=1$. Hence,  $\lambda$ is extendible to $\varphi$.  Therefore,   \cite[Corollary 6.17]{isaacs} yields that ${\rm Irr}(\lambda^D)=\{\varphi \theta : \theta \in {\rm Irr}(D/N)\}$. Let $\theta_1, \ldots,\theta_n$ be representatives of the action of $G/N$ on ${\rm Irr}(D/N)$.   So,
	$
	\{\varphi_1,\ldots,\varphi_t\} = \{\varphi\theta_1, \ldots,\varphi\theta_n,\theta_1, \ldots,\theta_n\}.
	$
	Assume that $\theta \in \{\theta_1, \ldots,\theta_n\}$ is linear  such that  $I_G(\varphi\theta)=G$ and there exists a character  $\alpha \in {\rm Irr}^*_{\mathbb{Q}}((\varphi\theta)^G)$. We claim that  ${\rm acd}^*_{\mathbb{Q}}((\varphi\theta)^{G}+\theta^{G/N})\geq 9/2$.   We have  $\alpha_D=e\varphi\theta$ for some positive  integer $e$. Hence, $\varphi\theta \in {\rm Irr}_{\mathbb{Q}}(D)$. However, $I_G(\varphi)=G$,  $\varphi \in {\rm Irr}_{\mathbb{Q}}(D)$ is linear and the linear characters do not vanish on any conjugacy classes.  Therefore, $\theta \in {\rm Irr}_{\mathbb{Q}}(D/N)$ and $I_{G/N}(\theta)=G/N$.
	Observe that  $(H,D,\varphi\theta)$ is a character triple and $H/D \cong M/N \cong {\rm Alt}_5$. Hence, $(\varphi\theta)^{H}=e_1 \psi_1+e_2\psi_2+e_3\psi_3+e_4\psi_4$, where $e_i$ is a positive integer and $\psi_i \in {\rm Irr}((\varphi\theta)^H)$ for every $i \in \{1,2,3,4\}$. Also, $\psi_1(1)=\psi_2(1)=2 $, $\psi_3(1)=4$ and  $\psi_4(1)=6$. As $(G,D,\varphi)$ is a character triple and $G/D \cong {\rm Sym}_5$, we can check that  $I_G(\psi_1)=I_G(\psi_2)=H$. Hence, $\psi_1^G=\psi_2^G \in {\rm Irr}((\varphi\theta)^G)$. This signifies that ${\rm Irr}((\varphi\theta)^G)=\{\psi_1^G\}\dot\cup{\rm Irr}(\psi_3^G)\dot\cup{\rm Irr}(\psi_4^G)$. Recall that
\begin{eqnarray}\label{645}
	&&\psi_1^G(1)=\psi_2^G(1)=4,\\\nonumber
	&&4 \text{ divides } \alpha(1)  {\rm ~for~ every~} \alpha \in {\rm Irr}(\psi_3^G),\\\nonumber
	&&6 \text{ divides } \alpha(1)~{\rm  for~ every~} \alpha \in {\rm Irr}(\psi_4^G).
	\end{eqnarray} Set
	\begin{eqnarray}\nonumber
	a_{\theta}&=&\left\{\begin{array}{ll}
	1, & \text{ if } \psi_1^G=\psi_2^G \in {\rm Irr}^*_{\mathbb{Q}}(G)\\
	0, & \text{ if } \psi_1^G=\psi_2^G \not\in {\rm Irr}^*_{\mathbb{Q}}(G)
	\end{array},\right.\\ \nonumber
	b_{\theta}&=&\left\{\begin{array}{ll}
	1, & \text{ if there exists } \chi \in {\rm Irr}^*_{\mathbb{Q}}(\psi_3^G)~\text{such that } \chi(1)=4 \\
	0, & \text{ if there exists no element in }  {\rm Irr}^*_{\mathbb{Q}}(\psi_3^G)~\text{of degree } 4.
	\end{array}\right. .%,\\\nonumber
	%c_{\theta}&=&\left\{\begin{array}{ll}
	%1, & \text{ if } \{\chi \in {\rm Irr}^*_{\mathbb{Q}}(\psi_3^G): \chi(1)>4\} \cup  {\rm Irr}^*_{\mathbb{Q}}(\psi_4^G)\neq \emptyset\\
	%0, &  \text{ if } \{\chi \in {\rm Irr}^*_{\mathbb{Q}}(\psi_3^G): \chi(1)>4\} \cup  {\rm Irr}^*_{\mathbb{Q}}(\psi_4^G)= \emptyset
	%\end{array}.\right.
	\end{eqnarray}	
	%%%%%%%%%%%%%%%%%%%%%%%%5
	Further, let
	\begin{eqnarray}\nonumber
	b'_{\theta}&=&|\{\chi \in {\rm Irr}^*_{\mathbb{Q}}(\psi_3^G): \chi(1)=4\}|\\\nonumber
	c'_{\theta}&=&| \{\chi \in {\rm Irr}^*_{\mathbb{Q}}(\psi_3^G): \chi(1)>4\}|+|  {\rm Irr}^*_{\mathbb{Q}}(\psi_4^G)|.
	\end{eqnarray} If there exists $\chi\in {\rm Irr}^*_{\mathbb{Q}}(\psi_3^G)$ of degree $4$, then $\chi(1)=\psi_3(1)$. Hence, $\chi_H=\psi_3$. It follows from  \cite[Theorem 6.17]{isaacs} that ${\rm Irr}^*_{\mathbb{Q}}(\psi_3^G)=\{\chi\rho:\rho \in {\rm Irr}_{\mathbb{Q}}(G/H)\}$. However, $[G:H]=2$. Thus,
	\begin{eqnarray}\label{1237}
	b'_{\theta}=|{\rm Irr}_{\mathbb{Q}}(G/H)|b_{\theta}=2b_{\theta}.
	\end{eqnarray}
	%%%%%%%%%%%%%%%%%%%%%%%%%%%%%%%%
	On the other hand, $\theta \in {\rm Irr}(D/N)$ and $H/N=M/N \times D/N$. As $M/N \cong {\rm Alt}_5$,  \begin{eqnarray}\nonumber
	{\rm Irr}(M/N)=\{\lambda_1,\lambda_2,\lambda_3,\lambda_4,\lambda_5\}
	\end{eqnarray}
	such that $\lambda_1=1_{M/N}$, $\lambda_2(1)=\lambda_3(1)=3$, $\lambda_4(1)=4$ and $\lambda_5(1)=5$. Also, $\lambda_2+\lambda_3$ is a rational-valued  character and in light of  Remark \ref{note1},  $\lambda_4,\lambda_5$ are extendible to $\rho_4,\rho_5 \in{\rm Irr}_{\mathbb{Q}}(G)$, respectively.   We have $(\theta\lambda_1)\lambda_i=\theta\lambda_i \in {\rm Irr}(H/N)$ for every $1 \leq i \leq 5$.  Since $I_G(\theta\lambda_1)=G$,   \cite[Theorem 6.16]{isaacs} forces  \begin{eqnarray} \label{eq3a}
	{\rm Irr}^*_{\mathbb{Q}}((\theta\lambda_i)^{G/N})=\{\rho\rho_i:\rho \in {\rm Irr}_{\mathbb{Q}}((\theta\lambda_1)^{G/N})\}~ {\rm for~}~i \in \{4,5\}
	\end{eqnarray}
	and
	\begin{eqnarray}\label{630}
	|{\rm Irr}^*_{\mathbb{Q}}((\theta\lambda_i)^{G/N})|= |{\rm Irr}_{\mathbb{Q}}((\theta\lambda_1)^{G/N})|~{\rm for~}~i \in \{4,5\}.
	\end{eqnarray}
	Regarding  $I_{G/N}(\lambda_2)=H/N \unlhd G/N$, $\lambda_2+\lambda_3$ is rational-valued, $I_{G/N}(\theta)=G/N$ and $\theta \in {\rm Irr}_{\mathbb{Q}}(D/N)$, we have $(\theta\lambda_2)^{G/N}=(\theta\lambda_3)^{G/N} \in {\rm Irr}^*_{\mathbb{Q}}(G/N)$ is of degree  $6$.  This signifies that
	\begin{eqnarray} \label{638}
	{\rm Irr}((\varphi\theta)^G) \cup {\rm Irr}(\theta^{G})=\{\psi_1^G\}\dot\cup(\dot\cup_{i \in \{3,4\}} {\rm Irr}(\psi_i^G)) \dot\cup\{(\theta\lambda_2)^{G/N}\}\dot\cup(\dot\cup_{i\in \{1,4,5\}} {\rm Irr}((\theta\lambda_i)^{G/N})).
	\end{eqnarray}
	Now, set
	\begin{eqnarray}\nonumber
	B_1&=&\{\gamma  \in {\rm Irr}(D/N):  \gamma(1)=1 ~{\rm and ~} \gamma \lambda_1~{\rm is~extendible~to~an~element~in~}{\rm Irr}_{\mathbb{Q}}(G/N) \}\\ \nonumber
	B_2&=&{\rm Irr}(D/N)-B_1.
	\end{eqnarray}
	Recall that $\theta(1)=1$. We have the following sub-cases:
	\\{\bf Sub-case a.}  Assume that $\theta \in B_2$.
	As $\theta(1)=1$, we deduce that $\theta\lambda_1$ is not extendible to any element of ${\rm Irr}_{\mathbb{Q}}(G/N)$. Therefore, $\beta(1) \geq 2$ for every $\beta \in {\rm Irr}_{\mathbb{Q}}((\theta\lambda_1)^{G/N})$ (if exists). Hence, $|{\rm Irr}^*_{\mathbb{Q}}((\theta\lambda_1)^{G/N})|=|{\rm Irr}_{\mathbb{Q}}((\theta\lambda_1)^{G/N})|$. Also, it follows from \eqref{eq3a} that $\alpha(1) \geq 8$ for every $\alpha \in {\rm Irr}^*_{\mathbb{Q}}((\theta\lambda_4)^{G/N})\cup {\rm Irr}^*_{\mathbb{Q}}((\theta\lambda_5)^{G/N})$ (if exists). Thus, \eqref{645}, \eqref{1237}, \eqref{630} and \eqref{638}   signify  that $${\rm acd}^*_{\mathbb{Q}}((\varphi\theta)^{G}+\theta^{G/N}) \geq  \frac{ 4a_{\theta}+4(2b_{\theta})+6c'_{\theta}+6+(8+8+2)|{\rm Irr}^*_{\mathbb{Q}}((\theta\lambda_1)^{G/N})|}{a_{\theta}+2b_{\theta}+c'_{\theta}+1+3|{\rm Irr}^*_{\mathbb{Q}}((\theta\lambda_1)^{G/N})|} \geq 9/2.$$
	\\{\bf Sub-case b.}  Assume that $\theta \in B_1$.
	So,  we have ${\rm Irr}^*_{\mathbb{Q}}((\theta\lambda_1)^{G/N})=\emptyset$ and as $[G:H]=2$,  ${\rm Irr}_{\mathbb{Q}}((\theta\lambda_1)^{G/N})$ contains two linear characters. Therefore, \eqref{645}, \eqref{1237}, \eqref{630} and \eqref{638}    show that $${\rm acd}^*_{\mathbb{Q}}((\varphi\theta)^{G}+\theta^{G/N}) \geq \frac{ 4a_{\theta}+4(2b_{\theta})+6c'_{\theta}+6+2(4)+2(5)}
	{a_{\theta}+2b_{\theta}+c'_{\theta}+1+2+2}\geq 9/2,$$ as claimed.
	
	We get from the above statements that if ${\rm ker}\varphi \cap N =1$, then ${\rm acd}^*_{\mathbb{Q}}((\varphi\theta)^{G}+\theta^{G/N})\geq 9/2$ for every $\theta\in \{\theta_1, \ldots, \theta_n\}$ such that $I_G(\varphi\theta)=G$, $\varphi\theta(1)=1$ and ${\rm Irr}^*_{\mathbb{Q}}((\varphi\theta)^G)\neq \emptyset$.
	Next, let ${\rm ker}\varphi \cap N \neq 1$. Then, $ \{\varphi\theta_1, \ldots,\varphi\theta_n,\theta_1, \ldots,\theta_n\}\subseteq
	\{\varphi_1,\ldots,\varphi_t\}.$  Set $K={\rm ker}\varphi \cap N$. Regarding $I_G(\varphi)=G$, we infer that $K \trianglelefteq G$. Replacing $G/K$ by $G$ in the above statement, we get that $M/K \cong {\rm SL}_2(5)$ and  if $\theta \in \{\theta_1,\ldots,\theta_n\}$ such that $I_{G/K}(\varphi\theta)=G/K$,  $\varphi\theta(1)=1$ and ${\rm Irr}^*_{\mathbb{Q}}((\varphi\theta)^G) \neq \emptyset$, then ${\rm Irr}^*_{\mathbb{Q}}(\theta^{G/N}) \neq \emptyset$ and ${\rm acd}^*_{\mathbb{Q}}((\varphi\theta)^{G}+\theta^{G/N})={\rm acd}^*_{\mathbb{Q}}((\varphi\theta)^{G/K}+\theta^{G/N}) \geq 9/2$.  If $\psi \in {\rm Irr}(D|N)-{\rm Irr}(D/K)$ such that $\psi(1)=1$ and $I_G(\psi)=G$, then $\psi_1=\psi_K \in {\rm Irr}(K)-\{1_{K}\}$ and $I_G(\psi_1)=G$. Hence, ${\rm ker}\psi_1 \unlhd G$ and $K/{\rm ker}\psi_1 \leq {\bf Z}(G/{\rm ker}\psi_1)$. Since $K/{\rm ker}\psi_1 \leq {\bf Z}(M/{\rm ker}\psi_1)$ and $M/{\rm ker}\psi_1$ is perfect, we deduce that $M/{\rm ker}\psi_1$ is  a Schur representation group for $M/K \cong {\rm SL}_2(5)$.  This forces ${\rm ker}\psi_1=K$, a contradiction. This implies that for every  $\psi \in {\rm Irr}(D|N)-{\rm Irr}(D/K)$, either $\psi(1) >1$ or $I_G(\psi) \neq G$. So, $\psi$ does not satisfy this case.
	
	Recall that ${\rm Irr}^*_{\mathbb{Q}}(G) =\dot{\cup}_{i=1}^t {\rm Irr}^*_{\mathbb{Q}}(\varphi_i^G)$.  Nevertheless, if none of the characters  $\varphi_1,\ldots,\varphi_t$ satisfies the assumptions of Case v, then Cases i-iv and Lemma \ref{later} force ${\rm acd}^*_{\mathbb{Q}}(G) \geq 9/2$. Otherwise, without loss of generality suppose that $\theta_1,\ldots,\theta_m$ are the only elements of $\{\theta_1,\ldots,\theta_n\}$ such that $I_G(\varphi\theta_i)=I_G(\varphi)=G$,  $\varphi\theta_i(1)=1$ and ${\rm Irr}^*_{\mathbb{Q}}((\varphi\theta_i)^G),{\rm Irr}^*_{\mathbb{Q}}(\varphi^G) \neq \emptyset$ for $i \in \{1,\ldots,m\}$ and for some $\varphi \in \{\psi_1,\ldots,\psi_t\}$. Therefore, as stated above  $\varphi\theta_i,\theta_i \in \{\varphi_1,\ldots,\varphi_t\}$ are distinct and ${\rm acd}^*_{\mathbb{Q}}((\varphi\theta_i)^{G}+\theta_i^{G/N}) \geq 9/2$ for every $1 \leq i \leq m$. If $\alpha \not \in \{\varphi\theta_1,\ldots,\varphi\theta_m,\theta_1,\ldots,\theta_m\}$, then we conclude from Lemma \ref{Nec} and Cases i-iv that ${\rm acd}^*_{\mathbb{Q}}(\alpha^G) \geq 9/2$. Therefore, Lemma \ref{later} forces ${\rm acd}^*_{\mathbb{Q}}(G) \geq 9/2$, as wanted.
	%%%%%%%%%%%%%%%%%%%%%%%%%%%%%%%%%%%%%%5
\end{proof}%%%%%%%%%%%%%%%%%%%%%%%%%%%%%%%%%%%%%%%%%%%%%%%5
%%%%%%%%%%%%%%%%%%%%%%%%%%%%%%%%%%%%%%%%%%%%%%%%%%%%%%%%%%%%%%%%%%%%%
\begin{lemma}\label{Q22} Let $M,C\unlhd G$ and $G=MC$  such that  $M \cap C={\bf Z}(M)$ and either  $M \cong 6\cdot {\rm Alt}_6$ or $M \cong 3\cdot {\rm Alt}_6$. If $\varphi \in {\rm Irr}(C|{\bf Z}(M))$ is linear  such that $I_G(\varphi)=G$, ${\bf Z}(M) \cap {\rm ker}\varphi=1$ and ${\rm Irr}^*_{\mathbb{Q}}(\varphi^G) \neq \emptyset$, then $ {\rm acd}_{\mathbb{Q}}^*(\varphi^G) \geq 6$.
\end{lemma}
\begin{proof} By the assumption,  $(G,C,\varphi)$ is a character triple. First, let $M \cong 6 \cdot {\rm Alt}_6$. Since ${\bf Z}(M)\cap {\rm ker}\varphi =1$, $\varphi^G=\Sigma_{i=1}^4 e_i\psi_i$, where $e_i$ is a positive integer and $\psi_i \in {\rm Irr}(\varphi^G)$ for every  $1 \leq i \leq 4$. Also, $\psi_1(1)=\psi_2(1)=6$ and $\psi_3(1)=\psi_4(1)=12$. Thus,   $\alpha(1) \geq 6$ for every $\alpha \in {\rm Irr}(\varphi^G)$. Therefore, ${\rm acd}^*_{\mathbb{Q}}(\varphi^G) \geq 6$, as wanted. Next, let $M \cong 3\cdot {\rm Alt}_6$.   Then,  $\varphi^G=\Sigma_{i=1}^5 e_i \psi_i$, where $e_i$ is a positive integer and $\psi_i \in {\rm Irr}(\varphi^G)$ for every $1 \leq i \leq 5$. Also, $\psi_1(1)=\psi_2(1)=3$, $\psi_3(1)=6$, $\psi_4(1)=9$ and $\psi_5(1)=15$. It follows from Remark \ref{rem1} that $\psi_5 \in {\rm Irr}^*_{\mathbb{Q}}(\varphi^G)$. Therefore, ${\rm acd}^*_{\mathbb{Q}}(\varphi^G) \geq 6$, as wanted.
\end{proof}
%%%%%%%%%%%%%%%%%%%%%%%%%%%%%%%555
\begin{lemma} \label{Q222} Let $C,M \unlhd G$ such that $M \cap C={\bf Z}(M)$, $M \cong {\rm SL}_2(9)$ and $G=MC$. If there exists a character $\varphi \in {\rm Irr}(C|{\bf Z}(M))$ with $I_G(\varphi)=G$, $\varphi(1)=1$ and ${\rm Irr}^*_{\mathbb{Q}}(\varphi^G) \neq \emptyset$, then ${\rm acd}^*_{\mathbb{Q}}(G) > 9/2$.
\end{lemma}
\begin{proof}  Let  $\mu \in {\rm Irr}^*_{\mathbb{Q}}(\varphi^G)$. By our assumption,  $G/{\bf Z}(M)=C/{\bf Z}(M) \times M/{\bf Z}(M)$. Regarding  $I_G(\varphi)=G$, $\mu_C=e\varphi$  for some positive integer $e$. Therefore, $\varphi \in {\rm Irr}_{\mathbb{Q}}(C)$. Note that $C\cap M={\bf Z}(M) $ is  a cyclic group of order $2$ and $\varphi \in {\rm Irr}(C|{\bf Z}(M))$.  Thus,   $\lambda=\varphi_{{\bf Z}(M)}$ is the only non-principal character of ${\bf Z}(M)$.  It follows from \cite[Corollary 6.17]{isaacs} that ${\rm Irr}(C|{\bf Z}(M))={\rm Irr}(\lambda^{C})=\{\varphi \theta : \theta \in {\rm Irr}(C/{\bf Z}(M))\}$. Let $\theta_1, \ldots,\theta_n$ be representatives of the action of $G/{\bf Z}(M)$ on ${\rm Irr}(C/{\bf Z}(M))$. Then, $\theta_1, \ldots,\theta_n,\varphi\theta_1, \ldots,\varphi\theta_n$ are representatives of the action of $G$ on ${\rm Irr}(C)$. By Atlas \cite{atlas},   ${\rm Irr}(M/ {\bf Z}(M))=\{ \lambda_1,\ldots,\lambda_{7}\}$ such that $\lambda_1=1_{M/{\bf Z}(M)}$, $\lambda_2,\lambda_3,\lambda_6,\lambda_7 \in {\rm Irr}^*_{\mathbb{Q}}(M/{\bf Z}(M))$,  $\lambda_2(1)=\lambda_3(1)=5$,  $\lambda_4(1)=\lambda_5(1)=8$, $\lambda_6(1)=9$ and $\lambda_7(1)=10$. Also, ${\rm Irr}(M|{\bf Z}(M))=\{ \psi_1,\ldots,\psi_{6}\}$ such that $\psi_1,\psi_2 \in {\rm Irr}^*_{\mathbb{Q}}(M|{\bf Z}(M))$,    $\psi_1(1)=\psi_2(1)=4$,  $\psi_3(1)=\psi_4(1)=8$ and $\psi_5(1)=\psi_6(1)=10$. Let $\theta \in \{\theta_1,\ldots,\theta_n\}$. Then,  $I_{G}(\varphi\theta)=G$ because $I_G(\varphi)=G$, $\theta\in {\rm Irr}(C/{\bf Z}(M))$ and $G/{\bf Z}(M)=C/{\bf Z}(M) \times M/{\bf Z}(M)$.  So, if $\tau \in {\rm Irr}^*_{\mathbb{Q}}((\varphi\theta)^{G})  $, then  $\tau_{C}=e_0\varphi\theta $ for some positive integer $e_0$. This forces $\varphi\theta \in {\rm Irr}_{\mathbb{Q}}(C)$. As $\varphi \in {\rm Irr}_{\mathbb{Q}}(C)$ is linear, we conclude  $\theta \in {\rm Irr}_{\mathbb{Q}}(C/{\bf Z}(M))$. Therefore, $\theta\lambda_6,\theta\lambda_7 \in {\rm Irr}^*_{\mathbb{Q}}(\theta^{G/{\bf Z}(M)})$ because $G/{\bf Z}(M) =C/{\bf Z}(M) \times M/{\bf Z}(M)$. We see that $(G,C,\varphi\theta)$ is a character triple. Hence, ${\rm Irr}^*_{\mathbb{Q}}((\varphi\theta)^{G})$ has at most two elements of degree $4$. Also, as stated above, $\theta\lambda_6,\theta\lambda_7 \in {\rm Irr}^*_{\mathbb{Q}}(\theta^{G/{\bf Z}(M)})$, and the degrees of the other  elements of ${\rm Irr}^*_{\mathbb{Q}}((\varphi\theta)^{G}) \dot{\cup} {\rm Irr}^*_{\mathbb{Q}}(\theta^{G/{\bf Z}(M)})$ (if exist)  are greater than $5$.  This shows that if ${\rm Irr}_{\mathbb{Q}}^*(\varphi\theta) \neq \emptyset$, then
	$
	{\rm acd}_{\mathbb{Q}}^*((\varphi\theta)^G+\theta^{G/{\bf Z}(M)}) > 9/2
	$ considering Lemma \ref{later}.  Recall that by Lemma \ref{Q20}, if ${\rm Irr}^*_{\mathbb{Q}}( \theta^{G/{\bf Z}(M)}) \neq \emptyset$, then ${\rm acd}^*_{\mathbb{Q}}(\theta^{G/{\bf Z}(M)}) >9/2$. However, ${\rm Irr}^*_{\mathbb{Q}}(G)=\dot{\cup}_{i=1}^n({\rm Irr}^*_{\mathbb{Q}}((\varphi\theta_i)^G) \dot{\cup}{\rm Irr}^*_{\mathbb{Q}}(\theta_i^{G/{\bf Z}(M)}))$. Hence, Lemma \ref{later} shows that ${\rm acd}^*_{\mathbb{Q}}(G) > 9/2$, as desired.
\end{proof}
%%%%%%%%%%%%%%%%%%%%%%%%%%%%%%%%%55%%%%%%%%%%%%%%%%%%%%%%%%%%%%%%%%%%%%%%%%%%%%%%%%%%%%%%%%%%%%%%%%%%%%%
\begin{theorem} \label{Q2}  Let $M\unlhd G$  be minimal such that $M$ is non-solvable. If $N \unlhd G$, $N \leq M$ and  $M/N \cong {\rm Alt}_6$, then ${\rm acd}_{\mathbb{Q}}^*(G) > 9/2$.
\end{theorem}
%%%%%%%%%%%%%%%%%%%%%%%%%%%%%%%%%%%%%%%%%%%%%%%%%%%%%%%%%%%%%%%%%%%%%%%%
\begin{proof} Set $D/N=C_{G/N}(M/N)$ and $H=MD$. Then, $\frac{G}{D} \cong \frac{G/N}{D/N}=\frac{G/N}{C_{G/N}(M/N)} \leq {\rm Aut}(M/N) $.  Let $\varphi_1, \ldots, \varphi_t$ be representatives of the action of $G$ on ${\rm Irr}(D)$.  By Lemma \ref{Lemma 2.6}, we know ${\rm Irr}_{\mathbb{Q}}^*(G) \neq \emptyset$. Hence, since  ${\rm Irr}(G)= \dot{\cup}_{i=1}^t {\rm Irr}(\varphi_i^G)$, we get ${\rm Irr}_{\mathbb{Q}}^*(\varphi_i^G) \neq \emptyset$ for some $i \in \{1,\ldots,t\}$. If  either ${\rm Irr}_{\mathbb{Q}}^*(\varphi_i^G)=\emptyset$ or ${\rm acd}_{\mathbb{Q}}^*(\varphi_i^G) \geq 5$ for every  $i \in \{1,\ldots,t\}$, then we get from Lemma \ref{later} that ${\rm acd}_{\mathbb{Q}}^*(G) \geq 5> 9/2$, as desired. Otherwise,  from Lemma \ref{Nec} there is a character  $\varphi \in \{\varphi_1, \ldots, \varphi_t\}$ such that ${\rm Irr}_{\mathbb{Q}}^*(\varphi^G) \neq \emptyset$ and $H \leq I_G(\varphi)$. Note that the assumption on $M$ guarantees that $M=M'$ and $N$ is solvable. % Suppose that $\varepsilon$ is the $|G|$-th root of unitary and $\sigma \in {\rm Gal}(Q(\varepsilon)/Q)$.
	Now, we continue  the proof in the following cases:
	%%%%%%%%%%%%%%%%%%%%%%%%%%%%%%%%%%%%%%%%%%%%%%%%%%%%%%%%%%%
	\\{\bf Case i.} Assume that $\varphi(1) >1$. Set $T=I_G(\varphi)$. We have  ${\rm Alt}_6 \cong  M/N \cong H/D \leq T/D \leq {\rm Aut}(M/N) $ and  $[{\rm Aut}(M/N):H/D]=4$. First, let $\varphi$ be extendible to $\chi \in {\rm Irr}(T)$.  Then,  \cite[Corollary 6.17]{isaacs} shows that ${\rm Irr}(\varphi^T)=\{\chi \psi: \psi \in {\rm Irr}(T/D)\}$. However, $H/D \leq T/D \leq {\rm Aut}(M/N) $. So, $T \unlhd G$ and by Atlas \cite{atlas}, ${\rm Irr}(T/D)=\{\psi_1,\ldots,\psi_t\}$ for some  $t \leq 7$ such that $\psi_1(1)=1$, $\psi_2(1)=9$, $\psi_i(1)\geq 5$  for every  $2 \leq i \leq t$ and $\psi_2$ is extendible to $\rho_2 \in {\rm Irr}^*_{\mathbb{Q}}({G/D})$. Thus, ${\rm Irr}^*_{\mathbb{Q}}(\varphi^G)=\cup_{i=1}^t{\rm Irr}^*_{\mathbb{Q}}((\chi\psi_i)^G)$. We have
	\begin{eqnarray}\label{1033}
	\alpha(1) \geq 10~{\rm for~ every~}\alpha\in \cup_{i=2}^t{\rm Irr}^*_{\mathbb{Q}}((\chi\psi_i)^G).
	\end{eqnarray} Therefore, if ${\rm Irr}^*_{\mathbb{Q}}((\chi\psi_1)^G)=\emptyset$, then ${\rm acd}_{\mathbb{Q}}^*(\varphi^G) \geq 10$. Now, let ${\rm Irr}^*_{\mathbb{Q}}((\chi\psi_1)^G)\neq \emptyset$. If $I_G(\chi)=G$, then  $T=I_G(\varphi)=I_G(\chi)=G$. This implies that $\chi \psi_1 \in {\rm Irr}^*_{\mathbb{Q}}(\varphi^G)$ and hence, $(\chi \psi_1)\psi_2=\chi\psi_2 \in {\rm Irr}^*_{\mathbb{Q}}(\varphi^G)$. Hence,  ${\rm acd}^*_{\mathbb
		{Q}}((\chi\psi_1)^G+(\chi\psi_2)^G)={\rm acd}^*_{\mathbb
		{Q}}(\chi\psi_1+\chi\psi_2)= (\varphi(1)+9\varphi(1))/2> 10$. If $I_G(\chi)\neq G$, then $ I_G(\varphi)=I_G(\chi\psi_1)=I_G(\chi\psi_2)=I_G(\chi)$ regarding Lemma \ref{1109}.  Also, $(\chi\psi_1)^G, (\chi\psi_2)^G \in {\rm Irr}^*_{\mathbb{Q}}(\varphi^G)$. This signifies that ${\rm acd}^*_{\mathbb
		{Q}}((\chi\psi_1)^G+(\chi\psi_2)^G)>10 > 9/2$.  As ${\rm Irr}^*_{\mathbb{Q}}(\varphi^G)=\dot\cup_{i \in \{1,2\}}{\rm Irr}^*_{\mathbb{Q}}((\chi\psi_i)^G) \cup (\cup_{i=3}^t {\rm Irr}^*_{\mathbb{Q}}((\chi\psi_i)^G))$, Lemma \ref{later} and \eqref{1033} lead to ${\rm acd}_{\mathbb{Q}}^*(\varphi^{G}) > 9/2$.  Next, assume that $\varphi$ is not extendible to $I_G(\varphi)$. Then,  $(T,D,\varphi)$ is a character triple and $H/D \leq T/D \leq {\rm Aut}(M/N) $. Recall that   all non-linear  irreducible characters of the central extensions of subgroups of  ${\rm Aut}(M/N)$  containing $M/N$ are of degrees at least $3$  and by Clifford correspondence, ${\rm Irr}(\varphi^G)=\{\alpha^G:\alpha\in {\rm Irr}(\varphi^T)\}$. Thus, since $\varphi(1) \geq 2$, we get  ${\rm acd}_{\mathbb{Q}}^*(\varphi^G) \geq 6$. We conclude that  ${\rm acd}_{\mathbb{Q}}^*(\varphi^{G}) > 9/2$.
	\\{\bf Case ii.}
	Assume $\varphi(1)=1$ and $H=G$. By our former assumption,  $I_G(\varphi)=H=G$.   If $\varphi$ is extendible to $\mu \in {\rm Irr}(G)$, then $G' \leq {\rm ker}\mu$. Hence, $N \leq M=M'\leq G' \leq {\rm ker}\mu$. This yields that  $N \leq {\rm ker}\mu \cap D={\rm ker}\varphi$. Recall that $G/N=D/N \times M/N$, $\varphi \in  {\rm Irr}(D/N)$, and ${\rm Irr}^*_{\mathbb{Q}}(\varphi^{G/N})={\rm Irr}^*_{\mathbb{Q}}(\varphi^G) \neq \emptyset$. So, Lemma \ref{Q20} forces ${\rm acd}^*_{\mathbb{Q}}(\varphi^{G})={\rm acd}^*_{\mathbb{Q}}(\varphi^{G/N}) >9/2$.
	
	Next, suppose that $\varphi$ is not extendible to $G$. Then, $\varphi \in {\rm Irr}(D|N)$ and  $(G,D,\varphi)$ is a character triple. Recall that   $G/D \cong M/N \cong {\rm Alt}_6$.   Set $K={\rm ker}\varphi$ and let $\varphi_0=\varphi_N$. Since $\varphi \in {\rm Irr}(D|N)$ and $\varphi(1)=1$, we can check at once that $\varphi_0 \in {\rm Irr}(N)-\{1_N\}$, $I_G(\varphi_0)=I_G(\varphi)=G$, $\varphi_0(1)=\varphi(1)=1$, and ${\rm ker}\varphi_0=N \cap K$. Therefore, $N\neq N\cap K \trianglelefteq G $ and $ N/(N \cap K) \leq {\bf Z}(G/(N \cap K))$. Since  $N/(N \cap K) \leq {\bf Z}(M/(N \cap K))$, $(M/(N \cap K))'=M'/(N \cap K)=M/(N \cap K)$ and $\frac{M/(N \cap K)}{N/(N \cap K)} \cong M/N \cong {\rm Alt}_6$, we get that $M/(N \cap K)$ is a Schur representation group for ${\rm Alt}_6$. Clearly, $\varphi \in {\rm Irr}(D/(N \cap K)|N/(N \cap K))$. If $M/(N \cap K) \cong 3\cdot {\rm Alt}_6$ or $6\cdot{\rm Alt}_6$, then   Lemma \ref{Q22} implies that ${\rm acd}^*_{\mathbb{Q}}(\varphi^{G})= {\rm acd}^*_{\mathbb{Q}}(\varphi^{G/(N \cap K)}) \geq 6$. Set $M/(N \cap K) \cong 2\cdot {\rm Alt}_6$. In light of Lemma \ref{Q222}, ${\rm acd}^*_{\mathbb{Q}}(G/(N \cap K))>9/2$. Now, let
	$ \psi \in \{\varphi_1,\ldots,\varphi_t\}-{\rm Irr}(D/(N \cap K))$.  If $\psi(1) \geq 2$ or $I_G(\psi)\neq G$, then  ${\rm acd}^*_{\mathbb{Q}}(\psi^G) > 9/2 $ considering Case i and  Lemma \ref{Nec}. Otherwise,  $\psi(1)=1$ and  $I_G(\psi)=G$.  As $\psi \not\in {\rm Irr}(D/(N \cap K))$, we see that $N \cap K \not \leq {\rm ker}\psi$. Thus, $\psi_0=\psi_{N \cap K} \in {\rm Irr}(N \cap K)-\{1_{N \cap K}\}$. Setting  $L={\rm ker}\psi$ and $L_0={\rm ker}\psi_0$, we have  $L_0=L\cap (N \cap K) \lneqq N \cap K$.  Since  $\psi_0(1)=\psi(1)=1$ and  $I_G(\psi_0)=I_{G}(\psi)=G$, we get that $L_0,L \trianglelefteq G$ and   $(N\cap K)/L_0 \leq {\bf Z}(G/L_0)$.   Regarding  $(N\cap K)/L_0 \leq {\bf Z}(M/L_0)$ and $(M/L_0)'=M/L_0$, we obtain  that  $M/L_0$ is a Schur representation group for $\frac{M/L_0}{(N \cap K)/L_0}\cong M/(N \cap K) \cong {\rm SL}_2(9)$. This forces $M/L_0 \cong 6 \cdot {\rm Alt}_6$. Rewriting  the same argument given for $\varphi$ shows that $M/(N \cap L)$ is a Schur representation group for $M/N \cong {\rm Alt}_6$. Recall that $N/L_0$ is the Schur Multiplier of ${\rm Alt}_6$  and as $\psi \in {\rm Irr}(D|N \cap K)$, we observe   that $N \cap L \neq N \cap K$. Taking into account that  $N/L_0$ is cyclic, $[N:N\cap L] \neq [N:N \cap K]=2$. Nevertheless, $M/(N \cap L) \cong 6 \cdot {\rm Alt}_6$ or $ 3 \cdot {\rm Alt}_6$. So,  Lemma \ref{Q22} forces ${\rm acd}^*_{\mathbb{Q}}(\psi^{G/(N \cap L)}) \geq 6$. It follows from  Lemma \ref{later} that   ${\rm acd}_{\mathbb{Q}}^*(G) > 9/2$.
	%%%%%%%%%%%%%%%%%%%%%%%%%%%%%%%%%%%%%%%%%%%%%%%%%%%%%%%%%%%%%55
	\\{\bf Case iii.}  Assume $H \neq G$, $\varphi(1)=1$, $I_G(\varphi)=G$ and   $\varphi$ is  extendible to $H$. Since $I_G(\varphi)=G$, we get that $\varphi \in {\rm Irr}_{\mathbb{Q}}(D)$ as stated in the proof of Lemma \ref{Q222}.  Also, $\varphi$ is extendible to $H$ and $\varphi(1)=1$. Hence,  $N \leq {\rm ker}\varphi $. Recall that $H/N = D/N \times M/N$ and $M/N \cong {\rm Alt}_6$. So,  ${\rm Irr}(\varphi^{H})={\rm Irr}(\varphi^{H/N})=\{\varphi \lambda_1,\ldots,\varphi \lambda_7\}$, where $\lambda_1,\ldots,\lambda_7 \in {\rm Irr}(M/N)$, $\lambda_1(1)=1$, $\lambda_2(1)=\lambda_3(1)=5$, $\lambda_4(1)=\lambda_5(1)=8$, $\lambda_6(1)=9$ and $\lambda_7(1)=10$ considering the character table of ${\rm Alt}_6$. So, if ${\rm Irr}^*_{\mathbb{Q}}((\varphi\lambda_1)^G)= \emptyset$, then ${\rm Irr}^*_{\mathbb{Q}}(\varphi^{G})={\rm Irr}^*_{\mathbb{Q}}(\varphi^{G/N})={\cup}_{i=2}^7 {\rm Irr}^*_{\mathbb{Q}}((\varphi\lambda_i)^{G/N})$. As $\lambda_i(1) \geq 5$ for every $2 \leq i \leq 7$, ${\rm acd}^*_{\mathbb{Q}}(\varphi^{G}) \geq 5 > 9/2$. Next, let ${\rm Irr}^*_{\mathbb{Q}}((\varphi\lambda_1)^G)\neq \emptyset$. In view of Remark \ref{note1}, $\lambda_6$ is extendible to $ \chi \in {\rm Irr}^*_{\mathbb{Q}}(G/N)$. Recall that $(\varphi\lambda_1)\lambda_6=\varphi\lambda_6 \in {\rm Irr}(H/N)$ and $I_{G/N}(\varphi\lambda_1)=G/N$. It follows from  Lemma \ref{Lemma 2.2}(ii) that
	${\rm acd}^*_{\mathbb{Q}}((\varphi\lambda_1)^{G/N}+(\varphi\lambda_6)^{G/N}) > 10/2$. On the other hand, if $2 \leq i \leq 7$, then $\varphi\lambda_i(1) \geq 5$. Since ${\rm Irr}^*_{\mathbb{Q}}(\varphi^G)={\rm Irr}^*_{\mathbb{Q}}(\varphi^{G/N})=\dot{\cup}_{i \in \{1,6\}}{\rm Irr}^*_{\mathbb{Q}}((\varphi\lambda_i)^{G/N})\dot{\cup} (\cup_{2 \leq i\leq 7,~i \neq 6}{\rm Irr}^*_{\mathbb{Q}}((\varphi\lambda_i)^{G/N}))$,  Lemma \ref{later} forces ${\rm acd}^*_{\mathbb{Q}}(\varphi^G) > 9/2$.
	%%%%%%%%%%%%%%%%%%%%%%%%%%%%%%%%%%%%%%%%%%%%%55
	%%%%%%%%%%%%%%%%%%%%%%%%%%%%%%%%%%%%%%%%%%%%%55
\\{\bf Case iv.} Assume $\varphi(1)=1$ and $I_G(\varphi) \neq G$. Set $T=I_G(\varphi)$. So, $[G:T] \geq 2$.   By our former assumption, $H \leq I_G(\varphi)$. It follows from Clifford correspondence  that
	\begin{eqnarray} \label{Q2iii}
	{\rm Irr}(\varphi^G)=\{\alpha^G:\alpha \in {\rm Irr}(\varphi^T)\}.
	\end{eqnarray} If $\varphi$ is extendible to $H$, then since $\varphi(1)=1$ and $N \leq M=M' \leq H'$, we deduce $N \leq {\rm ker} \varphi$. Thus, $\varphi \in {\rm Irr}(D/N)$. Recall that $H/N=M/N \times D/N$, $M/N \cong {\rm Alt}_6$ and ${\rm Irr}({\rm Alt}_6)=\{\lambda_1,\ldots,\lambda_7\}$ such that $\lambda_1=1_{{\rm Alt}_6}$, $\lambda_2(1)=\lambda_3(1)=5$, $\lambda_4(1)=\lambda_5(1)=8$, $\lambda_6(1)=9$ and $\lambda_7(1)=10$.  If ${\rm Irr}^*_{\mathbb{Q}}((\lambda_1\varphi)^{G/N}) =\emptyset$, then  ${\rm Irr}^*_{\mathbb{Q}}(\varphi^{G/N})={\cup}_{i=2}^7{\rm Irr}^*_{\mathbb{Q}}((\lambda_i\varphi)^{G/N})$. Hence, since  $\chi(1) \geq (\lambda_i\varphi)(1)\geq  5$ for every $\chi \in \cup_{i=2}^7{\rm Irr}((\lambda_i\varphi)^{G/N})$, we get ${\rm acd}^*_{\mathbb{Q}}(\varphi^{G})={\rm acd}^*_{\mathbb{Q}}(\varphi^{G/N}) \geq 5>9/2$, as desired.  Now, assume that ${\rm Irr}^*_{\mathbb{Q}}((\lambda_1\varphi)^{G/N}) \neq \emptyset$. Observing  character table of ${\rm Alt}_6$ and Remark \ref{note1}, $\lambda_6$ is extendible to $\rho_6 \in {\rm Irr}_{\mathbb{Q}}(G/N)$. As $H \leq T \leq G$ and $G/D \leq {\rm Aut}(M/N)$, we have  $T\trianglelefteq G$. Taking into account the former fact, we have  ${\rm Irr}_{\mathbb{Q}}((\lambda_1\varphi)^{G/N}) ={\rm Irr}^*_{\mathbb{Q}}((\lambda_1\varphi)^{G/N}) \neq \emptyset $.  Further, $\lambda_6(\lambda_1\varphi)=\lambda_6\varphi \in {\rm Irr}(H)$ and $I_G(\lambda_6\varphi)=T$. It follows from Lemmas \ref{Lemma 2.2}(ii) and \ref{2.2} that    ${\rm acd}^*_{\mathbb{Q}}((\lambda_6\varphi)^{G/N}+(\lambda_1\varphi)^{G/N})={\rm acd}_{\mathbb{Q}}((\lambda_6\varphi)^{G/N}+(\lambda_1\varphi)^{G/N})
	> 9/2$. Recall that ${\rm Irr}^*_{\mathbb{Q}}(\varphi^{G/N})={\rm Irr}^*_{\mathbb{Q}}((\lambda_1\varphi)^{G/N})\dot{\cup} {\rm Irr}^*_{\mathbb{Q}}((\lambda_6\varphi)^{G/N}) \dot{\cup}({\cup}_{i\in\{2,3,4,5,7\}}{\rm Irr}^*_{\mathbb{Q}}((\lambda_i\varphi)^{G/N}))$ and $\chi(1) \geq 5$ for every $\chi \in \cup_{i=2}^7{\rm Irr}((\lambda_i\varphi)^{G/N})$. Thus, Lemma \ref{later} shows that ${\rm acd}^*_{\mathbb{Q}}(\varphi^G)={\rm acd}^*_{\mathbb{Q}}(\varphi^{G/N}) > 9/2$.
	
	Next, suppose that $\varphi$ is not extendible to $H$. Then, since  $(H,D,\varphi)$ is a character triple,   $\varphi^{H}=\Sigma_{i=1}^ne_i \psi_i$, where $n$ and $e_i$ are positive integers,  $\psi_i \in {\rm Irr}(\varphi^H)$  and $\psi_i(1)\geq 3$ for every $i \in \{1,\ldots,n\}$. By \eqref{Q2iii},  ${\rm Irr}(\varphi^G)=\{\alpha^G: \alpha \in \cup_{i=1}^n{\rm Irr}(\psi_i^T)\}$. Since $\alpha^G(1)=[G:T] \alpha(1) \geq 2\psi_i(1)\geq 6$ for some $1 \leq i \leq n$ and every $\alpha \in {\rm Irr}(\psi_i^T)$, we get that   ${\rm acd}_{\mathbb{Q}}^*(\varphi^G) \geq 6 > 9/2$. Nevertheless, in this case, ${\rm acd}_{\mathbb{Q}}^*(\varphi^G)> 9/2$.
	%%%%%%%%%%%%%%%%%%%%%%%%%%%%%%%%%5
	\\{\bf Case v.}
	Assume  $H \neq G$, $I_G(\varphi)=G  $, $\varphi(1) =1 $ and $\varphi $ is not extendible to $H$.
	%%%%%%%%%%%%%%%%%%%%%%%%%%%%%%%%%%%%%%%%%%%%%%%%%
	Observe that  $(G,D,\varphi)$ is a character triple and $G/D \leq  {\rm Aut}({\rm Alt}_6)$. Set  $K={\rm ker}\varphi$ and  $\overline{G}=G/(N \cap K)$.  As stated in Case (ii), we can see that $\overline{M}$ is a Schur representation group for $M/N \cong {\rm Alt}_6$ and $\overline{N} \leq {\bf Z}(\overline{M})$. If $\alpha(1)\geq 5$ for every $\alpha \in {\rm Irr}^*_{\mathbb{Q}}(\varphi^{\overline{G}})$, then ${\rm acd}^*_{\mathbb{Q}}(\varphi^{\overline{G}}) \geq 5 > 9/2$. Next, let $\alpha \in {\rm Irr}^*_{\mathbb{Q}}(\varphi^{\overline{G}})$ with $\alpha(1)\leq 4$. Then, $\alpha \in {\rm Irr}^*_{\mathbb{Q}}(\overline{G}|\overline{N})$ and  considering the  degrees of elements of  ${\rm Irr}(\overline{M}|\overline{N})$ forces $\alpha_{\overline{M}} \in {\rm Irr}(\overline{M})$. Hence, $\alpha_{\overline{M}} \in {\rm Irr}_{\mathbb{Q}}(\overline{M}|\overline{N})$. Therefore, $G/D  \cong {\rm Sym}_6$ and $\overline{M}\cong  {\rm SL}_2(9)$, applying Atlas \cite{atlas}. So,   ${\rm Irr}(\varphi^{G})=\{ \psi_1,\ldots,\psi_6\}$, where  $\psi_1(1)=\psi_2(1)=\psi_3(1)=\psi_4(1)=4$, $\psi_5(1)=16$ and $\psi_6(1)=20$. As ${\rm Irr}^*_{\mathbb{Q}}(\varphi^G) \neq \emptyset $,  we get from Remark \ref{rem1} that $\psi_5,\psi_6 \in {\rm Irr}^*_{\mathbb{Q}}(\varphi^G)$. Therefore, we can check immediately that  ${\rm acd}_{\mathbb{Q}}^*(\varphi^G) > 9/2$.
	%\\{\bf Case iv.}
	%Assume  $H \neq G$, $I_G(\varphi)=G  $, $\varphi(1) =1 $ and $\varphi $ is not extendible to $H$.
	%%%%%%%%%%%%%%%%%%%%%%%%%%%%%%%%%%%%%%%%%%%%%%%%%%
	%Observe that  $(H,C,\varphi)$ is a character triple and $H/C \cong M/Z(M) \cong {\rm A}_6$. So,  $\phi^{H}=\sum_{i=1}^6 e_i \psi_i$, where $\psi_i \in {\rm Irr}(H)$, $\psi_1(1)=\psi_2(1)=4$, $\psi_3(1)=\psi_4(1)=8$ and $\psi_5(1)=\psi_6(1)=10$. Also, $I_G(\psi_3)=I_G(\psi_4)=H $ and $I_G(\psi_5)=I_G(\psi_6)=H $. Since $\psi_3^G=\psi_4^G$ and $\psi_5^G=\psi_6^G$ are the only irreducible constituents  of $\phi^G$ of degrees $16$ and  $20$, respectively , we get from Remark \ref{rem1} that $(\psi_3)^{G},(\psi_5)^{G} \in {\rm Irr}_{\mathbb{Q}}(\phi^G)$. As  ${\rm Irr}(\varphi^G)= \cup_{i=1}^6 {\rm Irr}(\psi_i^G)$, $[G:H]=2$ and $\varphi^H$ has exactly two irreducible constituents of degree $4$, we get that ${\rm Irr}_{\mathbb{Q}}(\varphi^G)$ has at most $4$ elements of degree $4$. Also, $(\psi_3)^{G},(\psi_5)^{G} \in {\rm Irr}_{\mathbb{Q}}(\phi^G)$ and for every $\alpha \in {\rm Irr}_{\mathbb{Q}}(\varphi^G)$, either $\alpha(1) = 4$ or $\alpha(1) \geq 8$. Therefore, we can check immediately that  ${\rm acd}_{\mathbb{Q}}^*(\phi^G) > 9/2$.
	
	Recall that ${\rm Irr}^*_{\mathbb{Q}}(G)=\dot{\cup}_{i=1}^t {\rm Irr}^*_{\mathbb{Q}}(\varphi_i^G)$. Thus, if $H \neq G$, then  it follows from Cases i, iii-v and Lemmas \ref{later}, \ref{Nec} that ${\rm acd}_{\mathbb{Q}}^*(G) >9/2$. Also, if $G=H$, then Case ii  completes the proof.
\end{proof}
%%%%%%%%%%%%%%%%%%%%%%%%%%%%%%%%%%%%%%%%%%%%%5555
%%%%%%%%%%%%%%%%%%%%%%%%%%%%%%%%%%%%%%%%%%%%%%%%%%%%%%%%%%%%%55
\begin{lemma}\label{prod} Let $M\unlhd  G$ be minimal such that $M$ is non-solvable and let   $M/{\bf Z}(M)$ be a chief factor of $G$. If $1 \neq {\bf Z}(M) \leq {\bf Z}(G)$, $\lambda \in {\rm Irr}({\bf Z}(M))-\{1_{{\bf Z}(M)}\}$  and there exists a  character $\chi \in {\rm Irr}^*_{\mathbb{Q}}(\lambda^G)$ with  $\chi(1)\leq 4$, then $M/{\bf Z}(M)$ is simple.     \end{lemma}
%%%%%%%%%%%%%%%%%5
%%%%%%%%%%%%%%%%%5
\begin{proof} Since $I_G(\lambda)=G$, we observe that ${\rm ker} \lambda \trianglelefteq G$. If ${\rm ker}\lambda \neq 1$, then as ${\rm Irr}(\lambda^G)={\rm Irr}(\lambda^{G/{\rm ker}\lambda})$ and ${\bf Z}(M/{\rm ker}\lambda)={\bf Z}(M)/{\rm ker}\lambda \leq {\bf Z}(G)/{\rm ker}\lambda \leq {\bf Z}(G/{\rm ker}\lambda)$, we get from induction that $$\frac{M/{\rm ker}\lambda}{{\bf Z}(M/{\rm ker}\lambda)}=\frac{M/{\rm ker}\lambda}{{\bf Z}(M)/{\rm ker}\lambda} \cong M/{\bf Z}(M)$$ is simple, as desired. Next, let ${\rm ker}\lambda=1$. As $M/{\bf Z}(M)$ is  non-solvable, $M/{\bf Z}(M) =S_1/{\bf Z}(M) \times \cdots\times S_t/{\bf Z}(M)$ such that
	$S_1/{\bf Z}(M) ,\ldots, $$S_t/{\bf Z}(M)$ are isomorphic to a non-abelian simple group $S$.  Recall that the assumption on $M$ leads to $M=M'$. Working towards a contradiction, suppose that  $t\geq 2$. Then, $G$ acts transitively on $\{S_1,\ldots, S_t\}$. As $S_1/{\bf Z}(M)$ is non-abelian simple, $S_1'{\bf Z}(M)=S_1$.  Also, ${\bf Z}(M) \leq {\bf Z}(G)$ and $S_1 ,\ldots,S_t$ are $G$-conjugate.  Thus, $S_1'\cap {\bf Z}(M)=S_i' \cap {\bf Z}(M)$ for every $1 \leq i \leq t$. Since $S_1'/(S_1' \cap {\bf Z}(M)) \cong S_1/{\bf Z}(M)$ is non-abelian simple, we get that  $M/(S_1'\cap {\bf Z}(M))=S_1'/(S_1'\cap {\bf Z}(M)) \times \cdots \times  S_t'/(S_1'\cap {\bf Z}(M)) \times {\bf Z}(M)/(S_1'\cap {\bf Z}(M))$. Regarding $M=M'$, ${\bf Z}(M)=S_1'\cap {\bf Z}(M)$. This yields that $S_1'=S_1'{\bf Z}(M)=S_1$.
	Thus, $S_i'=S_i$ for every $1 \leq i\leq t$.
	Let  $\psi \in {\rm Irr}(\chi_M)$ and $\psi_i \in {\rm Irr}(\psi_{S_i})$. Regarding $S_i'=S_i$ and $\psi_i\in {\rm Irr}(S_i|{\bf Z}(M))$, we have  $\psi_i(1) \geq 2$.   If $\psi_1(1)=2$, then $S_1 \cong {\rm SL}_2(5)$, by Corollary \ref{cpc}. We know that $S_1 \unlhd M$ and $S_1/{\bf Z}(M)$ is simple, $ C_M(S_1)\cap S_1={\bf Z}(S_1)={\bf Z}(M)$ and $M/C_M(S_1) \leq {\rm Aut}(S_1)={\rm Sym}_5$. This forces $S_2\cdots S_t \leq C_M(S_1)$. Hence, $I_M(\psi_1)=M$ and  $M$ is a central product of $S_1$ and $S_2 \cdots S_t$. Regarding  Lemma \ref{cp0},  $4 \mid \psi(1)$ and as $\psi(1) \leq \chi(1)
	\leq 4$, we conclude  that $\psi(1)=4$.  Consequently, $\chi_M=\psi$. Since $I_M(\psi_1)=M$,   $\psi_{S_1}=e_0\psi_1$ for some positive integer $e_0$. This yields that $\chi_{S_1}=e_0\psi_1$. Note that $\chi \in {\rm Irr}^*_{\mathbb{Q}}(G)$.  This guarantees that $\psi_1\in {\rm Irr}^*_{\mathbb{Q}}(S_1)$. Since $S_1 \cong {\rm SL}_2(5)$ and ${\rm SL}_2(5)$ does not admit any rational-valued irreducible character of degree $2$, we get a contradiction.  If $\psi_1(1)\geq 3$, then since $\psi_1(1)$ divides $\psi(1)$, $\psi(1)$ divides $ \chi(1)$ and $ \chi(1) \leq 4$, we deduce that $\psi_1(1)=\psi(1)=\chi(1)$. Therefore, $\chi_M=\psi$. Setting  $K={\rm ker}\psi$, we have  $K \trianglelefteq G$. So, $K{\bf Z}(M)/{\bf Z}(M) \trianglelefteq G/{\bf Z}(M)$. However, $K{\bf Z}(M) \leq M$ and $M/{\bf Z}(M)$ is a chief factor of $G$. This shows that $K{\bf Z}(M)=M$ or $K{\bf Z}(M)={\bf Z}(M)$. In the former case, $K=K/(K\cap {\bf Z}(M)) \cong M/{\bf Z}(M)$ is non-solvable because $K \cap {\bf Z}(M)={\rm ker}\lambda=1$. This is a contradiction with our assumption on $M$ as $K \lneq M$. In the latter case, $K{\bf Z}(M)={\bf Z}(M)$. This forces $K=K \cap {\bf Z}(M)={\rm ker}\lambda$ because $I_M(\lambda)=M$. This implies that $K=1$. Since $M/N$ is a non-abelian chief factor of $G$,  $|M/N|$ is divisible  by a prime $r \geq 5$. It follows from the main result of \cite{feit1} that a maximal subgroup of  a Sylow $r$-subgroup of  $M$ is normal in $M$. This forces $|M/N|_r=r$. This means that $M/N$ is non-abelian simple, as wanted.
\end{proof}
%%%%%%%%%%%%%%%%%%%%%
%%%%%%%%%%%%%%%%%%%%%%
\begin{proposition}\label{Q3} Let $M \unlhd G$ be minimal such that $M $ is non-solvable. Suppose that $M/N$ is a chief factor of $G$ such that $M/N \not \cong {\rm Alt}_5,{\rm  Alt}_6$. If $\lambda \in {\rm Irr}(N)$ such that ${\rm Irr}^*_{\mathbb{Q}}(\lambda^G) \neq \emptyset$, then
	${\rm acd}_{\mathbb{Q}}^*(\lambda^G) \geq 9/2$.
\end{proposition}
\begin{proof}
	By our assumption on $M$, $M=M'$, $M/N$ is non-abelian and $N$ is solvable.
	Set $T=I_G(\lambda)$. Then, one of the following cases occurs:
	\\{\bf Case i.} Suppose that $[G:T] \geq 5$. Since ${\rm Irr}(\lambda^G)=\{\alpha^G:\alpha \in {\rm Irr}(\lambda^T)\}$, we deduce that $\chi(1) \geq 5$ for every  $\chi \in {\rm Irr}_{\mathbb{Q}}^*(\lambda^G)$. Thus, ${\rm acd}_{\mathbb{Q}}^*(\lambda^G) \geq 5$, as wanted.
	%%%%%%%%%%%%%%%%%%%%%%%%%%%%%%%%%%%%%55
	\\{\bf Case ii.} Assume that $2 \leq [G:T] \leq 4$. Then, $\lambda\neq 1_N$ and $G/ {\rm Core}_G(T) \leq {\rm Sym}_4$ is solvable, where ${\rm Core}_G(T)=\cap_{g \in G}T^g$.   Since $N \leq T$ and $N \unlhd G$, we have $N \leq {\rm Core}_G(T)$. So, $N \leq M \cap {\rm Core}_G(T) \unlhd G$. However, $M/N $ is a chief factor of $G$  and $M \cap {\rm Core}_G(T) \leq M$. This implies that either $M \cap {\rm Core}_G(T)=N$ or $M \cap {\rm Core}_G(T)=M$. In the former case, $M/N=M/(M \cap {\rm Core}_G(T))\cong M {\rm Core}_G(T)/{\rm Core}_G(T) \leq G/{\rm Core}_G(T)$, which is a contradiction as $G/{\rm Core}_G(T)$ is solvable. So, $M=M \cap {\rm Core}_G(T) \leq T$. Recall that as $T \neq G$, $\lambda\neq 1_N$.
	If $\chi\in {\rm Irr}(\lambda^T)$, then since $\lambda\neq 1_N$, $N \leq M \unlhd T$ and $M$ is perfect, we have $\chi(1) \neq 1$. Hence,  $\chi(1) \geq 2\lambda(1)$ or $\lambda(1) \geq 2$. If $\chi(1) \geq 3$ or $[G:T] \geq 3$, then $\chi^G(1)\geq 6$. Thus, it is enough to consider the case when $[G:T]=\chi(1)=2$ and $\alpha=\chi^G \in {\rm Irr}^*_{\mathbb{Q}}(\lambda^G)$. As $\lambda \neq 1_N$, $M'=M \not \leq {\rm ker}\chi$. We obtain that $\psi(1) \geq 2$ for every $\psi \in {\rm Irr}(\chi_M)$. This implies that $\chi_M \in {\rm Irr}(M)$.  Set $K=\rm {ker} \chi \cap M $. Then, $K \unlhd T$. Regarding $\chi \in {\rm Irr}(\lambda^T)$, we observe that $N \not \leq K$. Hence, $NK/K$ is a non-trivial normal subgroup of $M/K$. It follows from Corollary \ref{cpc} that
	$M/K \cong {\rm SL}_2(5)$  and $|NK/K|=2$. Since $M/N$ is a chief factor of $G$, $M/N=S_1/N \times \cdots \times S_t/N$ such that $S_1/N, \cdots , S_t/N$ are isomorphic non-abelian simple groups.
	As $KN/N \unlhd M/N$ and $M/K \cong {\rm SL}_2(5)$, we can assume that $  K N =S_2 \cdots S_t$ and $S_1/N \cap KN/N=1$. Hence, $S_1 \unlhd T$ and $(S_1 \cap K)N=S_1 \cap KN=N$. This shows that $S_1 \cap K \leq N$ is solvable, which forces $S_1/(S_1 \cap K) \cong S_1K/K$ to be non-solvable.   Taking into account that $S_1K/K \unlhd M/K$ and $M/K \cong {\rm SL}_2(5)$, we get that $S_1K/K=M/K$. Therefore, $M=S_1K$. We observe that  $S_1K/NK=S_1(NK)/NK\cong S_1 /(S_1 \cap NK)=S_1 /(N (S_1 \cap K))=S_1/N $ is non-abelian simple. This implies that $S_1K/NK \cong {S_1/N} \cong  {\rm Alt}_5$. Also, $S_1/(S_1 \cap K) \cong S_1K/K=M/K \cong {\rm SL}_2(5)$ and   $S_1/(N \cap K)\cong \frac{S_1/(S_1\cap K)}{(N\cap K)/(S_1\cap K)}$.  Thus,  $S_1 \cap K=N \cap K \unlhd T$ because $S_1 \cap K \leq N \cap K \neq N$ and $S_1/N \cong {\rm Alt}_5$. On the other hand,  $S_1 \unlhd T$.  Our assumption on $M/N$ guarantees that $t \geq 2$ considering $S_1/N \cong {\rm Alt}_5$. As $G$ acts transitively on $\{S_1 ,\ldots, S_t\}$, $S_1 \unlhd T$ and $[G:T]=2$, we get that $t=[G:T]=2$.
	So, $KN=S_2$. This signifies that  $S_2/(S_1 \cap K)=KN/(N \cap K)=K/(N \cap K) \times N/(N \cap K)$.  Then, $K/(S_1 \cap K) \cong S_2/N \cong S_1/N \cong {\rm Alt}_5$. Thus,
$$M/(S_1 \cap K)= S_1S_2/(S_1 \cap K)= S_1/(S_1 \cap K) \times K/(S_1 \cap K) \cong {\rm SL}_2(5) \times {\rm Alt}_5.$$ Further, $\chi_M \in {\rm Irr}(M/(S_1\cap K))$ and  $\alpha_T = \chi + \chi^g$ for some
	$g \in G - T $. As $\chi(1)=2$, we have  $\chi_M= \psi_1 1_{{\rm Alt}_5 }$, where $\psi _1  \in {\rm Irr}(S_1/(S_1 \cap K))$ is of degree $2$. Note that $S_1^g=S_2$, $S_2^g=S_1$ and $S_1 \cap K=N \cap K \unlhd G$. So, reasoning by analogy as above shows that  $K^g ={\rm ker}\chi^g \cap M$, $S_2\cap K^g\unlhd T$, $\chi^g \in {\rm Irr}(T/(S_2\cap K^g))$ and $M/(S_2 \cap K^g) \cong {\rm Alt}_5 \times {\rm SL}_2(5)$. Let $x \in M$ such that $xN \in S_1/N $ be of order $5$  and $x^5(S_1 \cap K) \in N/(S_1 \cap K)$ be of order $2$. Then, we can assume that  $\alpha(x)=\chi(x(S_1 \cap K))+\chi^g(x(S_2 \cap K^g))=((-1+\sqrt{5})/2)+(\pm 1)$. Consequently, $\alpha$ is not rational. %, considering the irreducible characters of ${\rm SL}_2(5) \times {\rm Alt}_5$  and ${\rm Alt}_5 \times {\rm SL}_2(5)$ of degree $2$.
	This is  a contradiction. This forces $\alpha(1) \geq 6$ for every $\alpha \in {\rm Irr}_{\mathbb{Q}}^*(\lambda^G)$. Thus, ${\rm acd}_{\mathbb{Q}}^*(\lambda^G) \geq 6$, as wanted.
	%%%%%%%%%%%%%%%%%%%%%%%%%%%%%%%%%%%%%%%%%%%%%%%%%%%%%%%%55
	\\{\bf Case iii.} Suppose that $[G:T]=1$. Hence, ${\rm ker}\lambda \trianglelefteq G $. We know that ${\rm acd}_{\mathbb{Q}}^*(\lambda^G)={\rm acd}_{\mathbb{Q}}^*(\lambda^{G/{\rm ker}\lambda})$ and $G/{\rm ker}\lambda$ satisfies the assumptions. So, if ${\rm ker}\lambda \neq 1$, then we get from induction that ${\rm acd}_{\mathbb{Q}}^*(\lambda^G) \geq 9/2$.
	Now, we consider the case when ${\rm ker}\lambda =1$.   We continue the proof in the following sub-cases:
	%%%%%%%%%%%%%%%%%%%%%%%%55
	\\{\bf Sub-case a.} Let $\lambda$ be extendible to $\lambda_0 \in {\rm Irr}(M)$. By \cite[Corollary 6.17]{isaacs}, \begin{eqnarray} \label{eq1}
	{\rm Irr}(\lambda^M)=\{\lambda_0 \psi: \psi \in {\rm Irr}(M/N)\}.
	\end{eqnarray}
	By Lemma \ref{145}, $I_G(\lambda_0)=G$.
	Let $\psi_0=1_{M}, \psi_1,\ldots, \psi_t$ be representatives of the action of $G/N$ on ${\rm Irr}(M/N)$.  Then, $\lambda_0\psi_0, \lambda_0\psi_1,\ldots, \lambda_0\psi_t$ are representatives of the action of $G$ on ${\rm Irr}(\lambda^M)$. Hence, \begin{eqnarray}\label{eq1}
	{\rm Irr}^*_{\mathbb{Q}}(\lambda^G)=\dot{\cup}_{i=0}^t{\rm Irr}^*_{\mathbb{Q}} ((\lambda_0\psi_i)^G).
	\end{eqnarray}
	By Lemma \ref{Lemma 2.6}, we can assume that $\psi_t(1) \geq 8$ and $\psi_t$ is extendible to $\chi \in {\rm Irr}_{\mathbb{Q}}(G)$.
	Set $C={\rm ker}\lambda_0$. Then, $C \unlhd G$ and $C\cap N={\rm ker}\lambda=1$ considering $I_G(\lambda_0)=I_G(\lambda)=G$.
	We have $CN/N \unlhd G/N$ and  $CN/N \leq M/N$. So, either $CN=N$ or $CN=M$. In the former case, $C=C \cap N=1$. If $\lambda_0(1)=1$, then $M=M'\leq {\rm ker}\lambda_0=1$, a contradiction.  Therefore, $\lambda(1)=\lambda_0(1)>1$. It follows that $\lambda_0\psi_i(1) \geq 6$ for every $1 \leq i \leq t$ because $\psi_i(1) \geq 3$ for every $1 \leq i \leq t$.
	Hence,
	\begin{eqnarray} \label{eq732}
	\alpha(1) \geq 6 ~{\rm for~every~}  \alpha \in \cup_{i=1}^t{\rm Irr}((\lambda_0\psi_i)^G).
	\end{eqnarray}
	In the latter case,  since $C\cap N={\rm ker} \lambda=1$, we get that $M=N \times C$. However, $N$ is solvable and $M$ is perfect. Thus, $N=1$. Since $M=M'$ and $\lambda_0(1)=1$, we see that  $\lambda_0=\psi_0$ and $M$ is a minimal normal subgroup of $G$. % and ${\rm Irr}^*_{\mathbb{Q}}(\lambda^G)={\rm Irr}^*_{\mathbb{Q}}(G)$.
	In view of   Corollary \ref{corollary 2.7}, Lemmas \ref{simple} and \ref{l7}, \begin{eqnarray} \label{eq739}
	\alpha(1) \geq 5 ~{\rm for~every~}  \alpha \in \dot\cup_{i=1}^t{\rm Irr}(\psi_i^G).
	\end{eqnarray} Note that $\psi_0\psi_i=\psi_i$ for every $0 \leq i \leq t$. So, in both possibilities, considering  \eqref{eq732} and \eqref{eq739} lead to  \begin{eqnarray} \label{eq400}
	\alpha(1) \geq 5 ~{\rm for~every~}  \alpha \in \dot\cup_{i=1}^t{\rm Irr}((\lambda_0\psi_i)^G).
	\end{eqnarray}
	If ${\rm Irr}_{\mathbb{Q}}^*((\lambda_0\psi_0)^G) =\emptyset$,  then  Lemma \ref{later}, \eqref{eq1} and \eqref{eq400} guarantee that   ${\rm acd}^*_{\mathbb{Q}}(\lambda^G) \geq 5$. Otherwise, Lemma \ref{Lemma 2.2}(ii) forces ${\rm acd}_{\mathbb{Q}}^*((\lambda_0\psi_0)^G+(\lambda_0\psi_t)^G) \geq (\psi_t(1)+1)/2 \geq 9/2$.
	It follows from  Lemma \ref{later}, \eqref{eq1} and \eqref{eq400} that  ${\rm acd}^*_{\mathbb{Q}}(\lambda^G) \geq 9/2$, as desired.
	%%%%%%%%%%%%%%%%%%%%%%%%%%%%%%%%%555555
	\\{\bf Sub-case b.} Suppose that $\lambda$ is not extendible to $M$. Let $\chi \in {\rm Irr}^*_{\mathbb{Q}}(\lambda^G)$ be of  degree less than $5$ and $\psi \in {\rm Irr}(\chi_M)$.    First, assume $\lambda(1) >1$. Since $\lambda$ is not extendible to $M$, $\lambda(1) < \psi(1)$. However, $\lambda(1)>1$, $\lambda(1)$ divides $ \psi(1)$,  $\psi(1) $ divides $ \chi(1)$ and $\chi(1)<5$. This forces $\lambda(1)=2$ and $\psi(1)=\chi(1)=4$, which means that $\chi_M=\psi$. Set $K={\rm ker}\psi$.  Then, $K \trianglelefteq G$. Therefore, $KN/N \trianglelefteq G/N$. As $M/N$ is a chief factor of $G$ and $KN/N \leq M/N$, we deduce that either $KN=M$ or $KN=N$. In the former case, $K/(K\cap N)\cong M/N$. So, $K$ is non-solvable. However, $K\trianglelefteq G$ and $K \lneq M$. This is a contradiction with our assumption on $M$. In the latter case, $K =K \cap N={\rm ker}\lambda$ because $I_G(\lambda)=G$. So, $K=1$.   Since $M/N$ is a non-abelian chief factor of $G$,  $|M/N|$ is divisible  by a prime $r \geq 5$. It follows from the main result of \cite{feit1} that a maximal subgroup of  a Sylow $r$-subgroup of  $M$ is normal in $M$. This forces $|M/N|_r=r$. This means that $M/N$ is non-abelian simple. Recall that  $(M,N, \lambda)$ is a character triple and $\psi(1)=2\lambda(1)$.  Thus, \cite[Theorem 8.1]{zal} shows that $M/N \cong {\rm Alt}_5$, a contradiction.  %then  Let $\theta\in {\rm Irr}(\psi_{S_1})$. If $\theta(1)=\lambda(1)=2$, then $\theta_N=\lambda$. As mentioned in Subcase a, we can see that  $I_M(\theta)=M$. So, ${\rm ker }\theta=K \cap S_1=1$. Hence, every minimal normal subgroup of $S_1$ is abelian. It follows from \ref{cp} that $S_1=S_0C$ is a central product of $S_0$ and $C$ such that $S_0 \cong {\rm SL}_2(5)$ and $C \cap S_0=Z(S_0)$ is a cyclic group of order $2$. However, $N$ is the  solvable radical of $S_1$ and $S_1/N$ is non-abelian simple. Thus, $C=N$. As $\theta \in {\rm Irr}(S_1|Z(S_0))$ and $\lambda\in {\rm Irr}(C|Z(S_0))$ and $\lambda(1)>1$, we conclude  from Lemma \ref{cp0} that $\theta(1) \geq 4$, a contradiction.  %T%Therefore, ${\rm ker}\theta \cong S_1/N$. Assume that $\theta' \in {\rm Irr}(\psi_{S_1})$.  Taking into account that  $S_1 \unlhd M$ and $\psi \in {\rm Irr}(M)$, we see that $\theta'(1)=\lambda(1)$ and ${\rm ker}\theta' \cong S_1/N$. However, $N$ is solvable and $S_1/N$ is non-abelian simple. This signifies that ${\rm ker}\theta'={\rm ker}\theta$ for every $\theta' \in {\rm Irr}(\psi_{S_1})$. It follows immediately that ${\rm ker}\theta \leq {\rm ker}\psi=1$, a contradiction.
	%  This guarantees that $\theta(1) > \lambda(1)$. Note that  $\lambda(1) > 1$ and $\theta(1) \mid \psi(1)=4$. We conclude that $\theta(1)=\psi(1)$.  So, $\psi_{S_1}=\theta$. Hence, if $t \geq 2$, then ${\rm Irr}(\theta^M)=\{\psi\alpha: \alpha\in {\rm Irr}(M/S_1) \}$.  I_G(\theta)=G$. This forces $$ Therefore, $\theta^g(x)$ ${\rm  Irr}(\theta^M)=\{\}$ As $M/S_1$ is a direct product of some non-abelian simple groups, we get that  ${\rm ker} \psi \cong M/S_1$. However, ${\rm ker} \psi=1$. Consequently,  $t=1$ and $M/N=S_1/N$ is a non-abelian simple group. Recall that  $(M,N, \lambda)$ is a character triple and $\psi(1)=2\lambda(1)$.  Thus, \cite[Theorem 8.1]{zal} shows that $M/N \cong {\rm A}_5$, a contradiction.
	Next, suppose that $\lambda(1)=1$. Regarding ${\rm ker}\lambda=1$, we have $N \leq {\bf Z}(G) \cap M$. This implies that $N={\bf Z}(M)$. Since $\chi \in {\rm Irr}^*_{\mathbb{Q}}(\lambda^G)$ and $\chi(1) \leq 4$,   Lemma \ref{prod} forces $M/N=M/{\bf Z}(M)$ to be simple. However,  $M/N \not\cong {\rm Alt}_5 $ and ${\rm Alt}_6$. Therefore,  $\psi(1) \geq 4$, applying Corollary \ref{cpc}. Hence, $\psi(1)=\chi(1)=4$ and $\chi_M=\psi$. As $\chi \in {\rm Irr}^*_{\mathbb{Q}}(\lambda^G)$, we have  $\psi \in {\rm Irr}^*_{\mathbb{Q}}(\lambda^M)$. Thus,  Lemma \ref{quasi} guarantees that $M/N=M/{\bf Z}(M) \cong {\rm Alt}_5$ or ${\rm Alt}_6$, a contradiction. Therefore, $\chi(1) \geq 5$ for every $\chi \in {\rm Irr}^*_{\mathbb{Q}}(\lambda^G)$. Consequently,  ${\rm acd}^*_{\mathbb{Q}}(\lambda^G) \geq 5 > 9/2$, as desired.
\end{proof}
%=====================
%%%%%%%%%%%%%%%%%%%%%%%%%%%%%%%%%%%%%555555
\begin{theorem}  If ${\rm acd}^*_{\mathbb{Q}}(G)<9/2$, then $G$ is solvable.
\end{theorem}
\begin{proof} Working towards a contradiction, let $G$ be non-solvable and $M \unlhd G$ be minimal such that $M $ is non-solvable. Suppose that $M/N$ is a chief factor of $G$.  If $M/N \cong {\rm Alt}_5$ or ${\rm Alt}_6$, then the theorem follows from Theorems  \ref{Q1} and \ref{Q2}. If $M/N \not\cong {\rm Alt}_5,{\rm Alt}_6$, then  let $\lambda_0, \ldots,\lambda_t$ be representatives of the action of $G$ on ${\rm Irr}(N)$. Then, ${\rm Irr}(G)=\dot{\cup}_{i=0}^t{\rm Irr}(\lambda_i^G)$. By Proposition \ref{Q3}, ${\rm Irr}^*_{\mathbb{Q}}(\lambda_i^G)=\emptyset$ or ${\rm acd}^*_{\mathbb{Q}}(\lambda_i^G)  \geq  9/2$.  In light of Lemma  \ref{Lemma 2.6},  ${\rm Irr}_{\mathbb{Q}}^*(G) \neq \emptyset$. So, ${\rm Irr}^*_{\mathbb{Q}}(\lambda_i^G) \neq \emptyset$ for some $1 \leq i\leq t$. Thus,  Lemma \ref{later} completes the proof.
\end{proof}
\section{Proofs of Theorems B and D}
%#####################################################
In this section, we give the proofs of Theorems~B and D. Recall that in the  following $\Bbb F$  belongs to  $\{\Bbb C,\Bbb R, \Bbb Q\}$.
\begin{theorem}\label{key1}Let $G$ be a finite group and $N$ be a normal subgroup of $G$.
	\begin{enumerate}
		\item If $0<{\rm acd}_{\Bbb{C},even}(G|N)<{\rm acd}_{\Bbb{C},even}({\rm SL}_{2}(5)|{\rm SL}_{2}(5))=\frac{18}{5}$, then $N$ is solvable.
		\item If $0<{\rm acd}_{\Bbb{R},even}(G|N)<{\rm acd}_{\Bbb{R},even}({\rm SL}_{2}(5)|{\rm SL}_{2}(5))=\frac{18}{5}$, then $N$ is solvable.
		\item If $0<{\rm acd}_{\Bbb{Q},even}(G|N)<{\rm acd}_{\Bbb{Q},even}(\rm{Alt}_5|\rm{Alt}_5)=4$, then $N$ is solvable.
	\end{enumerate}
\end{theorem}

\begin{proof} Suppose $g(\Bbb{R})=g(\Bbb C)=18/5$ and $g(\Bbb{Q})=4$. Working towards the contradiction,  let $G$ be a counterexample of  minimal order. Hence, $N$ is a non-solvable normal subgroup of $G$
	such that  $\acdek(G|N)< g(\Bbb{F})$, which means  $$\frac{\sum\limits_{2\mid d}dn^{\Bbb F}_d(G|N)}{\sum\limits_{2\mid d}n^{\Bbb F}_d(G|N)}<g(\Bbb{F}).$$ This  leads us to
	$$\sum\limits_{2\mid d\geq 4}(5d-18)n^{\Bbb F}_d(G|N)< 8n^{\Bbb F}_2(G|N),$$ where $\Bbb{F}\in \{\Bbb{R},\Bbb{C}\}$ and
	$$\sum\limits_{2\mid d\geq 4}(d-4)n^{\Bbb F}_d(G|N)< 2n^{\Bbb F}_2(G|N),$$ where $\Bbb{F}=\Bbb{Q}$.
	
	Similar to the proof of Theorem \ref{key2}, we claim that every minimal normal subgroup of $G$ contained in $N$ is abelian. On the contrary, assume  $G$ has  a non-abelian minimal normal subgroup $L$
	contained in $N$.
	By Lemma \ref{us1}, $\Irr(L)$ possesses  a rational-valued character  $\phi$ such that $\phi(1)\geq 5$  and
	it is extendable to $\psi \in \Irr_{\Bbb Q}(G)$.  We deduce from Gallagher's theorem that
	$$n^{\Bbb F}_2(G/L|N/L)\leq n^{\Bbb F}_2(G/L)\leq \sum\limits_{ 2\mid d\geq 10} n^{\Bbb F}_d(G|L)\leq \sum\limits_{ 2\mid d\geq 10} n^{\Bbb F}_d(G|N).$$
	
	On the other hand, as $\Irr(G|L)$ does not contain any character of degree $2$, we obtain that  $n^{\Bbb F}_2(G|N)=n^{\Bbb F}_2(G/L|N/L)$, which means

 $$8n^{\Bbb F}_2(G|N)\leq \sum\limits_{ 2\mid d\geq 10} 8n^{\Bbb F}_d(G|N)\leq \sum\limits_{ 2\mid d\geq 10} (5d-18)n^{\Bbb F}_d(G|N),$$

 when $\Bbb F\in \{\Bbb R, \Bbb C\}$ and
 $$2n^{\Bbb F}_2(G|N)\leq \sum\limits_{ 2\mid d\geq 10} 2n^{\Bbb F}_d(G|N)\leq \sum\limits_{ 2\mid d\geq 10} (d-4)n^{\Bbb F}_d(G|N),$$

 when $\Bbb F=\Bbb Q$, is a  contradiction.
	%%%%%%%%%%%%%%%%%%%%%%%%%%%%%%%%%%%%5
	%	$8n^{\Bbb F}_2(G|N)\leq \sum\limits_{ 2\mid d\geq 10} 8n^{\Bbb F}_d(G|N)\leq \sum\limits_{ 2\mid d\geq 10} (5d-18)n^{\Bbb F}_d(G|N)$, a contradiction, in case $\Bbb{F}\in \{\Bbb C, \Bbb R\}$. Moreover, similarly 	 $2n^{\Bbb F}_2(G|N)\leq \sum\limits_{ 2\mid d\geq 10} 2n^{\Bbb F}_d(G|N)\leq \sum\limits_{ 2\mid d\geq 10} (d-4)n^{\Bbb F}_d(G|N)$, a contradiction, in case $\Bbb{F}=\Bbb Q$.
	%%%%%%%%%%%%%%%%%%%%%%%%%%%%%%%%%%%%%%%%%%%
	%So by Lemma \ref{lemma1}, $n_1(G|N) < n_1(G)\leq n_d(G|M)|G:I|\leq n_d(G|N)|G:I|$ and $n_2(G|N)\leq n_2(G)\leq n_{2d}(G|M)|G:I|+ n_d(G|M)|G:I|/2\leq n_{2d}(G|N)|G:I|+ n_d(G|N)|G:I|/2$, where $d=\phi(1)|G:I|$. Thus,
	%	$$3n_1(G|N) + n_2(G|N) <
	%	7n_d(G|N)|G:I|/2 + n_{2d}(G|N)|G:I|.$$
	%	On the other hand, as $\phi(1)\geq 8$, we have $ |G:I|\leq d/8$. Hence,  $7|G:I|/2\leq 7d/16<2d-5$ and $|G:I|\leq d/8 < 4d-5$. Therefore,
	%	$$3n_1(G|N) + n_2(G|N) <
	%	\sum_{
	%		2\mid k\geq 4}
	%	(2k-5)n_k(G|N),$$
	%	implying a contradiction.
	%
	%	Now, we assume $M\cong A_5$.
	%	Then, $M$ has an irreducible character $\phi$ of degree 4 that
	%	is extendable to $\Aut(M)$, and so $\phi$ is  extendable to $G$, by \cite[Lemma 5]{ema}. As $n_2(G)=n_2(G/M)$ and $n_1(G)=n_1(G/M)$, using Gallagher's theorem,   it follows that  $n_1(G) \leq n_4(G|M)\leq n_4(G|N)$ and $n_2(G) \leq n_8(G|M)\leq n_8(G|N)$. Thus,
	%	$$3n_1(G|N)+n_2(G|N)< 3n_1(G) + n_2(G) \leq 3n_4(G|M) + n_8(G|M)$$$$ \leq \sum_{2|k\geq4}
	%	(2k - 5)n_k(G|M)\leq \sum_{2|k\geq4}
	%	(2k - 5)n_k(G|N),$$
	%	a contradiction.
	%%%%%%%%%%%%%%%%%%%%%%%%%%%%%%%%%%%%
	Hence,  every minimal normal subgroup of $G$ contained in $N$ is abelian.
	Let $M$ be non-solvable and normal subgroup of $G$ contained in $N$ of  minimal order such
	that $M\leq N$ is non-solvable. Obviously, $M$ is perfect and contained in the last term of the derived
	series of $N$. We may  choose a minimal normal subgroup $T$ of $G$ such that $T \leq M$ and when
	$	[M,R]\not =1$ we choose $T \leq [M,R],$ where $R$ denotes the radical	solvable of $M$.
	We know  $T$ is abelian. It follows that
	$N/T$ is non-solvable since $N$ is non-solvable. By the minimality of $G$, we  have
	$	\acdek(G/T|N/T) \geq  g(\Bbb F)$. This implies that
	$\acdek(G|N) <  g(\Bbb F) \leq  \acdek(G/T|N/T)$.  So,
	%	\begin{center}

$$ 8n^{\Bbb F}_2(G/T|N/T)\leq  \sum\limits_{2\mid d\geq 4}(5d-18)n^{\Bbb F}_d(G/T|N/T)\leq   \sum\limits_{2\mid d\geq 4}(5d-18)n^{\Bbb F}_d(G|N)< 8n^{\Bbb F}_2(G|N),$$
when $\Bbb F\in \{\Bbb R, \Bbb C\}$ and

	$$ 2n^{\Bbb F}_2(G/T|N/T)\leq  \sum\limits_{2\mid d\geq 4}(d-4)n^{\Bbb F}_d(G/T|N/T)\leq  \sum\limits_{2\mid d\geq 4}(d-4)n^{\Bbb F}_d(G|N)< 2n^{\Bbb F}_2(G|N),$$
	
	when $\Bbb F= \Bbb Q$.
	%\end{eqnarray}
	%\end{center}
	%%%%%%%%%%%%%%%%%%%%%%%%%%%%%%%%%%%%%%%%%%%%%%%%%%55
	Therefore, $n^{\Bbb F}_2(G/T|N/T)< n^{\Bbb F}_2(G|N)$ and  there exists $\chi \in \Irr_{\Bbb F}(G|N)$ such that $\chi(1) = 2$ and
	$T\nleq \ker(\chi)$. Hence, by Lemma \ref{cp}, there exists  a normal subgroup $C$ of $M$ such that
	$G = MC$ is a central product
	with $T =M\cap C= {\bf Z}(M)$ of order $2$  and
	$	M \cong  \SL_2(5)$.  Note that $\chi_M\in \Irr_{\Bbb F}(M)$. As $\SL_2(5)$ does not contain any rational-valued irreducible  character of degree $2$,
	we get that ${\Bbb F}\in \{\mathbb{R}, \mathbb{C}\}$.
	%%%%%%%%%%%%%%%%%%%%%%%%
	From Lemma \ref{cp0} we have  that
	$$n^{\mathbb{C}}_d(G|T)=\sum\limits_{t\mid d}n^{\mathbb{C}}_t(C|T)n^{\mathbb{C}}_{d/t}(M|T), \ \ \ \ (*)$$
	and
	$2=\chi(1)=\alpha(1)\beta(1)$ for some $\beta\in \Irr(C|T)$  and $\alpha\in \Irr(M|T)$.
	Since $M\cong \SL_2 (5)$,  $\alpha(1)\in\{2,4,6\}$, and hence $\alpha(1)=2$ and   $\beta(1)=1$. Therefore,   $\beta \in \Irr(C|T)$ is an extension of
	the unique non-principal linear character  $\lambda\in \Irr(T)$. Thus,  using Gallagher's theorem \cite[Theorem 6.17]{isaacs}, we  have   $n^{{\Bbb C}}_d(C|T) = n^{{\Bbb C}}_d(C/T)$ for each positive integer $d$.
	%%%%%%%%%%%%%%%%%
	
	First, let  $\Bbb F=\Bbb R$. It follows from Lemma \ref{A5} that either $\Irr_{\mathbb{R}}(G|T)=\Irr(G|T)$
	or $\acd_{\Bbb R}(G|T)\geq 4$. If the second case occurs, then  we obtain, by Lemma \ref{later},
	$\acd_{\Bbb R}(G|N)\geq 18/5$, a contradiction. So, we  must have $\Irr_{\Bbb R}(G|T)=\Irr(G|T)$.	
	%Using Lemma \ref{cp0}, we have $$n^{\mathbb{C}}_d(G|T)=\sum\limits_{k\mid d}n^{\mathbb{C}}_k(C|T)n^{\mathbb{C}}_{d/k}(M|T). \ \ \ \ (*)$$

 	First, we show that   $\chi_C$ is not irreducible. Indeed, otherwise $\chi_C\in \Irr(C|\lambda)$ and as $\lambda$ extends to $\beta \in \Irr(C)$, Gallagher's theorem yields that $\chi_C=\eta\beta$ for some $\eta\in \Irr(C/T)$.  Noticing  $G/T\cong C/T\times M/T$, we have  $\eta$ extends to $\eta_0\in \Irr(G)$. Now,  applying  \cite[Theorem 6.16]{isaacs},  $\Irr((\eta\beta)^G)=\{\eta_0\theta\ |\ \theta\in \Irr(\beta^G)\}$. As $\eta\beta$ extends to $G$, we deduce that  $\beta$ is  extendable to $G$, which is impossible.  	
	Thus,  $\chi_C=2\beta$  for some linear character $\beta\in \Irr(C|\lambda)$.
		As  $\chi\in \Irr_{\Bbb R}(G|T)$ is an extension of $\chi_M\in \Irr_{\Bbb R}(M|T)$,  Gallagher's theorem yields that $\gamma\chi\in \Irr(G|T)$ for all $\gamma \in \Irr(G/M)=\Irr(C/T)$.	
	  Hence, $(\gamma\chi)_C=2\gamma\beta$ for all  character $\gamma\in \Irr(C/T)$.   From $\Irr_{\Bbb R}(G|T)=\Irr(G|T)$, we infer that  $\beta$ and $\gamma\beta$ are real-valued. As linear characters do not vanish on any conjugacy classes, we deduce that $\gamma$ is real-valued for every character $\gamma\in \Irr(C/T)$. Hence $n^{\Bbb C}_d(C|T)=n^{\Bbb C}_d(C/T)=n^{\Bbb R}_d(C/T)$.

	Defining $\Irr_{\Bbb F, odd}(M/T)=\Irr_{\Bbb F}(M/T)-\Irr_{\Bbb F, even}(M/T)$,  we have   $\Irr_{\Bbb R,even}(G/T|N/T)=\{\alpha\beta \ |\  \alpha\in \Irr_{\Bbb R, even}(C/T), 1\not =\beta\in \Irr_{\Bbb R,odd}(M/T)\}\cup  \{\alpha\beta \ |\  \alpha\in \Irr_{\Bbb R}(C/T), \beta\in \Irr_{\Bbb R,even}(M/T)\}\cup \{\alpha1_{M/T}\ |\ \alpha\in \Irr_{\Bbb R,even}(C/T|(C/T\cap N/T))\}$. Note that there are 4 characters in $\Irr(M|T)$, two of degree $2$, one of degree $4$ and one of degree $6$.
	% \begin{eqnarray}\nonumber
	%\acd_{\Bbb R, even}(G|T)=\acd_{\Bbb C, even}(G|T)=\frac{14\sum\limits_{d\geq 1}dn^{\Bbb C}_d(C/T)}{4\sum\limits_{d\geq 1}n^{\Bbb C}_d(C/T)}.
	%\end{eqnarray}
	Also,  looking at the degrees of the characters in $\Irr(M/T)=\Irr(\rm {Alt}_5)$, and Lemme \ref{cp0}, we have
	% \begin{eqnarray}\nonumber
	% \acd_{\Bbb R, even}(G/T|N/T)=\frac{4\sum\limits_{d\geq 1}dn_d^{\Bbb R}(C/T)+11\sum\limits_{2\mid d}dn^{\Bbb R}_d(C/T)+\sum\limits_{2\mid d}dn_d^{\Bbb R}(C/T|(N/T\cap C/T))+ 5\sum\limits_{2\mid d}dm_d(C/T)}{\sum\limits_{d\geq 1}n^{\Bbb R}_d(C/T) +3\sum\limits_{2\mid d}n^{\Bbb R}_d(C/T)+\sum\limits_{2\mid d}n^{\Bbb R}_d(C/T|(N/T\cap C/T))+\sum\limits_{2\mid d}m_d(C/T)}.
	% \end{eqnarray}
	$$	\acd_{\Bbb R,even}(G|N)=\frac{\sum\limits_{\theta\in \Irr_{\Bbb R, even}(G/T|N/T)}\theta(1) \ \ + \sum\limits_{\theta\in\Irr_{\Bbb R, even}(G|T)}\theta(1)}{|\Irr_{\Bbb R, even}(G/T|N/T)|+|\Irr_{\Bbb R, even}(G|T)| }=$$
		
	$$	\frac{4\sum\limits_{d\geq 1}dn_d^{\Bbb R}(C/T)+11\sum\limits_{2\mid d}dn^{\Bbb R}_d(C/T)+\sum\limits_{2\mid d}dn_d^{\Bbb R}(C/T|(N/T\cap C/T))+ 14\sum\limits_{d\geq 1}dn^{\Bbb R}_d(C/T)}{\sum\limits_{d\geq 1}n^{\Bbb R}_d(C/T) +3\sum\limits_{2\mid d}n^{\Bbb R}_d(C/T)+\sum\limits_{2\mid d}n^{\Bbb R}_d(C/T|(N/T\cap C/T))+ 4\sum\limits_{d\geq 1}n^{\Bbb R}(C/T)}$$
	$$
		=\frac{18\sum\limits_{d\geq 1}dn_d^{\Bbb R}(C/T)+11\sum\limits_{2\mid d}dn^{\Bbb R}_d(C/T)+\sum\limits_{2\mid d}dn_d^{\Bbb R}(C/T|(N/T\cap C/T))}{5\sum\limits_{d\geq 1}n_d^{\Bbb R}(C/T)+3\sum\limits_{2\mid d}n^{\Bbb R}_d(C/T)+\sum\limits_{2\mid
								 d}n^{\Bbb R}_d(C/T|(N/T\cap C/T))}$$

		$$\geq \frac{18\sum\limits_{d\geq 1}n_d^{\Bbb R}(C/T)+22\sum\limits_{2\mid d}n^{\Bbb R}_d(C/T)+2\sum\limits_{2\mid d}n_d^{\Bbb R}(C/T|(N/T\cap C/T))}{5\sum\limits_{d\geq 1}n_d^{\Bbb R}(C/T)+3\sum\limits_{2\mid d}n^{\Bbb R}_d(C/T)+\sum\limits_{2\mid d}n^{\Bbb R}_d(C/T|(N/T\cap C/T))}.$$

As    $\sum\limits_{2\mid d}n^{\Bbb R}_d(C/T)\geq \sum\limits_{2\mid d}n^{\Bbb R}_d(C/T|(N/T\cap C/T))$,  we infer

$$5(18\sum\limits_{d\geq 1}n_d^{\Bbb R}(C/T)+22\sum\limits_{2\mid d}n^{\Bbb R}_d(C/T)+2\sum\limits_{2\mid d}n_d^{\Bbb R}(C/T|(N/T\cap C/T)))$$$$\geq 18(5\sum\limits_{d\geq 1}n_d^{\Bbb R}(C/T)+3\sum\limits_{2\mid d}n^{\Bbb R}_d(C/T)+\sum\limits_{2\mid d}n^{\Bbb R}_d(C/T|(N/T\cap C/T))),$$

  equivalently $\acd_{\Bbb R, even}(G|N)\geq 18/5$, a contradiction. Now, let $\Bbb F=\Bbb C$, then by a similar argument we have
	
		$$\acd_{{\Bbb C}, even}(G)=\frac{4\sum\limits_{d\geq 1}dn_d^{\Bbb C}(C/T)+ 11\sum\limits_{2\mid d}dn_d^{\Bbb C}(C/T) +\sum\limits_{2\mid d}dn_d^{\Bbb C}(C/T|(N/T\cap C/T))+14\sum\limits_{d\geq 1}dn^{\Bbb C}_d(C/T)}{\sum\limits_{d\geq 1}n_d^{\Bbb C}(C/T)+3\sum\limits_{2\mid d}n_d^{\Bbb C}(C/T) ++\sum\limits_{2\mid d}n_d^{\Bbb C}(C/T|(N/T\cap C/T))+4\sum\limits_{d\geq 1}n^{\Bbb C}_d(C/T)}$$$$
		=\frac{18\sum\limits_{d \geq 1}dn_d^{\Bbb C}(C/T)+11\sum\limits_{2\mid d}dn_d^{\Bbb C}(C/T)+\sum\limits_{2\mid d}dn_d^{\Bbb C}(C/T|(N/T\cap C/T))}{5\sum\limits_{d> 1}n_d^{\Bbb C}(C/T)+3\sum\limits_{2\mid d}n_d^{\Bbb C}(C/T)+\sum\limits_{2\mid d}n_d^{\Bbb C}(C/T|(N/T\cap C/T))}\geq 18/5,$$
	
	which is the final contradiction.
\end{proof}
%############################################
%%%%%%%%%%%%%%%%%%%%%%%%%%%%%%%55
\begin{theorem}\label{us2}
	Let $N, M \unlhd G$, $N\subseteq M$, $M/N$ be a non-abelian chief
	factor of $G$ and $\lambda \in {\rm Irr}(N)$ such that  either  $M / N$
	is non-simple, $\lambda(1)[G:I_G(\lambda)]>1$,  or $\lambda$ is extendible to $M$. If ${\rm Irr}_{\mathbb{F},even}(\lambda^{G}) \neq \emptyset$, then
	${\rm acd}_{\mathbb{F},even}(\lambda^{G})
	\geq 4$.
\end{theorem}
\begin{proof} Set  $T= I_G(\lambda)$. If $[G:T] \geq 3$, then  Clifford correspondence forces ${\rm acd}_{\mathbb{F},even}(\lambda^G) \geq 4$, as wanted.  Next, assume that  $[G:T] \leq 2$. Then, $N \leq T \unlhd G$ and $N \leq M \cap T \unlhd G$. However, $M/N $ is a chief factor of $G$  and $M \cap T \leq M$. This implies that either $M \cap T=N$ or $M \cap T=M$. In the former case, $M/N=M/(M \cap T)\cong M T/T \leq G/T$, which is a contradiction as $G/T$ is of order at most  $2$. So, $M=M \cap T \leq T$.
	By Lemmas \ref{Lemma 2.6} and \ref{us1},  there exists a
	character  $\psi \in {\rm Irr}(M/N)$ which extends to an element in ${\rm Irr}_{\mathbb{F}}(G/N)$ and  either $\psi(1) \geq 8$ or $M/N \cong {\rm Alt}_5$ and $\psi(1)=5$.  Let $K={\rm ker}\lambda$.
	%%%%%%%%%%%%%%%%%%%%%%%%%%%%%%%%%%%%%%%%%%%%%%%%55
	Assume  $\lambda$ is extendible to $\lambda_0 \in {\rm Irr}(M)$. By \cite[Corollary 6.17]{isaacs}, $
	{\rm Irr}(\lambda^M)=\{\lambda_0 \psi: \psi \in {\rm Irr}(M/N)\}.
	$ Let $\psi_0=1_{M}, \psi_1,\ldots, \psi_t$ be representatives of the action of $T/N$ on ${\rm Irr}(M/N)$  with  $\psi_t=\psi$.  Then, $\lambda_0\psi_0, \lambda_0\psi_1,\ldots, \lambda_0\psi_t$ are representatives of the action of $T$ on ${\rm Irr}(\lambda^M)$. Hence, \begin{eqnarray}\label{eq723}
	{\rm Irr}_{\mathbb{F},even}(\lambda^T)=\dot{\cup}_{i=0}^t{\rm Irr}_{\mathbb{F},even} ((\lambda_0\psi_i)^T).
	\end{eqnarray}
	For every $1\leq i \leq t $, we have $\psi_i(1)\geq 3$. This yields that
	\begin{eqnarray} \label{eq113}
	\alpha(1) \geq 4 ~{\rm for~every~}  \alpha \in \dot\cup_{i=1}^t{\rm Irr}_{\Bbb{F},even}((\lambda_0\psi_i)^G).
	\end{eqnarray}
	If ${\rm Irr}_{\mathbb{F},even}((\lambda_0\psi_0)^G) =\emptyset$,  then  Lemma \ref{later}, \eqref{eq723} and \eqref{eq113} guarantee that   ${\rm acd}_{\mathbb{F},even}(\lambda^G) \geq 4$. Now, suppose that ${\rm Irr}_{\mathbb{F},even}((\lambda_0\psi_0)^G) \neq \emptyset$. Recall that if $T \neq G$, then  $[G:T]=2$. Hence, $T\vartriangleleft G$ and if $T \neq G$, then  ${\rm Irr}_{\mathbb{F},even}(\lambda^G)={\rm Irr}_{\mathbb{F}}(\lambda^G)  $. Also, $(\lambda_0\psi_0)\psi=\lambda_0\psi \in {\rm Irr}(M)$. By Lemma \ref{145}, $I_G(\lambda_0)=T$. So, $T \leq I_G(\lambda_0\psi)$. If $g \in I_G(\lambda_0\psi)$, then $(\lambda_0\psi)^g=\lambda_0\psi$. Hence, $\lambda^g=\lambda$. This signifies that  $g \in T$. Consequently, $I_G(\lambda_0\psi)=I_G(\lambda_0\psi_0)$. So, Lemmas  \ref{Lemma 2.2}(i) and \ref{2.2} force  ${\rm acd}_{\mathbb{F},even}((\lambda_0\psi)^{G}+(\lambda_0\psi_0)^{G})
	> 4$.  Thus, Lemma \ref{later},  \eqref{eq723}  and \eqref{eq113} show that ${\rm acd}_{\mathbb{F},even}(\lambda^G)\geq 4$, as desired.

	Next, suppose that
	$\lambda$ does not extend to $M$.  Let $\chi  \in
	{\rm Irr}(\lambda^T)$ be of minimal degree such that $\chi ^G \in{\rm Irr}_{\mathbb{F},even}(\lambda^G) $
	and $\varphi \in {\rm Irr}(\chi_{M})$. Since
	$\lambda$ does not extend to $M$, $\chi(1) \geq \varphi(1)\geq 2 \lambda(1)$. Let $T \neq G$ or $\lambda(1)\neq 1$. Then $\chi^G(1) \geq 4$. It follows  that  $\alpha(1) \geq 4$ for every $\alpha \in {\rm Irr}(\lambda^G)$. Consequently, ${\rm acd}_{\mathbb{F},even}(\lambda^G)\geq  4$, as wanted. Now, let $T=G$ and $\lambda(1)=1$. Set $K={\rm ker}\lambda$.  Hence, $K \unlhd G$. Also, the assumptions force  $M/N$ to be non-simple.    If $K \neq 1$, then since $G/K$ satisfies the assumptions of the theorem and ${\rm acd}_{\mathbb{F},even}(\lambda^G) ={\rm acd}_{\mathbb{F},even}(\lambda^{G/K})$, the proof follows from induction. It remains to consider the case when $K=1$. If $\chi(1) \geq 4$, then ${\rm acd}_{\mathbb{F},even}(\lambda^G) \geq 4$, as wanted.
	Finally, assume that $\chi(1) \leq 3$. Since $\chi(1)$ is even, $\chi(1)=2$.
	If
	$\varphi(1)< \chi (1) =2$, then $\varphi(1)=1$, a contradiction because $\lambda$ is not extendible to $M$.
	Thus $\varphi(1)=\chi(1) =2$. So, $\chi_M=\varphi$ and ${\rm ker}\varphi \trianglelefteq G$. Consequently, $N{\rm ker}\varphi/N \trianglelefteq G/N$.
	As $M/N$ is a chief factor of $G$, we have either $ {\rm ker}\varphi \leq N$ or $N{\rm ker}\varphi=M$. Recall that $I_G(\lambda)=G$. This forces ${\rm ker}\varphi\cap N={\rm ker}\lambda=1$. Therefore, ${\rm ker}\varphi \cong N{\rm ker}\varphi/N$ and in the latter case, we have ${\rm ker}\varphi \times N =M$.  Since $\lambda$ is not extendible to $M$, the latter case cannot occur.  In the former  case,   ${\rm ker}\varphi=1$. Since $M/N$ is non-solvable, there exists a prime divisor $r$ of $|M/N|$ such that $r \geq 5$. Hence, \cite{feit1} implies that a maximal subgroup of a Sylow $r$-subgroup of $M$ is normal in $M$, which is impossible as $M/N$ is not simple.
	
	Now, all possibilities have been considered and we are done.
\end{proof}
%%%%%%
%%%%%%%%%%%%%%%%%%%%%%%%%%%%%%%%%%%%%%%%%%%%%%5
In the following theorem, let $f(\mathbb{Q})=4$ and $f(\mathbb{R})=f(\mathbb{C})=7/2$.
\begin{theorem} \label{theoremb}    Let $N \unlhd G$ and $\lambda \in {\rm Irr}(N)$. If
	${\rm Irr}_{\mathbb{F},even}(\lambda^{G})\neq \emptyset$ and $G / N$ is
	non-solvable, then
	${\rm acd}_{\mathbb{F},even}(\lambda^{G}) \geqslant f(\mathbb{F})$.
\end{theorem}
%===============================
\begin{proof}  We complete the proof by induction on $|G|+[G:N]$.
	Let $E$ be a normal maximal subgroup of $G$ such that $N \leqslant
	E$ and $G / E$ is not solvable.
	%%%%%%%%%
	Then, $G/E$ has the unique minimal normal subgroup
	$M/E$ and $M/E$ is not solvable.  Let $\mathfrak{A}=\{\mu_{1}, \ldots,
	\mu_d\}$ be a subset of ${\rm Irr}(\lambda^{E})$ such that every element of ${\rm Irr}(\lambda^{E})$ is $G$-conjugate to exactly one element of $\mathfrak{A}$. Then, ${\rm Irr}(\lambda^G)=\dot{\cup}_{i=1}^d{\rm Irr}(\mu_i^G)$. Since ${\rm Irr}_{\mathbb{F},even}(\lambda^G)\neq \emptyset$, there exists   $i \in \{1,\ldots,d\}$ such that ${\rm Irr}_{\mathbb{F},even}(\mu_{i}^G) \neq \emptyset$. If $N < E$, then it follows from induction that  ${\rm acd}_{\mathbb{F},even}(\mu_{i}^G) \geq f(\mathbb{F})$. Hence, Lemma \ref{later} forces ${\rm acd}_{\mathbb{F},even}(\lambda^G)\geq  f(\mathbb{F})$, as wanted. Next, assume that $N=E$. If either $\lambda(1)[G:I_G(\lambda)] >1$, $\lambda$ is extendible to $M$, or $M/N$ is non-simple, then Theorem \ref{us2} shows that ${\rm acd}_{\mathbb{F},even}(\lambda^G)\geq 4$. So, the proof follows. Now, suppose that $\lambda(1)[G:I_G(\lambda)] =1$, $\lambda$ is not extendible to $M$ and $M/N$ is simple. So, ${\rm ker}\lambda \trianglelefteq G$. If ${\rm ker}\lambda\neq 1$, then as $G/{\rm ker}\lambda$ satisfies the assumptions and ${\rm acd}_{\mathbb{F},even}(\lambda^{G})={\rm acd}_{\mathbb{F},even}(\lambda^{G/{\rm ker}\lambda})$,  the proof follows from induction. Now, let ${\rm ker}\lambda=1$. Thus, $N \leq {\bf Z}(G)$. If $\chi(1) \geq 4$ for every $\chi \in {\rm Irr}_{\mathbb{F},even}(\lambda^{G})$, then ${\rm acd}_{\mathbb{F},even}(\lambda^{G}) \geq 4$. Otherwise, there exists $\chi \in {\rm Irr}_{\mathbb{F},even}(\lambda^{G})$
	such that $\chi(1)=2$.  Assume that $M_1\unlhd G$ is minimal  such that $M_1 \leq M$ and $M_1$ is non-solvable. So, $M_1N=M$ and $M_1'=M_1$. As $\lambda$ is not extendible to $M$ and ${\rm ker}\lambda=1$,
we see that $M_1 \cap N \neq 1$ and $\lambda_{M_1 \cap N}$ is not extendible to $M_1$.  It follows from Corollary \ref{cpc} that $G/N \cong {\rm Alt}_5$ and there exists a normal subgroup $C$ of $G$ such that $G=M_1C$ is a central  product of $M_1$ and $C$, ${\rm SL}_2(5) \cong M_1 \leq M$ and $M_1 \cap C={\bf Z}(M_1) \leq N$.
	So, $\chi_{M_1} \in {\rm Irr}(M_1)$. However, $\chi \in {\rm Irr}_{\mathbb{F},even}(\lambda^{G})$. Therefore, $\chi_{M_1} \in  {\rm Irr}_{\mathbb{F},even}(M_1)$. Observing the character table of ${\rm SL}_2(5)$, we get that  $\mathbb{F}\neq \mathbb{Q}$. This shows that $\alpha(1) \geq 4$ for every $\alpha \in {\rm Irr}_{\mathbb{Q},even}(\lambda^G)$, and consequently ${\rm acd}_{\mathbb{Q},even}(\lambda^G)\geq 4=f(\mathbb{Q})$. Next, assume that $\mathbb{F}\in \{ \mathbb{R},\mathbb{C} \}$. As $(G,N,\lambda)$ is a character triple and $G/N \cong {\rm Alt}_5$, $\lambda^G=\Sigma_{i=1}^4 e_i\psi_i$ for some positive integers $e_i$s and $\psi_i \in {\rm Irr}(\lambda^G)$ such that $\psi_1(1)=\psi_2(1)=2$, $\psi_3(1)=4$ and $\psi_4(1)=6$. Since ${\rm Irr}(\lambda^G)$
	contains exactly one character of degree $4$ and one character of degree $ 6$, we deduce that $\psi_3$ and $\psi_4$ are $\mathbb{F}$-valued. Also, $\chi \in \{\psi_1,\psi_2\}$ is $\mathbb{F}$-valued. Thus, $\psi_1 $ and $\psi_2$ are $\mathbb{F}$-valued and ${\rm acd}_{\mathbb{F},even}(\lambda^{G})\geq (2+2+4+6)/4$. Therefore, ${\rm acd}_{\mathbb{F},even}(\lambda^{G}) \geq 7/2$, as desired. Now, the proof is complete.
\end{proof}
%%%%%%%%%%%%%%%%%%%%%%%%%%%%%%%%%%%%%%%%%%%%%%%%%5

The following corollary, which follows immediately from Theorems \ref{key1} and \ref{theoremb}, is Theorem D.
\begin{corollary} Let $1 \neq N \unlhd G$. If $0<{\rm acd}_{\mathbb{F},even}(G|N) < f(\Bbb{F})$, then
	$G$ is solvable.
\end{corollary}
%%%%%%%%%%%%%%%%%%%%%%%%%%%%%%%%%%%%55
\begin{Remark} Let $1 \neq N \unlhd G$ and $\mathbb{F} \in \{\mathbb{R},\mathbb{C}\}$. Suppose that
	${\rm Irr}_{\mathbb{F},even}(G|N)\neq \emptyset$ and $G / N$ is
	non-solvable. Then, considering the argument given in the proof of Theorem  \ref{theoremb} leads us to see that either ${\rm acd}_{\mathbb{F},even}(G|N) \geq 18/5$ or $G/E \cong {\rm Alt}_5$ for some solvable subgroup $E$ with  $N \leq E \trianglelefteq G$ and ${\rm acd}_{\mathbb{F},even}(G|N) \geq 7/2$. In particular,  ${\rm acd}_{\mathbb{C},even}(G|N) = 7/2$ if and only if there exist the normal subgroups $M_1$ and $C$ of $G$ such that  $G=M_1C$ is a central product of $M_1$ and $C$,  $M_1\cong {\rm SL}_2(5)$, $C$ is abelian and $C\cap M_1={\bf Z}(M_1)=N$.
\end{Remark}
%\begin{Notation} Let $G$ be a finite group and $\mathbb{F}$ be one the fields $\mathbb{Q},\mathbb{R}$ and $\mathbb{C}$. We say that $(G,\mathbb{F})$ satisfies $(*)$ when $\mathbb{F} \neq \mathbb{Q}$ or no simple group ${\rm PSL}_2(3^{2f+1})$, with $f \geq 1$, is involved in $G$
%\end{Notation}
%%%%%%%%%%%%%%%%%%%%%%%%%%%%%%%55
%\begin{lemma} {\rm \cite[Theorem B]{pham}}\label{new}  Let $G$ be  non-solvable. Suppose that no simple group ${\rm PSL}_2(3^{2f+1})$, with $f \geq 1$, is involved in $G$. Then, $G$ has a  rational irreducible character of even degree.
%\end{lemma}
%%%%%%%%%%%%%%%%%%%%%%%%%%%%%%%%%%%%
%\begin{lemma} \label{newnew} {\rm \cite[Corollary 1.4]{hung}} If $G$ is a non-solvable group, then $G$ admits a real irreducible character of even degree.
%\end{lemma}
%%\begin{theorem} Let $(G,\mathbb{F})$ satisfy $(*)$. If ${\rm acd}_{\mathbb{F},even}(G) < f(\mathbb{F})$, then $G$ is solvable.
%%\end{theorem}
%%\begin{proof} Since $(G,\mathbb{F})$ satisfies $(*)$, Lemmas \ref{new} and \ref{newnew} show that ${\rm Irr}_{\mathbb{F},even}(G) \neq \emptyset$. By Taking  $N=1$ and $\lambda=1$ in Theorem \ref{theoremb}, the theorem follows.
%%\end{proof}
%%%%%%%%%%%%%%%%%%%%%%%%
%%%%%%%%%%%%%%%%%%%%%%%%%%%%%%%55
%\begin{theorem} Let $N$ be a normal subgroup of $G$  such that ${\rm Irr}_{\mathbb{F},even}(G|N) \neq \emptyset$. If ${\rm acd}_{\mathbb{F},even}(G|N) < 7/2$, then $G/N$ is solvable.
%\end{theorem}
%\begin{proof} The theorem  follows from Theorem \ref{theoremb}.
%\end{proof}
\maketitle

\bibliographystyle{amsplain}

\end{document}